%% file: integerRelationAlgebra.tex
\documentclass[a4paper]{amsart}

\usepackage{enumerate, amsmath, amsfonts, amssymb, amsthm, wasysym, graphics, graphicx, xcolor, url, hyperref, hypcap, a4wide, stmaryrd, shuffle, xargs, multicol, overpic, pdflscape, multirow, hvfloat, minibox, accents, xstring}
\hypersetup{colorlinks=true, citecolor=darkblue, linkcolor=darkblue, urlcolor=darkblue}
\usepackage[all]{xy}
\usepackage[bottom]{footmisc}
\usepackage{tikz}\usetikzlibrary{trees,snakes,shapes,arrows,matrix,calc}
\graphicspath{{figures/}}
\makeatletter
\def\input@path{{figures/}}
\makeatother

\title{The Hopf algebra of integer binary relations}

\thanks{VPi~was partially supported by the French ANR grants SC3A~(15\,CE40\,0004\,01) and CAPPS~(17\,CE40\,0018).}

\author{Vincent Pilaud}
\author{Viviane Pons}

\address[VPi]{CNRS \& LIX, \'Ecole Polytechnique, Palaiseau}
\email{vincent.pilaud@lix.polytechnique.fr}
\urladdr{http://www.lix.polytechnique.fr/~pilaud/}

\address[VPo]{LRI, Univ. Paris-Sud - CNRS - Centrale Supelec - Univ. Paris-Saclay}
\email{viviane.pons@lri.fr}
\urladdr{https://www.lri.fr/~pons/}

\newtheorem{theorem}{Theorem}

\newtheorem{proposition}[theorem]{Proposition}
\newtheorem{lemma}[theorem]{Lemma}
\newtheorem{definition}[theorem]{Definition}

\theoremstyle{definition}
\newtheorem{example}[theorem]{Example}
\newtheorem{remark}[theorem]{Remark}

\makeatletter
\@addtoreset{proofpart}{theorem}
\@addtoreset{proofcase}{theorem}
\makeatother

\newcommand{\N}{\mathbb{N}} 
\newcommand{\K}{\mathbf{k}} 

\newcommand{\set}[2]{\left\{ #1 \;\middle|\; #2 \right\}} 
\newcommand{\bigset}[2]{\big\{ #1 \;|\; #2 \big\}} 
\newcommand{\ssm}{\smallsetminus} 
\newcommand{\eqdef}{\mbox{\,\raisebox{0.2ex}{\scriptsize\ensuremath{\mathrm:}}\ensuremath{=}\,}} 
\renewcommand{\implies}{\Rightarrow} 

\newcommandx{\rel}[1][1=R]{\mathbin{\mathrm{#1}}} 
\newcommandx{\notrel}[1][1=R]{\mathbin{\!\raisebox{.02cm}{$\not$}\hspace{.02cm}\mathrm{#1}\hspace*{.01cm}}} 
\newcommandx{\drel}[1][1=R]{\mathbin{\mathbf{#1}}} 
\newcommand{\less}{\vartriangleleft} 
\newcommand{\more}{\vartriangleright} 
\newcommand{\bless}{\blacktriangleleft} 
\newcommand{\bmore}{\blacktriangleright} 

\newcommand{\IRel}{\mathcal{R}} 
\newcommand{\IPos}{\mathcal{P}} 

\newcommandx{\Inc}[1]{#1^{\mathsf{Inc}}} 
\newcommandx{\Dec}[1]{#1^{\mathsf{Dec}}}
\newcommand{\maxle}[1]{#1^{\mathsf{maxle}}} 
\newcommand{\minle}[1]{#1^{\mathsf{minle}}} 
\newcommand{\IWOIPid}[1]{#1^{\IWOIP\mathsf{id}}} 
\newcommand{\DWOIPdd}[1]{#1^{\DWOIP\mathsf{dd}}} 
\newcommand{\WOEPid}[1]{#1^{\WOEP\mathsf{id}}} 
\newcommand{\WOEPdd}[1]{#1^{\WOEP\mathsf{dd}}} 
\newcommand{\WOIPd}[1]{#1^{\WOIP\mathsf{d}}} 
\newcommand{\TOIPd}[1]{#1^{\TOIP\mathsf{d}}} 
\newcommandx{\PIPd}[2][2=\orientation]{#1^{{\PIP{#2}}\mathsf{d}}} 

\newcommand{\meetR}{\wedge_\IRel} 
\newcommand{\joinR}{\vee_\IRel} 
\newcommand{\joinWO}{\vee_\WO} 
\newcommand{\meetWO}{\wedge_\WO} 
\newcommand{\joinTO}{\vee_\TO} 
\newcommand{\meetTO}{\wedge_\TO} 
\newcommand{\joinWOIP}{\vee_\WOIP} 
\newcommand{\meetWOIP}{\wedge_\WOIP} 
\newcommand{\joinTOIP}{\vee_\TOIP} 
\newcommand{\meetTOIP}{\wedge_\TOIP} 

\DeclareMathOperator{\inv}{inv} 
\newcommand{\wole}{\preccurlyeq} 
\newcommand{\bt}{\mathrm{bt}} 
\newcommand{\st}{\mathrm{st}} 

\newcommand{\fS}{\mathfrak{S}} 
\newcommand{\fB}{\mathfrak{B}} 
\newcommand{\WO}{{\fS}} 
\newcommand{\TO}{\fB} 
\newcommandx{\CO}[1][1=\signature]{#1-\mathsf{CO}} 

\newcommand{\WOEP}{\mathsf{WOEP}} 
\newcommand{\WOFP}{\mathsf{WOFP}} 
\newcommand{\WOIP}{\mathsf{WOIP}} 
\newcommand{\IWOIP}{\mathsf{IWOIP}} 
\newcommand{\DWOIP}{\mathsf{DWOIP}} 
\newcommand{\TOEP}{\mathsf{TOEP}} 
\newcommand{\TOFP}{\mathsf{TOFP}} 
\newcommand{\TOIP}{\mathsf{TOIP}} 
\newcommand{\ITOIP}{\mathsf{ITOIP}} 
\newcommand{\DTOIP}{\mathsf{DTOIP}} 
\newcommandx{\COEP}[1][1=\signature]{\mathsf{COEP}(#1)} 
\newcommandx{\COFP}[1][1=\signature]{\mathsf{COFP}(#1)} 
\newcommandx{\COIP}[1][1=\signature]{\mathsf{COIP}(#1)} 
\newcommandx{\PIP}[1][1=\orientation]{\mathsf{PIP}_{#1}} 
\newcommandx{\PEP}[1][1=\orientation]{\mathsf{PEP}_{#1}} 
\newcommandx{\PFP}[1][1=\orientation]{\mathsf{PFP}_{#1}} 
\newcommand{\orientation}{\mathbb{O}} 

\DeclareMathOperator{\conv}{conv} 
\newcommandx{\Asso}[1][1=\signature]{\mathsf{Asso}(#1)} 
\newcommandx{\Perm}[1][1=n]{\mathsf{Perm}(#1)} 
\newcommandx{\Para}[1][1=n]{\mathsf{Para}(#1)} 

\newcommandx{\graphG}[1][1=G]{\mathrm{#1}} 
\newcommandx{\tree}[1][1=T]{\mathrm{#1}} 
\newcommandx{\tuple}[1][1=T]{\mathcal{#1}} 
\newcommandx{\poset}{{\tree[P]}} 
\newcommand{\signature}{\varepsilon} 

\newcommand{\Hvect}{\mathcal{H}} 

\newcommand{\product}{\cdot} 
\newcommand{\coproduct}{\triangle} 
\newcommand{\bcoproduct}{\blacktriangle} 
\newcommand{\rbcoproduct}{\overline{\bcoproduct}} 
\newcommand{\shiftedShuffle}{\,\bar\shuffle\,} 
\newcommand{\convolution}{\star} 
\newcommand{\underprod}[2]{{#1}\backslash{#2}} 
\newcommand{\overprod}[2]{{#1}\slash{#2}} 
\newcommand{\uRS}{\underprod{\rel[R]}{\rel[S]}} 
\newcommand{\subalg}{{\mathsf{sub}}} 
\newcommand{\subalgi}{{\mathsf{subi}}} 
\newcommand{\subalgd}{{\mathsf{subd}}} 
\newcommand{\quotient}{{\mathsf{quo}}} 

\newcommand{\F}{\mathbb{F}} 
\newcommand{\E}{\mathbb{E}} 
\newcommand{\HH}{\mathbb{H}} 
\newcommand{\FRel}{\F\IRel} 
\newcommand{\ERel}{\E\IRel} 
\newcommand{\HRel}{\HH\IRel} 
\newcommand{\FPos}{\F\IPos} 
\newcommand{\EPos}{\E\IPos} 
\newcommand{\HPos}{\HH\IPos} 
\newcommand{\FWOEP}{\F\mathsf{WE}} 
\newcommand{\FWOEPquo}{\F\mathsf{WE}^\quotient} 
\newcommand{\FWOEPsubi}{\F\mathsf{WE}^\subalgi} 
\newcommand{\FWOEPsubd}{\F\mathsf{WE}^\subalgd} 
\newcommand{\FTOEP}{\F\mathsf{TE}} 
\newcommand{\FWOIP}{\F\mathsf{WI}} 
\newcommand{\FWOIPquo}{\F\mathsf{WI}^\quotient} 
\newcommand{\FWOIPsub}{\F\mathsf{WI}^\subalg} 
\newcommand{\FTOIP}{\F\mathsf{TI}} 
\newcommand{\FWOFP}{\F\mathsf{WF}} 
\newcommand{\FWOFPquo}{\F\mathsf{WF}^\quotient} 
\newcommand{\FTOFP}{\F\mathsf{TF}} 

\newcommand{\fref}[1]{Figure~\ref{#1}} 
\newcommand{\ie}{\textit{i.e.}~} 
\newcommand{\eg}{\textit{e.g.}~} 
\definecolor{darkblue}{rgb}{0,0,0.7} 
\definecolor{green}{RGB}{57,181,74} 
\newcommand{\darkblue}{\color{darkblue}} 
\newcommand{\red}{\color{red}} 
\newcommand{\blue}{\color{blue}} 
\newcommand{\defn}[1]{\emph{\darkblue #1}} 
\usepackage{todonotes}

\def \interscale{0.5}

\makeatletter
\def\l@section{\@tocline{1}{4pt}{0pc}{}{}}
\makeatother
\let\oldtocpart=\tocpart
\renewcommand{\tocpart}[2]{\bf\large\oldtocpart{#1}{#2}}
\let\oldtocsection=\tocsection
\renewcommand{\tocsection}[2]{\bf\oldtocsection{#1}{#2}}

\begin{document}

\begin{abstract}
We construct a Hopf algebra on integer binary relations that contains under the same roof several well-known Hopf algebras related to the permutahedra and the associahedra: the Malvenuto--Reutenauer algebra on permutations, the Loday--Ronco algebra on planar binary trees, and the Chapoton algebras on ordered partitions and on Schr\"oder trees. We also derive from our construction new Hopf structures on intervals of the weak order on permutations and of the Tamari order on binary trees.
\end{abstract}

\vspace*{-.3cm}

\maketitle

\vspace{-.5cm}

\tableofcontents

\vspace{-.3cm}

An integer binary relation is a binary relation on~$[n] \eqdef \{1, \dots, n\}$ for some~$n \in \N$.
Integer posets are integer binary relations that are moreover posets (\ie reflexive, antisymmetric and transitive).
Many fundamental combinatorial objects (see Table~\ref{table:listObjects} left) can be thought of as specific integer posets.
This observation was used in~\cite{ChatelPilaudPons} to reinterpret classical lattice structures (see Table~\ref{table:listObjects} middle) as specializations (subposets or sublattices) of a lattice structure called the weak order on posets.
This interpretation enables to consider simultaneously all these specific integer posets and motivated the emergence of permutrees~\cite{PilaudPons-permutrees}, which are combinatorial objects interpolating between permutations, binary trees, and Cambrian trees~\cite{ChatelPilaud}.

In this paper, we continue the exploration of the algebraic combinatorics of integer binary relations and integer posets, focussing on Hopf structures.
We construct a Hopf algebra on integer binary relations where
\begin{enumerate}[(i)]
\item the product~$\rel[R] \product \rel[S]$ of two relations~$\rel[R], \rel[S]$ is the sum of all relations that contain~$\rel[R]$ at the beginning and~$\rel[S]$ at the end as induced subrelations,
\item the coproduct~$\coproduct(\rel)$ of a relation~$\rel$ is the sum of the tensor products of the subrelations induced by~$\rel$ over all possible partitions $[n] = A \sqcup B$ that correspond to a total cut of~$\rel$.
\end{enumerate}
We then reinterpret classical Hopf algebras~\cite{MalvenutoReutenauer, LodayRonco, Chapoton} (see Table~\ref{table:listObjects} right) as specializations (quotients or subalgebras) of the integer poset algebra.
Moreover, we obtain Hopf structures on the intervals of the weak order and on the intervals of the Tamari lattice, that remained undiscovered to the best of our knowledge.

\newpage

\begin{table}[t]
    \centerline{
    \begin{tabular}{c|c|c}
    combinatorial object & lattice structure & Hopf algebra \\
    \hline
    permutations & weak order & Malvenuto--Reutenauer algebra~\cite{MalvenutoReutenauer} \\
    binary trees & Tamari lattice & Loday--Ronco algebra~\cite{LodayRonco} \\
    ordered partitions & facial weak order~\cite{KrobLatapyNovelliPhanSchwer, PalaciosRonco, DermenjianHohlwegPilaud} & Chapoton algebra on ordered partitions~\cite{Chapoton} \\
    Schr\"oder trees & facial Tamari order~\cite{PalaciosRonco, DermenjianHohlwegPilaud} & Chapoton algebra on Schr\"oder trees~\cite{Chapoton} \\
    weak order intervals & interval lattice of the weak order & \textsc{new}, see Sections~\ref{subsubsec:intervalsQuotientAlgebra} and \ref{subsubsec:intervalsSubalgebra} \\
    Tamari order intervals & interval lattice of the Tamari lattice~\cite{ChatelPons} & \textsc{new}, see Section~\ref{subsec:subalgebrasTO}
    \end{tabular}
    }
    \caption{Algebraic structures on classical combinatorial objects that can be reinterpreted as integer binary relations. See also \fref{fig:roadMap} for the connections between the Hopf algebras.}
    \label{table:listObjects}
\end{table}

\section{Integer binary relations}
\label{sec:integerBinaryRelations}

Our main object of focus are binary relations on integers.
An \defn{integer (binary) relation} of size~$n$ is a binary relation on~$[n] \eqdef \{1, \dots, n\}$, that is, a subset~$\rel$ of~$[n]^2$.
As usual, we write equivalently~$(u,v) \in \rel$ or~$u \rel v$, and similarly, we write equivalently~$(u,v) \not\in \rel$ or~$u \notrel v$.
Throughout the paper, all relations are implicitly assumed to be reflexive ($x \rel x$ for all~$x \in [n]$), although we often forget to include the diagonal~$\set{(u,u)}{u \in [n]}$ in our descriptions.
We denote by~$\IRel_n$ the set of all (reflexive) binary relations on~$[n]$ and let~$\IRel \eqdef \bigsqcup_{n \ge 0} \IRel_n$.

\subsection{Weak order}

A lattice structure called the \defn{weak order} on integer binary relations has been defined in \cite{ChatelPilaudPons}.
We recall its definition here as we will latter use this order to give a combinatorial description of the product.

Let~$\rel[I]_{n} \eqdef \set{(a,b) \in [n]^{2}}{a \le b}$ and~$\rel[D]_{n} \eqdef \set{(b,a) \in [n]^{2}}{a \le b}$.
Observe that~$\rel[I]_{n} \cup \rel[D]_{n} = [n]^{2}$ while~$\rel[I]_{n} \cap \rel[D]_{n} = \set{(a,a)}{a \in [n]}$.
We say that the relation~$\rel \in \IRel_n$ is \defn{increasing} (resp.~\defn{decreasing}) when~$\rel \subseteq \rel[I]_n$ (resp.~$\rel \subseteq \rel[D]_n$).
The \defn{increasing} and \defn{decreasing subrelations} of an integer relation~$\rel \in \IRel_n$ are the relations defined~by:
\[
\Inc{\rel} \eqdef \rel \cap {}\rel[I]_n{} = \set{(a,b) \in \rel}{a \le b}
\quad\text{and}\quad
\Dec{\rel} \eqdef \rel \cap {}\rel[D]_n{} = \set{(b,a) \in \rel}{a \le b}.
\]
In our pictures, we always represent an integer relation~$\rel \in \IRel_n$ as follows: we write the numbers $1, \dots, n$ from left to right and we draw the increasing relations of~$\rel$ above in blue and the decreasing relations of~$\rel$ below in red.
Although we only consider reflexive relations, we always omit the relations~$(i,i)$ in the pictures (as well as in our explicit examples). See \eg \fref{fig:productInterval}.

\begin{definition}
\label{def:weakOrder}
The \defn{weak order} on~$\IRel_n$ is given by~$\rel \wole \rel[S]$ if~${\Inc{\rel} \supseteq \Inc{\rel[S]}}$ and~${\Dec{\rel} \subseteq \Dec{\rel[S]}}$.
\end{definition}

Note that the weak order is obtained by combining the refinement lattice on increasing subrelations with the coarsening lattice on decreasing subrelations.
It explains the following statement.

\begin{proposition}
\label{prop:reflexiveLattice}
The weak order~$(\IRel_n, \wole)$ is a graded lattice whose meet and join are given by
\[
{\rel} \meetR {\rel[S]} = ( \Inc{\rel} \cup \Inc{\rel[S]} ) \cup ( \Dec{\rel} \cap \Dec{\rel[S]} )
\qquad\text{and}\qquad
{\rel} \joinR {\rel[S]} = ( \Inc{\rel} \cap \Inc{\rel[S]} ) \cup ( \Dec{\rel} \cup \Dec{\rel[S]} ).
\]
\end{proposition}

\subsection{Hopf algebra}

We consider the vector space~$\K\IRel \eqdef \bigoplus_{n \ge 0} \K\IRel_n$ indexed by all integer binary relations of arbitrary size.
We denote by~$(\FRel_{\rel})_{\rel \in \IRel}$ the standard basis of~$\K\IRel$.
In this section, we define a product and coproduct that endow~$\K\IRel$ with a Hopf algebra structure.

We denote by~$\rel_X \eqdef \set{(i,j) \in [k]^2}{x_i \rel x_k}$ the \defn{restriction} of an integer relation~$\rel \in \IRel_n$ to a subset~$X = \{x_1, \dots, x_k\} \subseteq [n]$. Intuitively, it is just the restriction of the relation $\rel$ to the subset~$X$ which is then \emph{standardized} to obtain a proper integer binary relation.

\subsubsection{Product}

The product that we define on binary integer relation generalizes the shifted shuffle of permutations: for~$\rel[R] \in \IRel_m$ and~$\rel[S] \in \IRel_n$, deleting the first $m$ values (resp.~last $n$ values) in any relation of the shifted shuffle~$\rel[R] \shiftedShuffle \rel[S]$ yields the relation~$\rel[S]$ (resp.~$\rel[R]$).

For~$m,n \in \N$, we denote by~$\overline{[n]}^m \eqdef \{m+1, \dots, m+n\}$ the interval~$[n]$ shifted by~$m$.
For~$m \in \N$ and~$\rel[S] \in \IRel_n$, we denote by~$\overline{\rel[S]}^m \eqdef \set{(m+i,m+j)}{(i,j) \in \rel[S]}$ the shifted relation.
We also simply use~$\overline{[n]}$ and~$\overline{\rel[S]}$ when~$m$ is clear from the context.

For~$\rel[R] \in \IRel_m$ and~$\rel[S] \in \IRel_n$, we define~${\uRS \eqdef {\rel[R]} \cup {\overline{\rel[S]}} \cup ([m] \times \overline{[n]})}$ and ${\overprod{\rel[R]}{\rel[S]} \eqdef {\rel[R]} \cup {\overline{\rel[S]}} \cup (\overline{[n]} \times [m])}$.

\begin{definition}
For two relations~$\rel[R] \in \IRel_m$ and $\rel[S] \in \IRel_n$, define the \defn{shifted shuffle}~$\rel[R] \shiftedShuffle \rel[S]$ as the set of relations~$\rel[R] \sqcup \overline{\rel[S]} \sqcup \rel[I] \sqcup \rel[D]$ for all possible~$\rel[I] \subseteq [m] \times \overline{[n]}$ and~$\rel[D] \subseteq \overline{[n]} \times [m]$.
\end{definition}

\begin{remark}
Note that the shifted shuffle of~$\rel[R] \in \IRel_m$ and~$\rel[S] \in \IRel_n$ has cardinality~$|{\rel[R] \shiftedShuffle \rel[S]}| = 2^{2mn}$.
\end{remark}

\begin{example}
For instance,
\input{productRelations}
where the sum ranges over all relations in the interval of \fref{fig:productInterval}.

\begin{figure}[t]
    \includegraphics[scale=.6]{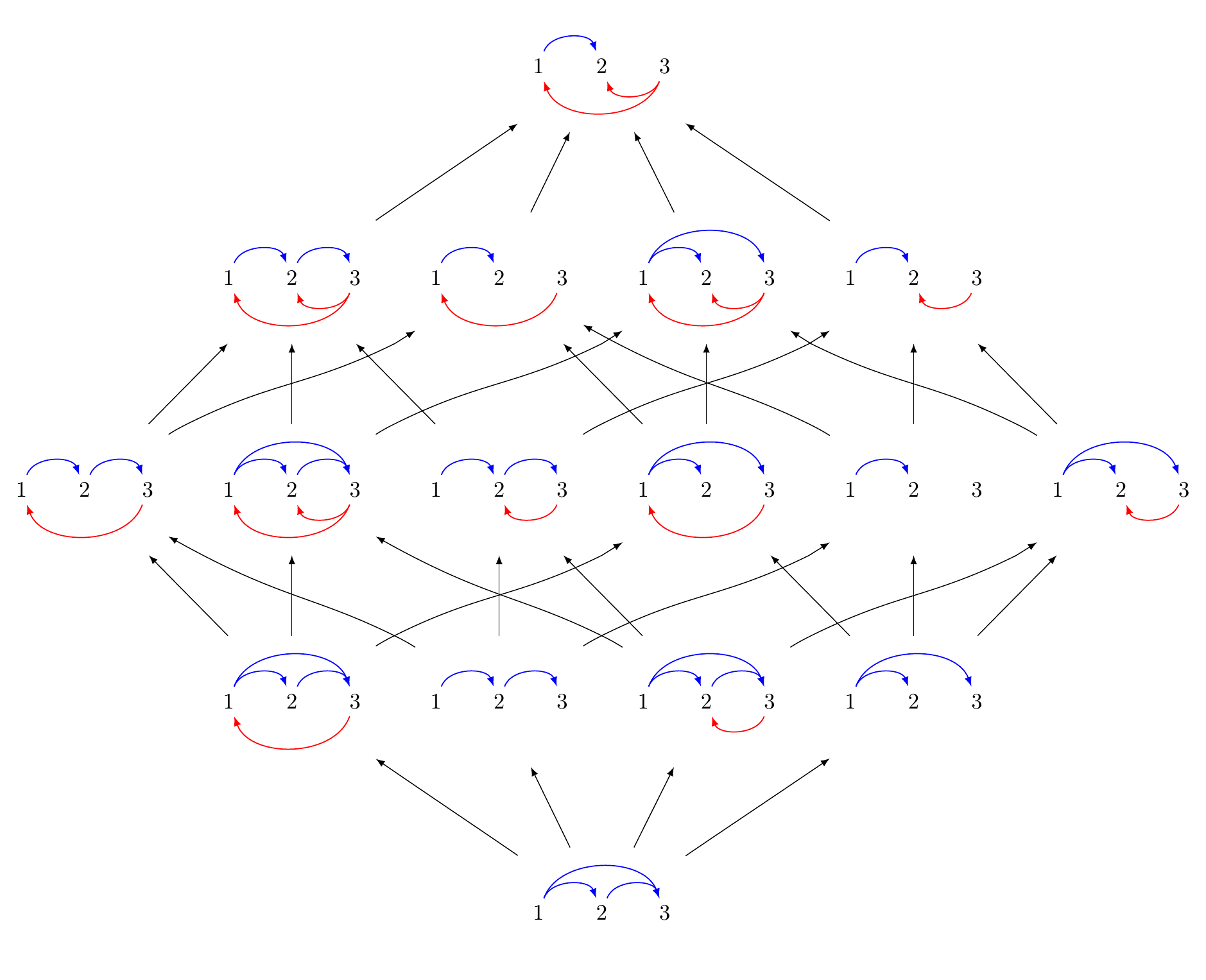}
    \caption{Interval corresponding to a product of relations.}
    \label{fig:productInterval}
\end{figure}
\end{example}

\begin{proposition}
\label{prop:characterizationShuffleProduct}
For~$\rel[R] \in \IRel_m$, $\rel[S] \in \IRel_n$ and~$\rel[T] \in \IRel_{m+n}$, we have
\[
\rel[T] \in \rel[R] \shiftedShuffle \rel[S]
\quad \iff \quad
\rel[T]_{[m]} = \rel[R] \text{ and } \rel[T]_{\overline{[n]}} = \rel[S]
\quad \iff \quad
\uRS \wole \rel[T] \wole \overprod{\rel[R]}{\rel[S]}.
\]
\end{proposition}

\begin{proof}
The first equivalence is immediate.
Assume now that~$\rel[T] = \rel[R] \sqcup \overline{\rel[S]} \sqcup \rel[I] \sqcup \rel[D]$ for some~$\rel[I] \subseteq [m] \times \overline{[n]}$ and~$\rel[D] \subseteq \overline{[n]} \times [m]$. Then~$\Inc{\rel[T]} = \Inc{\rel[R]} \cup \Inc{\overline{\rel[S]}} \cup \rel[I] \subseteq \Inc{\rel[R]} \cup \Inc{\overline{\rel[S]}} \cup ([m] \times \overline{[n]}) = \Inc{(\uRS)}$ and $\Dec{\rel[T]} = \Dec{\rel[R]} \cup \Dec{\overline{\rel[S]}} \cup \rel[D] \supseteq \Dec{\rel[R]} \cup \Dec{\overline{\rel[S]}} = \Dec{(\uRS)}$ so that~$\uRS \wole \rel[T]$. Similarly~$\rel[T] \wole \overprod{\rel[R]}{\rel[S]}$. Conversely, if~$\uRS \wole \rel[T] \wole \overprod{\rel[R]}{\rel[S]}$, then~$\rel[R] = \uRS_{[m]} \wole \rel[T]_{[m]} \wole \overprod{\rel[R]}{\rel[S]}_{[m]} = \rel[R]$, thus~$\rel[T]_{[m]} = \rel[R]$. Similarly~${\rel[T]_{\overline{[n]}} = \rel[S]}$.
\end{proof}

\begin{definition}
\label{def:product}
The \defn{product} of two integer relations~$\rel[R] \in \IRel_m$ and~$\rel[S] \in \IRel_n$ is
\[
\FRel_{\rel[R]} \product \FRel_{\rel[S]} \eqdef \sum_{\rel[T] \in \rel[R] \shiftedShuffle \rel[S]} \FRel_{\rel[T]}.
\]
\end{definition}

\begin{proposition}
\label{prop:algebra}
The product~$\product$ defines an associative graded algebra structure on~$\K\IRel$.
\end{proposition}

\begin{proof}
For~$\rel[R] \in \IRel_m$ and~$\rel[S] \in \IRel_n$, all relations in~$\rel[T] \shiftedShuffle \rel[S]$ belong to~$\IRel_{m+n}$ by definition.
Moreover, for~$\rel[R] \in \IRel_m, \rel[S] \in \IRel_n, \rel[T] \in \IRel_o$, Proposition~\ref{prop:characterizationShuffleProduct} ensures that the relations in~$(\rel[R] \shiftedShuffle \rel[S]) \shiftedShuffle \rel[T]$ and in~$\rel[R] \shiftedShuffle (\rel[S] \shiftedShuffle \rel[T])$ are the relations~$\rel[U] \in \IRel_{m+n+o}$ such that~$\rel[U]_{[m]} = \rel[T]$, $\rel[U]_{\overline{[n]}^m} = \rel[S]$ and~$\rel[U]_{\overline{[o]}^{m+n}} = \rel[T]$.
\end{proof}

By Proposition~\ref{prop:characterizationShuffleProduct}, we know that this product can be interpreted as a sum over an interval. We now prove a property slightly more general.

\begin{proposition}
\label{prop:productInterval}
The product of two intervals is an interval: for $\rel \wole \rel[R']$ in~$\IRel_m$ and $\rel[S] \wole \rel[S']$ in~$\IRel_n$,
\begin{equation*}
\Bigg( \sum_{\rel \wole \rel[U] \wole \rel[R']} \FRel_{\rel[U]} \Bigg) \product \Bigg( \sum_{\rel[S] \wole \rel[V] \wole \rel[S']} \FRel_{\rel[V]} \Bigg)
=
\Bigg( \sum_{\uRS \wole \rel[T] \wole \overprod{\rel[R']}{\rel[S']}} \FRel_{\rel[T]} \Bigg).
\end{equation*}
\end{proposition}

\begin{proof}
We have
\begin{align*}
\Bigg( \sum_{\rel \wole \rel[U] \wole \rel[R']} \FRel_{\rel[U]} \Bigg) \product \Bigg( \sum_{\rel[S] \wole \rel[V] \wole \rel[S']} \FRel_{\rel[V]} \Bigg)
=
\sum_{\substack{\rel \wole \rel[U] \wole \rel[R'] \\ \rel[S] \wole \rel[V] \wole \rel[S']}} \sum_{\rel[T] \in \rel[U] \shiftedShuffle \rel[V]}
\FRel_{\rel[T]}.
\end{align*}
First note that all coefficients are equal to 1. Indeed, any relation $\rel[T]$ of the sum belongs to exactly one set $\rel[U] \shiftedShuffle \rel[V]$ as $\rel[U]$ and $\rel[V]$ are uniquely defined by $\rel[U] = \rel[T]_{[m]}$ and $\rel[V] = \rel[T]_{\overline{[n]}}$.
The only thing to prove is then
\[
\set{T \in \IRel}{\uRS \wole T \wole \overprod{\rel[R']}{\rel[S']}} = \bigsqcup_{\substack{\rel \wole \rel[U] \wole \rel[R'] \\ \rel[S] \wole \rel[V] \wole \rel[S']}} \set{T \in \IRel}{\underprod{\rel[U]}{\rel[V]} \wole T \wole \overprod{\rel[U]}{\rel[V]}}
\]
where the union on the right is disjoint. Let us call $A$ the set on the left and $B$ the set on the right.
It is clear that $B \subseteq A$.
Indeed, for $T \in B$, we have $\underprod{\rel[U]}{\rel[V]} \wole T \wole \overprod{\rel[U]}{\rel[V]}$ for some $\rel[R] \wole \rel[U] \wole \rel[R']$ and $\rel[S] \wole \rel[V] \wole \rel[S']$: this gives directly $\uRS \wole \rel[T] \wole \overprod{\rel[R']}{\rel[S']}$.
Conversely, let $\rel[T] \in \IRel$ be such that $\uRS \wole \rel[T] \wole \overprod{\rel[R']}{\rel[S']}$.
This means $\Inc{(\overprod{\rel[R']}{\rel[S']})} \subseteq \Inc{\rel[T]} \subseteq \Inc{(\uRS)}$ and $\Dec{(\overprod{\rel[R']}{\rel[S']})}  \supseteq \Dec{\rel[T]} \supseteq \Dec{(\uRS)}$.
This is still true if the relations are restricted to $[m]$ (resp.~$\overline{[n]}$).
For $\rel[U] \eqdef \rel[T]_{[m]}$ and $\rel[V] \eqdef \rel[T]_{\overline{[n]}}$, we get that $(\uRS)_{[m]} = \rel[R] \wole \rel[U] \wole \rel[R'] = (\overprod{\rel[R']}{\rel[S']})_{[m]}$ and $(\uRS)_{\overline{[n]}} = \rel[S] \wole \rel[V] \wole \rel[S'] = (\overprod{\rel[R']}{\rel[S']})_{\overline{[n]}}$.
Now $\rel[T] = \rel[U] \cup \rel[V] \cup I$ with $I \subseteq ([m] \times \overline{[n]}) \cup (\overline{[n]} \times [m])$ which means $T \in \rel[U] \shiftedShuffle \rel[V]$.
\end{proof}

\subsubsection{Coproduct}

We now define a coproduct on integer relations using total cuts.

\begin{definition}
\label{def:totalCut}
A \defn{total cut}~$(X,Y)$ of a relation~$\rel[T] \in \IRel_p$ is a partition~$[p] = X \sqcup Y$ such that~$x \rel[T] y$ and $y \notrel[T] x$ for all $x \in X$ and~$y \in Y$.
For two relations~$\rel[R] \in \IRel_m$ and $\rel[S] \in \IRel_n$, define the \defn{convolution}~$\rel[R] \convolution \rel[S]$ as the set of relations~$\rel[T] \in \IRel_{m+n}$ which admit a total cut~$(X,Y)$ such that~$\rel[T]_X = \rel[R]$ and~$\rel[T]_Y = \rel[S]$.
\end{definition}

\begin{remark}
Note that the convolution of~$\rel[R] \in \IRel_m$ and~$\rel[S] \in \IRel_n$ has cardinality~$|{\rel[R] \convolution \rel[S]}| = \binom{m+n}{m}$.
\end{remark}

\begin{definition}
\label{def:coproduct}
The \defn{coproduct} of an integer relation~$\rel[T] \in \IRel$ is
\[
\coproduct(\FRel_{\rel[T]}) \eqdef \sum_{\rel[T] \in \rel[R] \convolution \rel[S]} \FRel_{\rel[R]} \otimes \FRel_{\rel[S]}.
\]
\end{definition}

\begin{example}
For instance,
\input{coproductRelations}
where the terms in the coproduct arise from the total cuts $(\{1,2,3\}, \varnothing)$, $(\{1\}, \{2,3\})$, $(\{1,3\},\{2\})$, and $(\varnothing, \{1,2,3\})$.
\end{example}

\begin{proposition}
\label{prop:coalgebra}
The coproduct~$\coproduct$ defines a coassociative graded coalgebra structure on~$\K\IRel$.
\end{proposition}

\begin{proof}
For~$\rel[R] \in \IRel_m$ and~$\rel[S] \in \IRel_n$, all relations in~$\rel[T] \convolution \rel[S]$ belong to~$\IRel_{m+n}$ by definition.
Moreover, for~$\rel[R] \in \IRel_m, \rel[S] \in \IRel_n, \rel[T] \in \IRel_o$, the relations in~$(\rel[R] \convolution \rel[S]) \convolution \rel[T]$ and in~$\rel[R] \convolution (\rel[S] \convolution \rel[T])$ are precisely the relations~$\rel[U] \in \IRel_{m+n+o}$ such that there is a partition~$[m+n+o] = X \sqcup Y \sqcup Z$ such that
\begin{itemize}
\item $\{(x,y), (x,z), (y,z)\} \subseteq \rel[U]$ and~$\{(y,x), (z,x), (z,y)\} \cap \rel[U] = \varnothing$ for all~$x \in X$, $y \in Y$, $z \in Z$,
\item $\rel[U]_X = \rel[T]$, $\rel[U]_Y = \rel[S]$ and~$\rel[U]_Z = \rel[T]$.
\qedhere
\end{itemize}
\end{proof}

\subsubsection{Hopf algebra}

We now combine the algebra and coalgebra structures on~$\K\IRel$ to a Hopf algebra.
Recall that a \defn{combinatorial Hopf algebra} is a combinatorial vector space endowed with an associative product~$\product$ and a coassociative coproduct~$\coproduct$ which satisfy the compatibility relation
\(
\coproduct(\FRel_{\rel[R]} \product \FRel_{\rel[S]}) = \coproduct (\FRel_{\rel[R]}) \product \coproduct (\FRel_{\rel[S]}),
\)
where the last product is defined componentwise by~${(a \otimes b) \product (c \otimes d) = (a \product c) \otimes (b \product d)}$.

\begin{proposition}
\label{prop:HopfAlgebra}
The product~$\product$ of Definition~\ref{def:product} and the coproduct~$\coproduct$ of Definition~\ref{def:coproduct} endow~$\K\IRel$ with a Hopf algebra structure.
\end{proposition}

\begin{example}
Before giving the formal proof, let us illustrate on an example the compatibility relation~$\coproduct(\FRel_{\rel[R]} \product \FRel_{\rel[S]}) = \coproduct (\FRel_{\rel[R]}) \product \coproduct (\FRel_{\rel[S]})$.
\input{product_coproduct}
\end{example}

\begin{proof}[Proof of Proposition~\ref{prop:HopfAlgebra}]
We have
\[
\coproduct (\FRel_{\rel[R]}) \product \coproduct (\FRel_{\rel[S]}) = \sum \FRel_{\rel[R]_X \sqcup \overline{\rel[S]_U}^{|X|} \sqcup \rel[I] \sqcup \rel[D]} \otimes \FRel_{\rel[R]_Y \sqcup \overline{\rel[S]_V}^{|Y|} \sqcup \rel[I]' \sqcup \rel[D]'} = \coproduct(\FRel_{\rel[R]} \product \FRel_{\rel[S]}),
\]
where the sum ranges over all total cuts~$(X,Y)$ of~$\rel[R]$ and~$(U,V)$ of~$\rel[S]$ and all relations
\[
\rel[I] \subseteq [|X|] \times \overline{[|U|]}^{|X|},
\quad
\rel[D] \subseteq \overline{[|U|]}^{|X|} \times [|X|],
\quad
\rel[I]' \subseteq [|Y|] \times \overline{[|V|]}^{|Y|}
\quad\text{and}\quad
\rel[D]' \subseteq \overline{[|V|]}^{|Y|} \times [|Y|].
\]
The first equality directly follows from the definitions.
For the second equality, observe that for any~$\rel[T] \in \rel[R] \shiftedShuffle \rel[S]$, the total cuts of~$\rel[T]$ are precisely of the form~$(X \sqcup \overline{U}^m, Y \sqcup \overline{V}^m)$ where~$(X,Y)$ is a total cut of~$\rel[R]$ and~$(U,V)$ is a total cut of~$\rel[S]$ such that~$X \times \overline{V}^m$ and $Y \times \overline{U}^m$ are both subsets of~$\rel[T]$ while~$\overline{V}^m \times X$ and~$\overline{U}^m \times Y$ are both subsets of the complement of~$\rel[T]$.
\end{proof}

\subsubsection{Multiplicative bases}

In this section, we describe multiplicative bases of~$\K\IRel$ and study the indecomposable elements of~$\IRel$ for these bases.
For a relation~$\rel \in \IRel$, we define
\[
\ERel^{\rel} = \sum_{\rel[R] \wole \rel[R']} \FRel_{\rel[R']}
\qquad\text{and}\qquad
\HRel^{\rel} = \sum_{\rel[R'] \wole \rel[R]} \FRel_{\rel[R']}.
\]

\begin{example}
\input{multbasesRel}
\end{example}

\begin{proposition}
\label{prop:Rmultbases}
The sets~$(\ERel^{\rel})_{\rel \in \IRel}$ and~$(\HRel^{\rel})_{\rel \in \IRel}$ form multiplicative bases of~$\K\IRel$ with
\[
\ERel^{\rel[R]} \product \ERel^{\rel[S]} = \ERel^{\uRS}
\qquad\text{and}\qquad
\HRel^{\rel[R]} \product \HRel^{\rel[S]} = \HRel^{\overprod{\rel[R]}{\rel[S]}}.
\]
\end{proposition}

\begin{proof}
First note that the elements of $(\ERel^{\rel})_{\rel \in \IRel}$ (resp.~$(\HRel^{\rel})_{\rel \in \IRel}$) are linearly independent: each element $\ERel^{\rel}$ contains a leading term $\FRel_{\rel}$ and so the transition matrix is triangular.
The product formula is a direct consequence of Proposition~\ref{prop:productInterval}.
\end{proof}

\begin{example}
For instance,
\input{productRelationsMultbases}
\end{example}

Note that even though $\ERel$ and $\HRel$ have very simple definitions for the product, the definition of the coproduct is now more complicated than on $\FRel$.
In particular, we now have some coefficients greater than 1 which appear as in the example below.

\begin{example}
For instance,
\input{coproductRelationsMultbases}
\end{example}

\begin{definition}
We say that a relation $\rel[T]$ is \defn{under-indecomposable} (resp.~\defn{over-indecomposable})  if there is no $\rel[R]$ and $\rel[S]$ in $\IRel$ with $|{\rel[R]}| \geq 1$ and $|{\rel[S]}| \geq 1$ such that $\rel[T] = \uRS$ (resp.~$\rel[T] = \overprod{\rel[R]}{\rel[S]}$).
\end{definition}

\begin{proposition}
\label{prop:relfree}
The algebra $\K\IRel$ is freely generated by the elements $\ERel^{\rel[T]}$ where $\rel[T]$ is under-indecomposable (resp.~by the elements $\HRel^{\rel[T]}$ where $\rel[T]$ is over-indecomposable).
\end{proposition}

We will prove this proposition only for~$\ERel^{\rel[T]}$.
Besides, we will work solely with the notion of \emph{under-indecomposable} which we will simply call \emph{indecomposable} in the rest of paper when there is no ambiguity.
The proof of Proposition~\ref{prop:relfree} relies on the results of~\cite{LodayRoncoCoFree} on the (co-)freeness of (co-)associative algebras.

\begin{definition}[{\cite[p.\,7]{LodayRoncoCoFree}}]
A \defn{unital infinitesimal bialgebra}~$(\Hvect, \product , \bcoproduct)$ is a vector space~$\Hvect$ equipped with a unital associative product~$\product$ and a counital coassociative coproduct~$\bcoproduct$ which are related by the unital infinitesimal relation:
\begin{equation}
\label{eq:unital-inf}
\tag{$\star$}
\rbcoproduct( x \product y ) = (x \otimes 1) \product \rbcoproduct(y) + \rbcoproduct(x) \product (1 \otimes y) + x \otimes y,
\end{equation}
where the product~$\product$ on~$\Hvect \otimes \Hvect$ and the reduced coproduct~$\rbcoproduct$ are given by
\[
(x \otimes y) \product (x' \otimes y') = (x \product x') \otimes (y \product y'),
\qquad\text{and}\qquad
\rbcoproduct(x) \eqdef \bcoproduct(x) - (x \otimes 1 + 1 \otimes x).
\]
\end{definition}

Note that this is not the classical compatibility relation satisfied by the product~$\product$ and the coproduct~$\coproduct$ of a Hopf algebra.
In particular, $(\IRel, \product, \coproduct)$ is not a unital inifinesimal bialgebra.
Nevertheless we will prove that for another coproduct~$\bcoproduct$, then~$(\IRel, \product, \bcoproduct)$ is a unital inifinesimal bialgebra.
We can then use the main result of~\cite{LodayRoncoCoFree}.

\begin{theorem}[{\cite[p.\,2]{LodayRoncoCoFree}}]
\label{thm:loday-ronco-cofree}
Any graded unital infinitesimal bialgebra is isomorphic to the  non-commutative polynomials algebra equipped with the deconcatenation coproduct.
\end{theorem}

The isomorphism is explicit.
Each element~$x$ of~$\Hvect$ can be written uniquely as a product $x = x_1 \product x_2 \dots \product x_k$ such that the elements $x_i$ are \defn{primitive}, \ie $\rbcoproduct(x_i) = 0$.
In other words, the algebra~$(\Hvect, \product)$ is freely generated by the primitive elements for the coproduct~$\bcoproduct$.
In our case, we will exhibit a coproduct~$\bcoproduct$ such that~$(\IRel, \product, \bcoproduct)$ satisfies \eqref{eq:unital-inf} and, as a corollary of~\cite{LodayRoncoCoFree}, we get that~$(\IRel, \product)$ is freely generated by the primitive elements of~$\bcoproduct$.

\begin{definition}
A \defn{primitive cut} is a total cut of the form~$([i],[p] \ssm [i])$ for some~$0 \le i \le p$.
\end{definition}

For example, the relation $\raisebox{.05cm}{\scalebox{.5}{\input{figures/relations/r3_13_23}}}$ admits a primitive cut at~$2$.
Every relation~$\rel[T] \in \IRel_p$ admits at least two primitive cuts~$(\varnothing, [p])$ and~$([p], \varnothing)$ which we call the \defn{trivial primitive cuts}.
Moreover, $\rel[T] = \uRS$ if and only if $T$ admits a primitive cut at~$|{\rel[R]}|$.
In particular, if~$T$ is indecomposable, then~$T$ does not admit any non-trivial primitive cut.
We define a coproduct~$\bcoproduct$ on the basis~$\ERel$ by
\[
\bcoproduct(\ERel^{\rel[T]}) \eqdef \sum_{T = \uRS} \ERel^{\rel[R]} \otimes \ERel^{\rel[S]}.
\]
By definition, this is the dual of the product~$\product$ on~$\ERel$.
This is also a sum over all primitive cuts of the relation~$\rel[T]$ and by extension, $\rbcoproduct(\ERel^{\rel[T]})$ is a sum over all non-trivial primitive cuts of~$\rel[T]$.

\begin{example}
For instance
\input{coproductDualRelationsMultbases}
\end{example}

We have that~$\ERel^{\rel[T]}$ is primitive for~$\rbcoproduct$ (\ie $\rbcoproduct(\ERel^{\rel[T]}) = 0$) if and only if~$\rel[T]$ is indecomposable.
Now, Proposition~\ref{prop:Rmultbases} is a direct consequence of the following statement together with Theorem~\ref{thm:loday-ronco-cofree}.

\begin{proposition}
$(\IRel, \product, \bcoproduct)$ is a unital infinitesimal bialgebra.
\end{proposition}

\begin{proof}
Let $\rel[R] \in \IRel_m$ and $\rel[S] \in \IRel_n$ with $p = m + n$.
We have on the one hand
\[
A \eqdef \rbcoproduct \big( \ERel^{\rel[R]} \product \ERel^{\rel[S]} \big) = \rbcoproduct \big( \ERel^{\uRS} \big) = \sum_{\underprod{\rel[R']}{\rel[S']} = \uRS} \ERel^{\rel[R']} \otimes \ERel^{\rel[S']},
\]
and on the other hand
\begin{align*}
B & \eqdef \Big( \ERel^{\rel[R]} \otimes \ERel^{\emptyset} \Big) \product \rbcoproduct \big( \ERel^{\rel[S]} \big) + \rbcoproduct \big( \ERel^{\rel[R]} \big) \product \Big( \ERel^{\emptyset} \otimes \ERel^{\rel[S]} \Big) + \ERel^{\rel[R]} \otimes \ERel^{\rel[S]} \\
& = \sum_{\rel[S] = \underprod{\rel[S_1]}{\rel[S_2]}} \ERel^{\underprod{\rel[R]}{\rel[S_1]}} \otimes \ERel^{\rel[S_2]}
+ \sum_{\rel[R] = \underprod{\rel[R_1]}{\rel[R_2]}} \ERel^{\rel[R_1]} \otimes \ERel^{\underprod{\rel[R_2]}{\rel[S]}}
+ \ERel^{\rel[R]} \otimes \ERel^{\rel[S]}.
\end{align*}
We want to prove that $A = B$.
The sum $A$ is over all non-trivial primitive cuts of $\uRS$.
The relation $\uRS$ admits a primitive cut at $m$ by definition which means that the term  $\ERel^{\rel[R]} \otimes \ERel^{\rel[S]}$ appears in the sum.
Now, let $0 < k < m$.
\begin{itemize}
\item If $\uRS$ admits a primitive cut at $k'$, this means in particular that $\rel[R]$ admits a primitive cut at $k$.
We have $\rel[R] = \underprod{\rel[R_1]}{\rel[R_2]}$ for some $\rel[R_1] \in \IRel_k$ and $\rel[R_2] \in \IRel_{m - k}$.
It is easy to check that $\uRS$ restricted to $\lbrace k+1, \dots, p \rbrace$ is indeed equal to $\underprod{\rel[R_2]}{\rel[S]}$.
\item Reciprocally, if $\rel[R]$ admits a primitive cut at $k$, \ie $\rel[R] = \underprod{\rel[R_1]}{\rel[R_2]}$ with $\rel[R_1] \in \IRel_k$, we have all $(i,j) \in \uRS$ and $(j,i) \notin \uRS$ for $i \leq k$ and $k < j \leq m$ by definition of the primitive cut and also for $i \leq k$ and $m < j \leq p$ by definition of $\uRS$.
This means that $\uRS$ admits a primitive cut $\underprod{\rel[R']}{\rel[S']} = \uRS$ at $k$.
The relation $\rel[R']$ is the restriction of $\uRS$ to $[k]$ and it is then equal to $\rel[R]_1$.
The relations $\rel[S']$ is the restriction of $\uRS$ to $\lbrace k+1, \dots, p \rbrace$ and is equal to $\underprod{\rel[R_2]}{\rel[S]}$.
\end{itemize}
We can use a similar argument for $k > m$ and we then obtain that the primitive cuts of $\uRS$ exactly correspond to the primitive cuts of $\rel[S]$ which proves the result.
\end{proof}

As an algebra, $\K\IRel$ is then generated by indecomposable relations.
It is well known that there is a direct relation between the Hilbert series of an algebra and the generating series of its indecomposable elements.
Namely, if
\[
F(x) \eqdef \sum_{n \geq 0} R_n x^n = 1 + x + 4x^2 + 64 x^3 +4098 x^4 + \dots
\]
is the Hilbert series of~$\K\IRel$, where~$R_n = 2^{n(n-1)}$ is the number of (reflexive) integer binary relations, then it is related to the the generating series~$I(x)$ of indecomposable relations by
\[
\frac{1}{1 - I(x)} = R(x).
\]
In particular, the number of indecomposable relations~$I_n$ can be computed by an inclusion--exclusion formula
\[
I_n = \sum_{n_1 + \dots + n_k = n} (-1)^{k+1} R_{n_1} \dots R_{n_k}
\]
which gives the coefficients of Table~\ref{table:nb-indecomp-rel}.
There does not seem to be another, more direct, combinatorial enumeration.
\begin{table}[t]
    \begin{tabular}{r|lllll}
        $n$ & 1 & 2 & 3 & 4 & 5 \\ \hline
        $R_n$ & 1 & 4 & 64 & 4096 & 1048576 \\ \hline
        $I_n$ & 1 & 3 & 57 & 3963 & 1040097
    \end{tabular}
    \caption{Number of binary relations and indecomposable binary relations on~$[n]$.}
    \label{table:nb-indecomp-rel}
\end{table}
Nevertheless, indecomposable relations do have an interesting structural property when looking at the weak order lattice.

\begin{proposition}
The set of indecomposable relations of size $n$ forms an upper-ideal of the weak order lattice on $\IRel_n$ (\ie if~$\rel[R]$ is indecomposable, then any~$\rel[S]$ with~$\rel[R] \wole \rel[S]$ is also indecomposable).
\end{proposition}

\begin{proof}
Consider two binary relations~$\rel[R], \rel[S] \in \IRel_n$ such that~$\rel[R] \wole \rel[S]$ and $\rel[S]$ admits a primitive cut at some $k$.
For all~$i \leq k < j$, we have~$i \rel[S] j$ and~$j \notrel[S] i$ since~$k$ is a primitive cut of~$\rel[S]$.
Since~$\rel[R] \wole \rel[S]$, we have~${\Inc{\rel} \supseteq \Inc{\rel[S]}}$ and~${\Dec{\rel} \subseteq \Dec{\rel[S]}}$, and thus~$i \rel[R] j$ and~$j \notrel [S] i$  for all $i \leq k < j$.
This implies that $\rel[R]$ also admits a primitive cut at $k$.
\end{proof}

Note however that the ideal of indecomposable relations might have multiple minimal elements.
For example for~$n=2$, there are~$3$ indecomposable relations (over $4$ relations in total) and $2$ minimal elements: $\raisebox{.05cm}{\scalebox{.5}{\input{figures/relations/r2_0}}}$ and $\raisebox{.05cm}{\scalebox{.5}{\input{figures/relations/r2_12_21}}}$.

\section{Integer posets}
\label{sec:integerPosets}

We now focus on \defn{integer posets}, \ie integer relations that are reflexive ($x \rel x$), transitive ($x \rel y \rel z \implies x \rel z$) and antisymmetric ($x \rel y \implies y \notrel x$).
Let~$\IPos_n$ be the set of all posets on~$[n]$ and let~$\IPos \eqdef \bigsqcup_{n \ge 0} \IRel_n$.

As we will only work with posets in the rest of the paper, we generally prefer to use notations like~$\less, \bless, \dashv$ which speak for themselves, rather than our previous notations~$\rel[R], \rel[S]$ for arbitrary binary relations.
It also allows us to write~$a \more b$ for~$b \less a$, in particular when~$a < b$.

We still denote by~$\wole$ the weak order given in Definition~\ref{def:weakOrder}.
The following statement is the keystone of~\cite{ChatelPilaudPons}.

\begin{theorem}[{\cite[Thm.~1]{ChatelPilaudPons}}]
The weak order on the integer posets of~$\IPos_n$ is a lattice.
\end{theorem}

We now define a Hopf algebra on posets.
We consider the vector space~$\K\IPos \eqdef \bigoplus_{n \ge 0} \K\IPos_n$ indexed by all integer posets of arbitrary size.
We denote by~$(\FPos_{\less})_{{\less} \in \IPos}$ the standard basis of~$\K\IPos$.

\begin{proposition}
\label{prop:algebraPosets}
For any~$\rel[R], \rel[S] \in \IRel$,
\begin{enumerate}[(i)]
\item if the shifted shuffle~$\rel[R] \shiftedShuffle \rel[S]$ contains at least a poset, then~$\rel[R]$ and~$\rel[S]$ are both posets,
\item if~$\rel[R]$ and~$\rel[S]$ are both posets, then all relations in the convolution~$\rel[R] \convolution \rel[S]$ are posets.
\end{enumerate}
Therefore, the vector subspace of~$\K\IRel$ generated by integer relations which are not posets is a Hopf ideal of~$(\K\IRel, \product, \coproduct)$.
The quotient of the integer relation algebra~$(\K\IRel, \product, \coproduct)$ by this ideal is thus a Hopf algebra~$(\K\IPos, \product, \coproduct)$ on integer posets.
\end{proposition}

\begin{proof}
For~(i), let~$\rel[R] \in \IRel_m$ and~$\rel[S] \in \IRel_n$ be such that the shifted shuffle~$\rel[R] \shiftedShuffle \rel[S]$ contains a poset~$\rel[T]$.
Then~$\rel[R] = \rel[T]_{[m]}$ and~$\rel[S] = \rel[T]_{\overline{[n]}}$ are antisymmetric and transitive since~$\rel[T]$ is.

For~(ii), consider two posets~$\rel[R], \rel[S] \in \IPos$ and let~$\rel[T] \in \rel[R] \convolution \rel[S]$.
Let~$(X,Y)$ be the total cut of~$\rel[T]$ such that~$\rel[T]_X = \rel[R]$ and~$\rel[T]_Y = \rel[S]$.
We prove that~$\rel[T]$ is a poset:
\begin{description}
\item[Antisymmetry]
Let~$u, v \in \N$.
If~$u$ and~$v$ both belong to~$X$ (resp.~to~$Y$), then~${u \rel[T] v \implies v \notrel[T] u}$ since~$\rel[T]_X = \rel[R]$ (resp.~$\rel[T]_Y = \rel[S]$) is antisymmetric.
Otherwise, $u \rel[T] v$ if and only if~$u \in X$ and~$v \in Y$, while~$v \rel[T] u$ if and only if~$v \in X$ and~$u \in Y$.
Thus~$\rel[T]$ is antisymmetric.

\item[Transitivity]
Let~$u, v, w \in \N$ such that~$u \rel[T] v \rel[T] w$.
If~$u$ and~$w$ both belong to~$X$ (resp.~to~$Y$), then so does~$v$ and~$u \rel[T] w$ since~$\rel[T]_X = \rel[R]$ (resp.~$\rel[T]_Y = \rel[S]$) is transitive.
Otherwise, $u \in X$ and~$w \in Y$ (since~$\rel[T] \cap (Y \times X) = \varnothing$ and~$u \rel[T] v \rel[T] w$), thus~$u \rel[T] w$ (since~$X \times Y \subseteq \rel[T]$).
Thus~$\rel[T]$ is transitive.
\qedhere
\end{description}
\end{proof}

\begin{remark}
\label{rem:convolutionPosets}
Although not needed for the Hopf algebra quotient, observe that the convolution satisfies a property similar to Proposition~\ref{prop:algebraPosets}\,(i):
if~$\rel[R], \rel[S] \in \IRel$ are such that the convolution~$\rel[R] \convolution \rel[S]$ contains at least a poset, then~$\rel[R]$ and~$\rel[S]$ are both posets.
\end{remark}

For any poset~$\less$, we denote by~$\FPos_{\less}$ the image of~$\FRel_{\less}$ through the trivial projection~$\K\IRel \to \K\IPos$.

\begin{example}
In practice, for two posets~${\less}, {\bless} \in \IPos$, we compute the product~$\FPos_{\less} \product \FPos_{\bless}$ in~$\K\IPos$ by deleting all non-poset summands in the product~$\FRel_{\less} \product \FRel_{\bless}$ in~$\K\IRel$:
\input{productPosets}
The coproduct is even simpler: all relations that appear in the coproduct~$\coproduct(\FRel_{\less})$ of a poset~$\less$ are automatically posets by Remark~\ref{rem:convolutionPosets}:
\input{coproductPosets}
\end{example}

\begin{proposition}
\label{prop:productPoset}
For~${\less} \in \IPos_m$ and~${\bless} \in \IPos_n$, the product~$\FPos_{\less} \product \FPos_{\bless}$ is the sum of~$\FPos_{\dashv}$, where~$\dashv$ runs over the interval between $\underprod{\less}{\bless}$ and $\overprod{\less}{\bless}$ in the weak order on~$\IPos_{m+n}$.
\end{proposition}

\begin{proof}
It is a direct consequence of Proposition~\ref{prop:characterizationShuffleProduct} and the fact that for any two posets~${\less}, {\bless} \in \IPos$, the relations~$\underprod{\less}{\bless}$ and~$\overprod{\less}{\bless}$ are both posets.
\end{proof}

\begin{example}
For instance, the product~$\FPos_{\scalebox{.5}{\input{figures/relations/r2_12}}} \! \! \product \FPos_{\scalebox{.5}{\input{figures/relations/r1}}}$ corresponds to the interval of \fref{fig:productIntervalPosets} from~$\FPos_{\scalebox{.5}{\input{figures/relations/r3_12_13_23}}}$ to~$\FPos_{\scalebox{.5}{\input{figures/relations/r3_12_31_32}}}$ in the weak order on~$\IPos_3$.

\begin{figure}[t]
    \includegraphics[scale=.6]{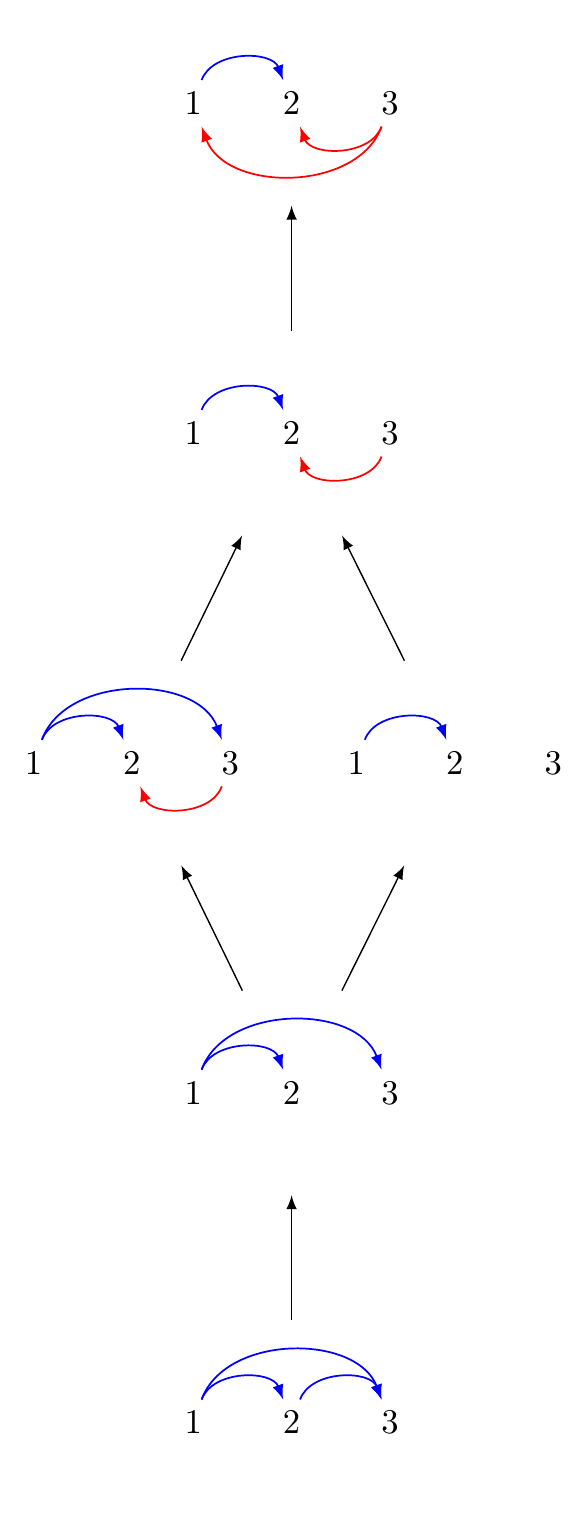}
    \caption{Interval corresponding to a product of posets. Its is obtained from \fref{fig:productInterval} by deleting all binary relations which are not posets.}
    \label{fig:productIntervalPosets}
\end{figure}
\end{example}

In a similar way, we directly obtain multiplicative bases~$\EPos$ and~$\HPos$ by taking the quotient of the bases~$\ERel$ and~$\HRel$ respectively:
\[
\EPos^{\less} = \sum_{\less \wole \bless} \FPos_{\bless}
\qquad\text{and}\qquad
\HPos^{\less} = \sum_{\bless \wole \less} \FPos_{\bless}.
\]
Note that if~$\rel$ is a relation that is not a poset, then the quotient~$\EPos^{\rel}$ of the element~$\ERel^{\rel}$ is not equal to~$0$: the leading term~$\FRel^{\rel}$ is sent to~$0$ and so~$\EPos^{\rel}$ can be expressed as a sum of elements~$\EPos^{\rel[R']}$ where~$\rel[R']$ is a poset with~$\rel[R'] \prec \rel[R]$.

\begin{proposition}
The sets~$(\EPos^{\less})_{\less \in \IPos}$ and~$(\HPos^{\less})_{\less \in \IPos}$ form multiplicative bases of~$\K\IPos$ with
\[
\EPos^{\less} \product \EPos^{\bless} = \EPos^{\underprod{\less}{\bless}}
\qquad\text{and}\qquad
\HPos^{\less} \product \HPos^{\bless} = \HPos^{\overprod{\less}{\bless}}.
\]
Besides, as an algebra, $\K\IPos$ is freely generated by the elements $(\EPos^{\less})$ where $\less$ is under-indecomposable and, equivalently, by the elements $(\HPos^{\less})$ where $\less$ is over-indecomposable.
\end{proposition}

\begin{proof}
This derives directly from Proposition~\ref{prop:Rmultbases} and the fact that if $\less$ and $\bless$ are posets, then $\underprod{\less}{\bless}$ and $\overprod{\less}{\bless}$ are also posets.
To prove that the algebra is freely generated by the indecomposable elements, one can follow the proof of Proposition~\ref{prop:relfree} as everything still holds when restricting to posets.
\end{proof}

In the next two sections of the paper, we will use the Hopf algebra on integer posets constructed in this section to reinterpret classical Hopf algebras on permutations~\cite{MalvenutoReutenauer} (see Sections~\ref{subsec:MalvenutoReutenauer}, \ref{subsubsec:permutationsQuotientAlgebra} and~\ref{subsubsec:permutationsSubalgebra}), ordered partitions~\cite{Chapoton} (see Section~\ref{subsubsec:orderedPartitionsQuotientAlgebra}), binary trees~\cite{LodayRonco} (see Sections~\ref{subsec:LodayRonco} and~\ref{subsec:subalgebrasTO}) and Schr\"oder trees~\cite{Chapoton} (see Section~\ref{subsec:subalgebrasTO}).
Moreover, we obtain Hopf structures on the intervals of the weak order (see Sections~\ref{subsubsec:intervalsQuotientAlgebra} and~\ref{subsubsec:intervalsSubalgebra}) and on the intervals of the Tamari lattice (see Section~\ref{subsec:subalgebrasTO}), that remained undiscovered to the best of our knowledge.
All these algebras and their connections are summarized in Table~\ref{table:listObjects} and \fref{fig:roadMap}.

\begin{figure}[h]
	\centerline{\input{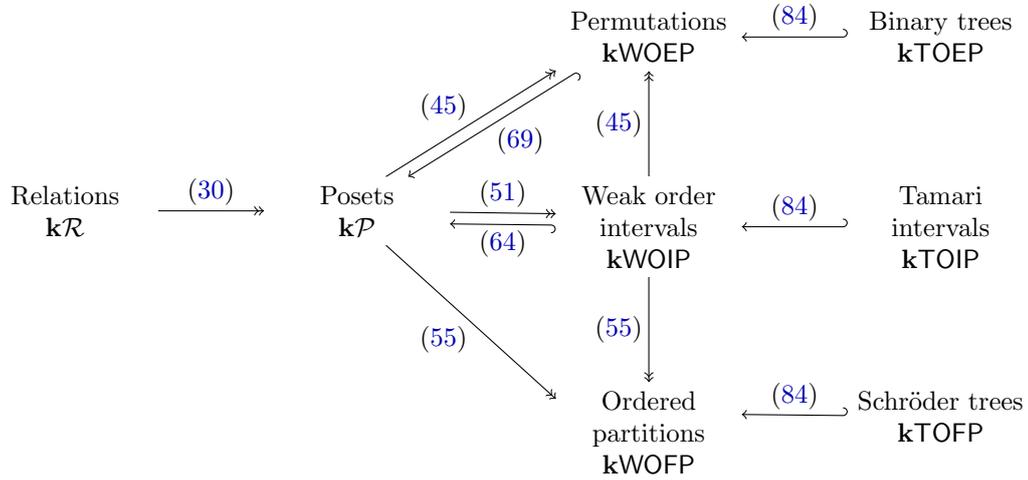}}
	\caption{A roadmap through the different Hopf algebras studied in this paper. An arrow~$\twoheadrightarrow$ indicates a quotient Hopf algebra, while an arrow~$\hookrightarrow$ indicates a Hopf subalgebra. The label on each arrow refers to the corresponding proposition.}
	\label{fig:roadMap}
\end{figure}

\section{Permutations, weak order intervals, and ordered partitions}
\label{sec:WO}

We now consider our first three families of specific integer posets.
These families respectively correspond to the elements ($\WOEP$), the intervals ($\WOIP$) and the faces ($\WOFP$) in the classical weak order on permutations.
We construct Hopf algebras on~$\WOEP$, $\WOIP$ and~$\WOFP$ as quotients (Section~\ref{subsec:quotientAlgebrasWO}) or subalgebras (Section~\ref{subsec:subalgebrasWO}) of the integer poset Hopf algebra~$(\IPos, \product, \coproduct)$.
For the constructions as quotients, the important point is that all these families of posets are defined by local conditions on their relations, and that a contradiction to these conditions cannot be destroyed by the product or the coproduct.
For the constructions as subalgebras, we use surjections from posets to $\WOEP$ or~$\WOIP$ whose fibers are stable by product and coproduct.
Using quotients or subalgebras, we construct Hopf algebras on $\WOEP$ (resp.~$\WOFP$) isomorphic to the Malvenuto--Reutenauer Hopf algebra on permutations~\cite{MalvenutoReutenauer} (resp.~to the Chapoton Hopf algebra on surjections~\cite{Chapoton}), and we obtain a Hopf algebra on intervals of the weak order that was not constructed earlier to the best of our knowledge.

\subsection{Permutations and the Malvenuto--Reutenauer algebra}
\label{subsec:MalvenutoReutenauer}

Recall that the classical \defn{weak order} on the permutations of~$\fS_n$ is defined by $\sigma \wole \tau$ if and only if $\inv(\sigma) \subseteq \inv(\tau)$, where $\inv(\sigma) \eqdef \set{(a,b) \in [n]^2}{a \le b \text{ and } \sigma^{-1}(a) \le \sigma^{-1}(b)}$ denotes the \defn{inversion set} of~$\sigma$.
This order is a lattice with minimal element~$[1,2,\dots,n]$ and maximal element~$[n,\dots,2,1]$.

For two permutations~${\sigma \in \fS_m}$ and~${\tau \in \fS_n}$, the \defn{shifted shuffle}~$\sigma \shiftedShuffle \tau$ (resp.~the \defn{convolution}~$\sigma \convolution \tau$) is the set of permutations of~$\fS_{m+n}$ whose first~$m$ values (resp.~positions) are in the same relative order as~$\sigma$ and whose last~$n$ values (resp.~positions) are in the same relative order as~$\tau$.
For example,
\begin{align*}
{\red 12} \shiftedShuffle {\blue 231} & = \{ {\red 12}{\blue 453}, {\red 1}{\blue 4}{\red 2}{\blue 53}, {\red 1}{\blue 45}{\red 2}{\blue 3}, {\red 1}{\blue 453}{\red 2}, {\blue 4}{\red 12}{\blue 53}, {\blue 4}{\red 1}{\blue 5}{\red 2}{\blue 3}, {\blue 4}{\red 1}{\blue 53}{\red 2}, {\blue 45}{\red 12}{\blue 3}, {\blue 45}{\red 1}{\blue 3}{\red 2}, {\blue 453}{\red 12} \}, \\
\text{and}\quad
{\red 12} \convolution {\blue 231} & = \{ {\red 12}{\blue 453}, {\red 13}{\blue 452}, {\red 14}{\blue 352}, {\red 15}{\blue 342}, {\red 23}{\blue 451}, {\red 24}{\blue 351}, {\red 25}{\blue 341}, {\red 34}{\blue 251}, {\red 35}{\blue 241}, {\red 45}{\blue 231} \}.
\end{align*}

\noindent
Recall that the Malvenuto--Reutenauer Hopf algebra~\cite{MalvenutoReutenauer} is the Hopf algebra on permutations with product~$\product$ and coproduct~$\coproduct$ defined by
\[
\F_\sigma \product \F_\tau \eqdef \sum_{\rho \in \sigma \shiftedShuffle \tau} \F_\rho
\qquad\text{and}\qquad
\coproduct(\F_\rho) \eqdef \sum_{\rho \in \sigma \convolution \tau} \F_\sigma \otimes \F_\tau.
\]

\begin{example}
\label{exm:algebraFQSym}
For example, we have
\[
\F_{12} \product \F_1 = \F_{123} + \F_{132} + \F_{312}
\quad\text{and}\quad
\coproduct (\F_{132}) = \F_{132} \otimes \F_{\varnothing} + \F_{1} \otimes \F_{21} + \F_{12} \otimes \F_{1} + \F_{\varnothing} \otimes \F_{132}.
\]
\end{example}

\begin{remark}
\label{rem:productIntervalsMR}
For~$\sigma \in \fS_m$ and~$\tau \in \fS_n$, we have~$\sigma \shiftedShuffle \tau = [\underprod{\sigma}{\tau}, \overprod{\sigma}{\tau}]$.
More generally, for any permutations~${\sigma \wole \sigma' \in \fS_m}$ and~${\tau \wole \tau' \in \fS_n}$, we have
\[
\bigg( \sum_{\sigma \wole \lambda \wole \sigma'} \F_\lambda \bigg) \product \bigg( \sum_{\tau \wole \mu \wole \tau'} \F_\mu \bigg) = \sum_{\underprod{\sigma}{\tau} \wole \nu \wole \overprod{\sigma'}{\tau'}} \F_\nu.
\]
In other words, weak order intervals are stable by the product~$\product$ on~$\fS$. Note that there are not stable by the coproduct~$\coproduct$.
\end{remark}

\subsection{Weak order element, interval and face posets}
\label{subsec:WOEIFP}

We now briefly recall how the elements, the intervals and the faces of the classical weak order can be interpreted as specific interval posets as developed in~\cite{ChatelPilaudPons}.

\subsubsection{Elements}

We interpret each permutation~$\sigma \in \fS_n$ as the \defn{weak order element poset}~$\less_\sigma$ defined by~$u \less_\sigma v$ if~${\sigma^{-1}(u) \le \sigma^{-1}(v)}$.
In other words, the poset~$\less_\sigma$ is the chain~$\sigma(1) \less_\sigma \dots \less_\sigma \sigma(n)$.
See \fref{fig:woep} for an example with~$\sigma = 2751346$.
\begin{figure}[ht]
    \vspace{-.3cm}
    \input{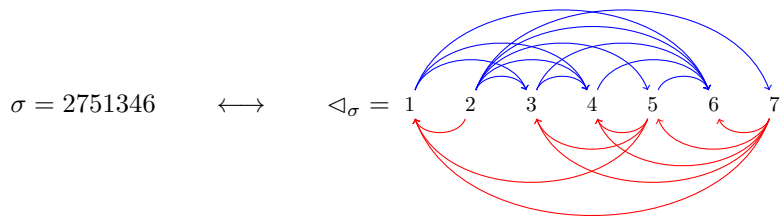}
    \vspace{-.4cm}
    \caption{A Weak Order Element Poset ($\WOEP$).}
    \label{fig:woep}
    \vspace{-.4cm}
\end{figure}

We define
\[
\WOEP_n \eqdef \set{{\less_\sigma}}{\sigma \in \fS_n}
\qquad\text{and}\qquad
\WOEP \eqdef \bigsqcup_{n \in \N} \WOEP_n.
\]
These posets are clearly characterized as follows, which enables to recover the weak order on permutations.

\begin{lemma}
\label{lem:characterizationWOEP}
A poset~${\less}$ is in~$\WOEP_n$ if and only if $a \less b$ or~$a \more b$ for all~$a,b \in [n]$.
\end{lemma}

\begin{proposition}[{\cite[Prop.~23 \& 24]{ChatelPilaudPons}}]
\label{prop:weakOrderWOEP}
The map~$\sigma \mapsto {\less}_\sigma$ is a lattice isomorphism from the weak order on permutations of~$\fS_n$ to the sublattice of the weak order on~$\IPos_n$ induced by $\WOEP_n$.
\end{proposition}

\subsubsection{Intervals}
\label{subsubsec:WOIP}

We now present a similar interpretation of the intervals of the weak order.
For~${\sigma \wole \sigma' \in \fS_n}$, we consider the weak order interval~$[\sigma, \sigma'] \eqdef \set{\tau \in \fS_n}{\sigma \wole \tau \wole \sigma'}$.
The permutations of the interval~$[\sigma, \sigma']$ are precisely the linear extensions of the \defn{weak order interval poset} ${\less_{[\sigma,\sigma']} \eqdef \bigcap_{\sigma \wole \tau \wole \sigma'} {\less_\tau} = {\less_{\sigma}} \cap {\less_{\sigma'}} = {\Inc{\less_{\sigma'}}} \cup {\Dec{\less_\sigma}}}$, see the example on~\fref{fig:woip}.
\begin{figure}[ht]
    \vspace{-.7cm}
    \input{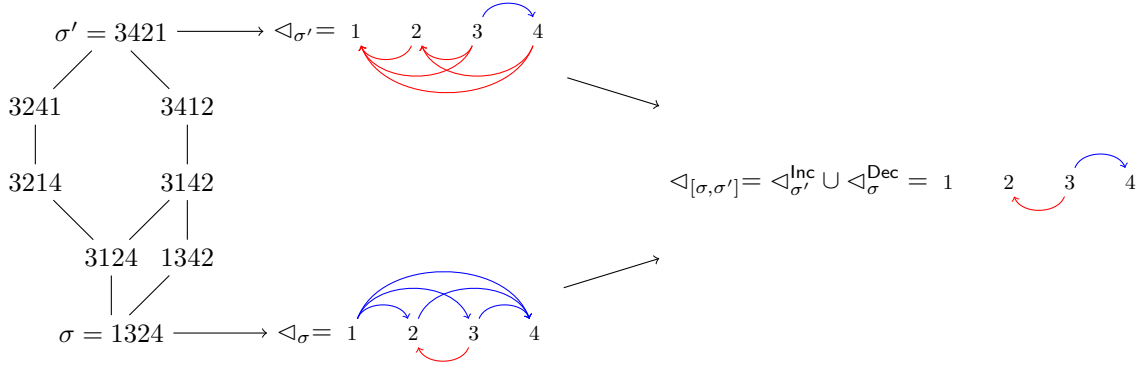}
    \vspace{-.7cm}
    \caption{A Weak Order Interval Poset ($\WOIP$).}
    \label{fig:woip}
    \vspace{-.5cm}
\end{figure}

We define
\[
\WOIP_n \eqdef \set{{\less_{[\sigma,\sigma']}}}{\sigma \wole \sigma' \in \fS_n}
\qquad\text{and}\qquad
\WOIP \eqdef \bigsqcup_{n \in \N} \WOIP_n.
\]

Weak order interval posets are precisely the integer posets which admit both a minimal and a maximal linear extension.
They were characterized by A.~Bjorner and M.~Wachs~\cite{BjornerWachs} as follows.

\begin{proposition}[{\cite[Thm.~6.8]{BjornerWachs}}]
\label{prop:characterizationWOIP}
A poset~${\less}$ is in~$\WOIP_n$ if and only if ${a \less c \Rightarrow (a \less b \text{ or } b \less c)}$ and ${a \more c \Rightarrow (a \more b \text{ or } b \more c)}$ for all~$1 \le a < b < c \le n$.
\end{proposition}

This condition clearly contains two separate conditions on the increasing subrelation and on the decreasing subrelation of~$\less$, and it will be convenient to split these conditions.
We thus consider the set~$\IWOIP_n$ (resp.~$\DWOIP_n$) of posets of~$\IPos_n$ which admit a weak order maximal (resp.~minimal) linear extension.
These posets are characterized as follows.

\begin{proposition}[{\cite[Prop.~32]{ChatelPilaudPons}}]
\label{prop:characterizationIWOIPDWOIP}
For a poset~${\less} \in \IPos_n$,
\begin{itemize}
\item ${\less} \in \IWOIP_n\;\; \iff \forall \; 1 \le a < b < c \le n, \;\; a \less c \implies a \less b \text{ or } b \less c,$
\item ${\less} \in \DWOIP_n \iff \forall \; 1 \le a < b < c \le n, \;\; a \more c \implies a \more b \text{ or } b \more c.$
\end{itemize}
Moreover,
\begin{itemize}
\item if~${\less} \in \IWOIP_n$, its maximal linear extension is ${\maxle{\less}} \eqdef {\less} \cup \set{(b,a)}{a < b \text{ incomparable in } {\less}}$,
\item if~${\less} \in \DWOIP_n$, its minimal linear extension is~${\minle{\less}} \eqdef {\less} \cup \set{(a,b)}{a < b \text{ incomparable in } {\less}}$.
\end{itemize}
\end{proposition}

See~\fref{fig:iwoip_dwoip} for an example.

\begin{figure}[ht]
\input{figures/example_maxminle}
\caption{Examples of $\IWOIP$ and $\DWOIP$ with their maximum (resp. minimum) linear extensions.}
\label{fig:iwoip_dwoip}
\end{figure}

Finally, the weak order on~$\WOIP_n$ corresponds to the Cartesian product lattice on the intervals, described in the following statement.

\begin{proposition}[{\cite[Prop.~27 \& Coro.~28]{ChatelPilaudPons}}]
\label{prop:weakOrderWOIP}
~
\begin{enumerate}[(i)]
\item If~$\sigma \wole \sigma'$ and~$\tau \wole \tau'$ in~$\fS_n$, then \mbox{${\less_{[\sigma,\sigma']}} \wole {\less_{[\tau,\tau']}} \! \iff \! \sigma \wole \tau$ and~$\sigma' \wole \tau'$}.
\item The weak order on~$\WOIP_n$ is a lattice with meet~${\less_{[\sigma,\sigma']}} \meetWOIP {\less_{[\tau,\tau']}} = {\less_{[\sigma \meetWO \tau, \, \sigma' \meetWO \tau']}}$ and join~${\less_{[\sigma,\sigma']}} \joinWOIP {\less_{[\tau,\tau']}} = {\less_{[\sigma \joinWO \tau, \, \sigma' \joinWO \tau']}}$.
However, the weak order on~$\WOIP_n$ is not a sublattice of the weak order on~$\IPos_n$.
\end{enumerate}
\end{proposition}

\subsubsection{Faces}

Recall that the permutahedron~$\Perm \eqdef \conv\set{(\sigma_1, \dots, \sigma_n)}{\sigma \in \fS_n}$ has vertices in bijections with permutations of~$\fS_n$ and faces in bijections with ordered partitions of~$[n]$.
In this paper, we see an ordered partition~$\pi$ of~$[n]$ as a \defn{weak order face poset}~$\less_\pi$ defined by~$u \less_\pi v$ if~$\pi^{-1}(u) < \pi^{-1}(v)$ (meaning the block of~$u$ is before the block of~$v$ in~$\pi$).
See~\fref{fig:wofp} for an example with~$\pi = 125|37|46$.

\begin{figure}[!h]
    \vspace{-.45cm}
    \input{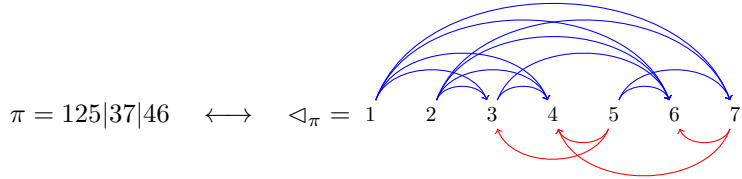}
    \vspace{-1.3cm}
    \caption{A Weak Order Face Poset ($\WOFP$).}
    \label{fig:wofp}
    \vspace{-.4cm}
\end{figure}

We define
\[
\WOFP_n \eqdef \set{{\less_\pi}}{\pi \text{ ordered partition of } [n]}
\qquad\text{and}\qquad
\WOFP \eqdef \bigsqcup_{n \in \N} \WOFP_n.
\]
Again, the posets of~$\WOFP_n$ admit a simple local characterization and the weak order on~$\WOFP_n$ corresponds to a relevant lattice on faces of the permutahedron previously considered in~\cite{KrobLatapyNovelliPhanSchwer, PalaciosRonco, DermenjianHohlwegPilaud}.

\begin{proposition}[{\cite[Prop.~30]{ChatelPilaudPons}}]
\label{prop:characterizationWOFP}
A poset~${\less}$ is in~$\WOFP_n$ if and only if~${\less} \in \WOIP_n$ (Proposition~\ref{prop:characterizationWOIP}) and ${(a \less b \iff b \more c)}$ and~${(a \more b \iff b \less c)}$ for all~$ a < b < c$ with $a \not\less c$ and~$a \not\more c$.
\end{proposition}

\begin{proposition}[{\cite[Sect.~2.1.3]{ChatelPilaudPons}}]
\label{prop:weakOrderWOFP}
~
\begin{enumerate}[(i)]
\item For any ordered partitions~$\pi, \pi'$ of~$[n]$, we have~${\less_\pi} \wole {\less_{\pi'}} \iff \pi \wole \pi'$ in the facial weak order of~\cite{KrobLatapyNovelliPhanSchwer, PalaciosRonco, DermenjianHohlwegPilaud}.
\item The weak order on~$\WOFP_n$ is a lattice but not a sublattice of the weak order on~$\IPos_n$, nor on~$\WOIP_n$.
\end{enumerate}
\end{proposition}

\subsection{Quotient algebras}
\label{subsec:quotientAlgebrasWO}

We now construct Hopf algebras on~$\WOEP$, $\WOIP$ and~$\WOFP$ as quotient algebras of the integer poset algebra.
For all these constructions, the crucial point is that all these families of posets are characterized by local conditions on their relations, and that a contradiction to these conditions cannot be destroyed by the product or the coproduct on posets.

\subsubsection{Elements}
\label{subsubsec:permutationsQuotientAlgebra}

We first interpret the Malvenuto--Reutenauer Hopf algebra as a quotient of the integer poset Hopf algebra~$(\K\IPos, \product, \coproduct)$.
For that, we use the characterization of Lemma~\ref{lem:characterizationWOEP}.

\begin{proposition}
\label{prop:algebraWOEP}
For any~${\less}, {\bless} \in \IPos$,
\begin{enumerate}[(i)]
\item if the shifted shuffle~${\less} \shiftedShuffle {\bless}$ contains a total order, then~${\less}$ and~${\bless}$ are total orders,
\item if~${\less}$ and~${\bless}$ are both total orders, then all relations in the convolution~${\less} \convolution {\bless}$ are total orders.
\end{enumerate}
Therefore, the vector subspace of~$\K\IPos$ generated by integer posets that are not total is a Hopf ideal of~$(\K\IPos, \product, \coproduct)$.
The quotient of the integer poset Hopf algebra~$(\K\IPos, \product, \coproduct)$ by this ideal is thus a Hopf algebra~$(\K\WOEP^\quotient, \product, \coproduct)$ on total orders.
\end{proposition}

\begin{proof}
For~(i), let~${\less} \in \IPos_m$ and~${\bless} \in \IPos_n$ be such that the shifted shuffle~${\less} \shiftedShuffle {\bless}$ contains a total order~$\dashv$.
Then~${\less} = {\dashv}_{[m]}$ and~${\bless} = {\dashv}_{\overline{[n]}}$ are total orders.

For~(ii), consider two total orders~${\less}, {\bless} \in \WOEP$ and let~${\dashv} \in {\less} \convolution {\bless}$.
Let~$(X,Y)$ be the total cut of~$\dashv$ such that~${\dashv}_X = {\less}$ and~${\dashv}_Y = {\bless}$.
Let~$u,v \in \N$.
If~$u$ and~$v$ both belong to~$X$ (resp.~to~$Y$), then~$u \dashv v$ or~$v \dashv u$ since~${\dashv}_X = {\less}$ (resp.~${\dashv}_Y = {\bless}$) is total.
Otherwise, $u \dashv v$ if and only if~$u \in X$ and~$v \in Y$, while~$v \dashv u$ if and only if~$v \in X$ and~$u \in Y$.
Thus~$\dashv$ is total.
\end{proof}

\begin{remark}
\label{rem:convolutionWOEP}
Although not needed for the Hopf algebra quotient, observe that the convolution satisfies a property similar to Proposition~\ref{prop:algebraWOEP}\,(i):
if~${\less}, {\bless} \in \IPos$ are such that the convolution~${\less} \convolution {\bless}$ contains at least a total order, then~${\less}$ and~${\bless}$ are both total orders.
\end{remark}

For any weak order element poset~$\less$, we denote by~$\FWOEPquo_{\less}$ the image of~$\FRel_{\less}$ through the trivial projection~$\K\IRel \to \K\IPos$.

\begin{example}
\label{exm:algebraWOEP}
In practice, for any two total orders~${\less}, {\bless} \in \WOEP$, we compute the product $\FWOEPquo_{\less} \product \FWOEPquo_{\bless}$ in~$\K\WOEP^\quotient$ by deleting all summands not in~$\WOEP$ in the product~$\FRel_{\less} \product \FRel_{\bless}$ in~$\K\IRel$:
\input{productWOEP}
The coproduct is even simpler: all relations that appear in the coproduct~$\coproduct(\FRel_{\less})$ of a total order~${\less} \in \WOEP$ are automatically in $\WOEP$ by Remark~\ref{rem:convolutionWOEP}:
\input{coproductWOEP}
\end{example}

\begin{proposition}
\label{prop:isomorphismMalvenutoReutenauer}
The map~$\F_\sigma \mapsto \FWOEPquo_{\less_\sigma}$ defines a Hopf algebra isomorphism from the Malvenuto--Reutenauer Hopf algebra on permutations~\cite{MalvenutoReutenauer} to the quotient Hopf algebra~$(\K\WOEP^\quotient, \product, \coproduct)$.
\end{proposition}

\begin{proof}
We just need to show that for any two permutations~$\sigma \in \fS_m$ and~$\tau \in \fS_n$, we have
\[
({\less}_\sigma \shiftedShuffle {\less}_\tau) \cap \WOEP = \set{{\less}_\rho}{\rho \in \sigma \shiftedShuffle \tau}
\quad\text{and}\quad
{\less}_\sigma \convolution {\less}_\tau = \set{{\less}_\rho}{\rho \in \sigma \convolution \tau}.
\]
For the shuffle product, consider first an element of~$({\less}_\sigma \shiftedShuffle {\less}_\tau) \cap \WOEP$.
By definition, it is of the form~$\less_\rho$ for some~$\rho \in \fS_{m+n}$.
Moreover, ${\less_\sigma} \subset {\less_\rho}$ so that the first~$m$ values of~$\rho$ are in the order of~$\sigma$ and~$\overline{{\less_\tau}} \subset {\less_\rho}$ so that the last~$n$ values of~$\rho$ are in the order of~$\tau$.
Thus~$\rho \in \sigma \shiftedShuffle \tau$.
Conversely, for~$\rho \in \sigma \shiftedShuffle \tau$, we have
\[
{\less_{\rho}} = {\less_\sigma} \cup \overline{\less_\tau} \cup \set{(i,j)}{i \in [m], j \in [n], \rho^{-1}(i) < \rho^{-1}(j)} \cup \set{(j,i)}{i \in [m], j \in [n], \rho^{-1}(i) > \rho^{-1}(j)}.
\]
Therefore, ${\less_\rho} \in {\less_\sigma} \shiftedShuffle {\less_\tau}$.
The proof for the convolution is similar and left to the reader.
\end{proof}

\begin{example}
Compare Examples~\ref{exm:algebraFQSym} and~\ref{exm:algebraWOEP}.
\end{example}

In particular, Proposition~\ref{rem:productIntervalsMR} can be restated on $\WOEP$ as follow.

\begin{proposition}
\label{prop:isomorphisMalvenutoReutenauer}
For any~${\less} \in \WOEP_m$ and~${\bless} \in \WOEP_n$, the product~$\FWOEPquo_{\less} \product \FWOEPquo_{\bless}$ is the sum of~$\FWOEPquo_{\dashv}$, where~$\dashv$ runs over the interval between $\underprod{\less}{\bless}$ and $\overprod{\less}{\bless}$ in the weak order on~${\WOEP_{m+n}}$.
\end{proposition}

\begin{proof}
It is a direct consequence of Proposition~\ref{prop:characterizationShuffleProduct} and the fact that for any two~${{\less}, {\bless} \in \WOEP}$, the relations~$\underprod{\less}{\bless}$ and~$\overprod{\less}{\bless}$ are both total orders.
Indeed, for any two permutations~$\sigma$ and~$\tau$, we have~$\underprod{\less_\sigma}{\less_\tau} = \less_{\underprod{\sigma}{\tau}}$ and~$\overprod{\less_\sigma}{\less_\tau} = \less_{\overprod{\sigma}{\tau}}$.
\end{proof}

\subsubsection{Intervals}
\label{subsubsec:intervalsQuotientAlgebra}

We now present a Hopf algebra structure on intervals of the weak order.
Before we start, let us make some observations:
\begin{itemize}
\item By Remark~\ref{rem:productIntervalsMR}, the product~$\product$ of the Malvenuto--Reutenauer algebra provides a natural algebra structure on weak order intervals. However, this does not define a Hopf algebra on weak order intervals as intervals are not stable by the coproduct~$\coproduct$ on~$\fS$. Indeed, note that~$\coproduct(\F_{1423} + \F_{4123})$ contains~$\F_1 \otimes \F_{123}$ and~$\F_1 \otimes \F_{312}$ but not~$\F_1 \otimes \F_{132}$ while~$132$ belongs to the weak order interval~$[123, 312]$.
\item To the best of our knowledge, the Hopf algebra presented below did not appear earlier in the literature. In fact, we are not aware of any Hopf algebra on weak order intervals.
\item The construction below relies on the local conditions characterizing~$\WOIP$ in Proposition~\ref{prop:characterizationWOIP}. As shown in Proposition~\ref{prop:characterizationIWOIPDWOIP}, these conditions can be split to characterize separately~$\IWOIP$ and~$\DWOIP$. Therefore, we obtain similar Hopf algebra structures on~$\IWOIP$ and~$\DWOIP$, although we do not explicitly state all results for~$\IWOIP$ and~$\DWOIP$.
\end{itemize}

\begin{proposition}
\label{prop:algebraWOIP}
For any~${\less}, {\bless} \in \IPos$,
\begin{enumerate}[(i)]
\item if~$({\less} \shiftedShuffle {\bless}) \cap \WOIP \ne \varnothing$, then~${\less} \in \WOIP$ and~${\bless} \in \WOIP$,
\item if~${\less} \in \WOIP$ and~${\bless} \in \WOIP$, then~$({\less} \convolution {\bless}) \subseteq \WOIP$.
\end{enumerate}
Therefore, the vector subspace of~$\K\IPos$ generated by~$\IPos \ssm \WOIP$ is a Hopf ideal of~$(\K\IPos, \product, \coproduct)$.
The quotient of the integer poset algebra~$(\K\IPos, \product, \coproduct)$ by this ideal is thus a Hopf algebra~$(\K\WOIP^\quotient, \product, \coproduct)$ on weak order intervals.
A similar statement holds for~$\IWOIP$ and~$\DWOIP$.
\end{proposition}

\begin{proof}
We make the proof for~$\IWOIP$, the proof for~$\DWOIP$ is symmetric, and the result follows for~$\WOIP = \IWOIP \cap \DWOIP$.
For~(i), let~${\less} \in \IPos_m$ and~${\bless} \in \IPos_n$ be such that the shifted shuffle~${\less} \shiftedShuffle {\bless}$ contains a poset~${\dashv} \in \IWOIP_{m+n}$.
Let~$1 \le a < b < c \le m$ be such that~$a \less c$.
Then~$a \dashv c$ (since~${\dashv}_{[m]} = {\less}$), which ensures that~$a \dashv b$ or~$b \dashv c$ (since~${\dashv} \in \IWOIP_{m+n}$), and we obtain that~$a \less b$ or~$b \less c$ (again since~${\dashv}_{[m]} = {\less}$).
We conclude that~${\less} \in \IWOIP_m$ by the characterization of Proposition~\ref{prop:characterizationIWOIPDWOIP}, and we prove similarly that~${\bless} \in \IWOIP_n$.

For~(ii), consider two weak order interval posets~${\less}, {\bless} \in \IWOIP$ and let~${\dashv} \in {\less} \convolution {\bless}$.
Let~$(X,Y)$ be the total cut of~$\dashv$ such that~${\dashv}_X = {\less}$ and~${\dashv}_Y = {\bless}$.
Consider~$a < b < c$ such that~$a \dashv c$.
We distinguish three situations according to the repartition of~$\{a,b,c\}$ in the partition~$X \sqcup Y$:
\begin{itemize}
\item Assume first that~$\{a, b, c\}  \subseteq X$. Since~${\dashv}_X = {\less}$ is in~$\IWOIP$, we obtain that~$a \dashv b$ or~$b \dashv c$. The argument is identical if~$\{a, b, c\}  \subseteq Y$.
\item Assume now that~$\{a, c\} \subseteq X$ and~$b \in Y$. Then~$a \dashv b$ since~$X \times Y \subseteq {\dashv}$. The argument is identical if~$\{a,c\} \in Y$ and~$b \in X$.
\item Finally, assume that~$\{a,c\} \not\subseteq X$ and~$\{a,c\} \not\subseteq Y$. Then we have~$a \in X$ and~$c \in Y$ (since~$a \dashv c$ and~$(Y \times X) \cap {\dashv} = \varnothing$). Since~$X \times Y \subseteq {\dashv}$, we obtain that~$a \dashv b$ if~$b \in Y$, while~$b \dashv c$ if~$b \in X$.
\end{itemize}
We therefore obtain that~$a \dashv c \implies (a \dashv b \text{ or } b \dashv c)$ and we conclude that~${\dashv} \in \IWOIP$ by the characterization of Proposition~\ref{prop:characterizationIWOIPDWOIP}.
\end{proof}

\begin{remark}
\label{rem:convolutionWOIP}
Although not needed for the Hopf algebra quotient, observe that the convolution satisfies a property similar to Proposition~\ref{prop:algebraWOIP}\,(i):
if~${\less}, {\bless} \in \IPos$ are such that the convolution~${\less} \convolution {\bless}$ contains at least one element in $\WOIP$, then~${\less}$ and~${\bless}$ are both in $\WOIP$.
\end{remark}

For any weak order interval poset~$\less$, we denote by~$\FWOIPquo_{\less}$ the image of~$\FRel_{\less}$ through the trivial projection~$\K\IRel \to \K\IPos$.

\begin{example}
In practice, for any two weak order interval posets~${\less}, {\bless} \in \WOIP$, we compute the product $\FWOIPquo_{\less} \product \FWOIPquo_{\bless}$ in~$\K\WOIP^\quotient$ by deleting all summands not in~$\WOIP$ in the product~${\FRel_{\less} \product \FRel_{\bless}}$ in~$\K\IRel$:
\input{productWOIP}
The coproduct is even simpler: all relations that appear in the coproduct~$\coproduct(\FRel_{\less})$ of an element in  $\WOIP$ are automatically in $\WOIP$ by Remark~\ref{rem:convolutionWOIP}:
\input{coproductWOIP}
\end{example}

\begin{proposition}
For any~${\less} \in \WOIP_m$ and~${\bless} \in \WOIP_n$, the product~$\FWOIPquo_{\less} \product \FWOIPquo_{\bless}$ is the sum of~$\FWOIPquo_{\dashv}$, where~$\dashv$ runs over the interval between $\underprod{\less}{\bless}$ and $\overprod{\less}{\bless}$ in the weak order on~${\WOIP_{m+n}}$.
\end{proposition}

\begin{proof}
It is a direct consequence of Proposition~\ref{prop:characterizationShuffleProduct} and the fact that for any two weak order interval posets~${{\less}, {\bless} \in \WOIP}$, the relations~$\underprod{\less}{\bless}$ and~$\overprod{\less}{\bless}$ are both weak order interval posets.
Indeed, for any~$\sigma \wole \sigma'$ and~$\tau \wole \tau'$, we have~$\underprod{\less_{[\sigma,\sigma']}}{\less_{[\tau,\tau']}} = \less_{[\underprod{\sigma}{\tau}, \underprod{\sigma'}{\tau'}]}$ and $\overprod{\less_{[\sigma,\sigma']}}{\less_{[\tau,\tau']}} = \less_{[\overprod{\sigma}{\tau}, \overprod{\sigma'}{\tau'}]}$.
\end{proof}

\subsubsection{Faces}
\label{subsubsec:orderedPartitionsQuotientAlgebra}

We now construct a Hopf algebra on faces of the permutahedra as a Hopf subalgebra of the poset Hopf algebra.
We will see in Proposition~\ref{prop:isomorphismChapoton} that the resulting Hopf algebra was already considered by F.~Chapoton in~\cite{Chapoton}.

\begin{proposition}
\label{prop:algebraWOFP}
For any~${\less}, {\bless} \in \IPos$,
\begin{enumerate}[(i)]
\item if~$({\less} \shiftedShuffle {\bless}) \cap \WOFP \ne \varnothing$, then~${\less} \in \WOFP$ and~${\bless} \in \WOFP$,
\item if~${\less} \in \WOFP$ and~${\bless} \in \WOFP$, then~$({\less} \convolution {\bless}) \subseteq \WOFP$.
\end{enumerate}
Therefore, the vector subspace of~$\K\IPos$ generated by~$\IPos \ssm \WOFP$ is a Hopf ideal of~$(\K\IPos, \product, \coproduct)$.
The quotient of the poset Hopf algebra~$(\K\IPos, \product, \coproduct)$ by this ideal is thus a Hopf algebra~$(\K\WOFP^\quotient, \product, \coproduct)$ on faces of the permutahedron.
\end{proposition}

\begin{proof}
For~(i), let~${\less} \in \IPos_m$ and~${\bless} \in \IPos_n$ be such that the shifted shuffle~${\less} \shiftedShuffle {\bless}$ contains a weak order face poset~${\dashv} \in \WOFP_{m+n}$.
Since~$\WOFP \subset \WOIP$, Proposition~\ref{prop:algebraWOIP}\,(i) ensures that~${\less} \in \WOIP_m$ and~${\bless} \in \WOIP_n$.
Consider now~$1 \le a < b < c \le m$ such that~$a \not\less c$ and~$a \not\more c$.
Then~$a \not\dashv c$ and~$a \not\vdash c$ (since~${\dashv}_{[m]} = {\less}$), which ensures that~$a \dashv b \iff b \vdash c$ and~$a \vdash b \iff b \dashv c$ (since~${\dashv} \in \WOFP_{m+n}$), and we obtain that~$a \less b \iff b \more c$ and~$a \more b \iff b \less c$ (again since~${\dashv}_{[m]} = {\less}$).
We conclude that~${\less} \in \WOFP_m$ by the characterization of Proposition~\ref{prop:characterizationWOFP}, and we prove similarly that~${\bless} \in \WOFP_n$.

For~(ii), consider two weak order face posets~${\less}, {\bless} \in \WOFP$ and let~${\dashv} \in {\less} \convolution {\bless}$.
Since we have~$\WOFP \subset \WOIP$, Proposiiton~\ref{prop:algebraWOIP}\,(ii) ensures that~${\dashv} \in \WOIP$.
Let~$(X,Y)$ be the total cut of~$\dashv$ such that~${\dashv}_X = {\less}$ and~${\dashv}_Y = {\bless}$.
Consider~$a < b < c$ such that~$a \not\dashv c$ and~$a \not\vdash c$.
Since~$X \times Y \subseteq {\dashv}$, we obtain that~$\{a,c\} \subseteq X$ or~$\{a,c\} \subseteq Y$.
Since~${\dashv}_X = {\less}$ and~${\dashv}_Y = {\bless}$ are in~$\WOFP$, we obtain that~$a \dashv b \iff b \vdash c$ and~$a \vdash b \iff b \dashv c$.
We conclude that~${\dashv} \in \WOFP$ by the characterization of Proposition~\ref{prop:characterizationWOFP}.
\end{proof}

\begin{remark}
\label{rem:convolutionWOFP}
Although not needed for the Hopf algebra quotient, observe that the convolution satisfies a property similar to Proposition~\ref{prop:algebraWOFP}\,(i):
if~${\less}, {\bless} \in \IPos$ are such that the convolution~${\less} \convolution {\bless}$ contains at least one element in $\WOFP$, then~${\less}$ and~${\bless}$ are both in $\WOFP$.
\end{remark}

For any weak order face poset~$\less$, we denote by~$\FWOFPquo_{\less}$ the image of~$\FRel_{\less}$ through the trivial projection~$\K\IRel \to \K\IPos$.

\begin{example}
\label{exm:algebraWOFP}
In practice, for two weak order face posets~${\less}, {\bless} \in \WOFP$, we compute the product~$\FWOFPquo_{\less} \product \FWOFPquo_{\bless}$ in~$\K\WOFP^\quotient$ by deleting all summands not in~$\WOFP$ in the product~$\FRel_{\less} \product \FRel_{\bless}$ in~$\K\IRel$:
\input{productWOFP}
The coproduct is even simpler: all relations that appear in the coproduct~$\coproduct(\FRel_{\less})$ of an element in  $\WOFP$ are automatically in $\WOFP$ by Remark~\ref{rem:convolutionWOFP}:
\input{coproductWOFP}
\end{example}

It turns out that the resulting algebra was studied in~\cite{Chapoton}.
Consider an ordered partition~$\pi$ of~$[m]$ into~$k$ parts, and an ordered partition~$\rho$ of~$[n]$ into~$\ell$ parts.
As for permutations (see Section~\ref{subsec:MalvenutoReutenauer}), the \defn{shifted shuffle}~$\pi \shiftedShuffle \rho$ (resp.~the \defn{convolution}~$\pi \convolution \rho$) is the set of ordered partitions whose first~$k$ values (resp.~first $m$ positions) are in the same relative order as~$\pi$ and whose last~$\ell$ values (resp.~last~$n$ positions) are in the same relative order as~$\rho$.
Here, relative order means in an earlier block, in the same block, or in a later block.
Note that the shifted shuffle may merge blocks of~$\pi$ with blocks of~$\rho$: all ordered partitions in~$\pi \shiftedShuffle \rho$ have at least~$\min(k,\ell)$ blocks and at most~$k+\ell$ blocks.
In contrast, the convolution just adds up the numbers of blocks.
For example,
\begin{align*}
{\red 1} | {\red 2} \shiftedShuffle {\blue 2} | {\blue 31} = \{ &
    {\red 1} | {\red 2} | {\blue 4} | {\blue 53}, \;
    {\red 1} | {\red 2}{\blue 4} | {\blue 53}, \;
    {\red 1} | {\blue 4} | {\red 2} | {\blue 53}, \;
    {\red 1} | {\blue 4} | {\red 2}{\blue 53}, \;
    {\red 1} | {\blue 4} | {\blue 53} | {\red 2}, \;
    {\red 1}{\blue 4} | {\red 2} | {\blue 53}, \;
    {\red 1}{\blue 4} | {\red 2}{\blue 53}, \\[-.1cm]
&   {\red 1}{\blue 4} | {\blue 53} | {\red 2}, \;
    {\blue 4} | {\red 1} | {\red 2} | {\blue 53}, \;
    {\blue 4} | {\red 1} | {\red 2}{\blue 53}, \;
    {\blue 4} | {\red 1} | {\blue 53} | {\red 2}, \;
    {\blue 4} | {\red 1}{\blue 53} | {\red 2}, \;
    {\blue 4} | {\blue 53} | {\red 1} | {\red 2} \}, \\[.2cm]
{\red 1} | {\red 2} \convolution {\blue 2} | {\blue 31} = \{ &
    {\red 1} | {\red 2} | {\blue 4} | {\blue 53}, \;
    {\red 1} | {\red 3} | {\blue 4} | {\blue 52}, \;
    {\red 1} | {\red 4} | {\blue 3} | {\blue 52}, \;
    {\red 1} | {\red 5} | {\blue 3} | {\blue 42}, \;
    {\red 2} | {\red 3} | {\blue 4} | {\blue 51}, \\[-.1cm]
&   {\red 2} | {\red 4} | {\blue 3} | {\blue 51}, \;
    {\red 2} | {\red 5} | {\blue 3} | {\blue 41}, \;
    {\red 3} | {\red 4} | {\blue 2} | {\blue 51}, \;
    {\red 3} | {\red 5} | {\blue 2} | {\blue 41}, \;
    {\red 4} | {\red 5} | {\blue 2} | {\blue 31} \}.
\end{align*}
\noindent
The Chapoton Hopf algebra~\cite{Chapoton} is the Hopf algebra on ordered partitions with product~$\product$ and coproduct~$\coproduct$ defined by
\[
\F_\sigma \product \F_\tau \eqdef \sum_{\rho \in \sigma \shiftedShuffle \tau} \F_\rho
\qquad\text{and}\qquad
\coproduct(\F_\rho) \eqdef \sum_{\rho \in \sigma \convolution \tau} \F_\sigma \otimes \F_\tau.
\]
We refer to~\cite{Chapoton, ChatelPilaud} for more details and just provide an example of product and coproduct in this Hopf algebra.

\begin{example}
\label{exm:algebraChapotonOrderedPartitions}
For example, we have
\[
\F_{12} \product \F_1 = \F_{12|3} + \F_{123} + \F_{3|12}
\qquad\text{and}\qquad
\coproduct (\F_{13|2}) = \F_{13|2} \otimes \F_{\varnothing} + \F_{12} \otimes \F_{1} + \F_{\varnothing} \otimes \F_{13|2}.
\]
\end{example}

\begin{proposition}
\label{prop:isomorphismChapoton}
The map~$\F_\pi \mapsto \FWOFPquo_{\less_\pi}$ defines a Hopf algebra isomorphism from the Chapoton Hopf algebra on ordered partitions~\cite{Chapoton} to~$(\K\WOFP^\quotient, \product, \coproduct)$.
\end{proposition}

\begin{proof}
The proof is similar to that of Proposition~\ref{prop:isomorphismMalvenutoReutenauer} and left to the reader.
\end{proof}

\begin{example}
Compare Examples~\ref{exm:algebraChapotonOrderedPartitions} and~\ref{exm:algebraWOFP}.
\end{example}

\begin{proposition}
For any~${\less} \in \WOFP_m$ and~${\bless} \in \WOFP_n$, the product~$\FWOFPquo_{\less} \product \FWOFPquo_{\bless}$ is the sum of~$\FWOFPquo_{\dashv}$, where~$\dashv$ runs over the interval between $\underprod{\less}{\bless}$ and $\overprod{\less}{\bless}$ in the weak order on~${\WOFP_{m+n}}$.
\end{proposition}

\begin{proof}
It is a direct consequence of Proposition~\ref{prop:characterizationShuffleProduct} and the fact that for any two weak order face posets~${{\less}, {\bless} \in \WOFP}$, the relations~$\underprod{\less}{\bless}$ and~$\overprod{\less}{\bless}$ are both weak order face posets.
\end{proof}

\subsection{Subalgebras}
\label{subsec:subalgebrasWO}

We now construct Hopf algebras on~$\WOEP$ and~$\WOIP$ as subalgebras of the integer poset algebra.
For this, we use surjections from all posets to $\WOEP$ or~$\WOIP$ whose fibers are stable by product and coproduct.
We consider the \defn{$\IWOIP$ increasing deletion}, the \defn{$\DWOIP$ decreasing deletion}, and the \defn{$\WOIP$ deletion} defined in~\cite[Sect.~2.1.4]{ChatelPilaudPons} by
\begin{align*}
{\IWOIPid{\less}} & = {\less} \ssm \bigset{(a,c)}{\exists \; a < b_1 < \dots < b_k < c, \; a \not\less b_1 \not\less \dots \not\less b_k \not\less c}, \\
{\DWOIPdd{\less}} & = {\less} \ssm \bigset{(c,a)}{\exists \; a < b_1 < \dots < b_k < c, \; a \not\more b_1 \not\more \dots \not\more b_k \not\more c}, \\
{\WOIPd{\less}} & = {\IWOIPid{\big( \DWOIPdd{\less} \big)}} = {\DWOIPdd{ \big( \IWOIPid{\less} \big)}}.
\end{align*}
These operations are illustrated on \fref{fig:IWOIPid/DWOIPdd}.

\begin{figure}[h]
    \vspace{-1cm}
    \input{example_iwoipid_dwoipdd}
    \vspace{-.6cm}
    \caption{The $\IWOIP$ increasing deletion and the $\DWOIP$ decreasing deletion.}
    \label{fig:IWOIPid/DWOIPdd}
\end{figure}
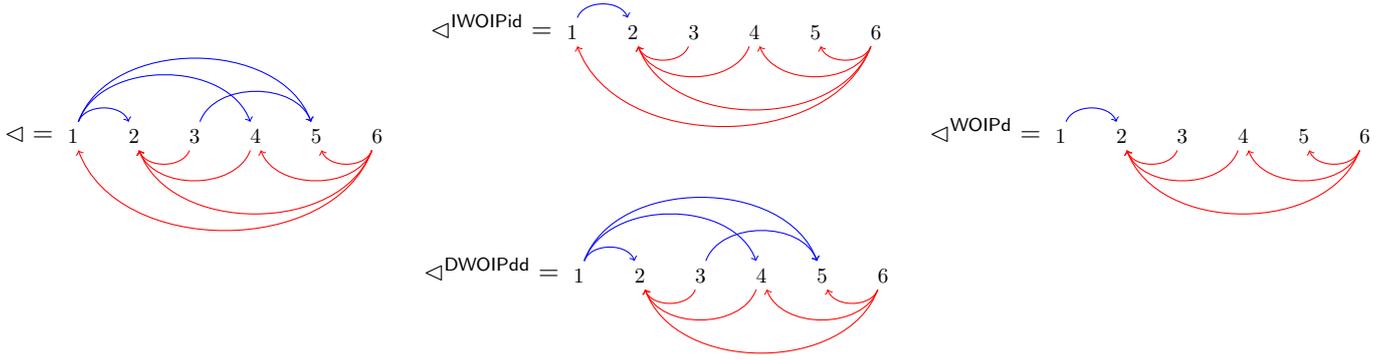

It is proved in~\cite[Sect.~2.1.4]{ChatelPilaudPons} that~${\IWOIPid{\less}} \in \IWOIP$ and~${\DWOIPdd{\less}} \in \DWOIP$ for any~${{\less} \in \IPos}$.
In fact, ${\less} \mapsto {\IWOIPid{\less}}$ (resp.~${\less} \mapsto {\DWOIPdd{\less}}$) is a projection from~$\IPos$ to~$\IWOIP$ (resp.~$\DWOIP$).
Therefore, ${\WOIPd{\less}} \in \WOIP$ for any~${{\less} \in \IPos}$ and~${\less} \mapsto {\WOIPd{\less}}$ is a projection from~$\IPos$ to~$\WOIP$.

\subsubsection{Intervals}
\label{subsubsec:intervalsSubalgebra}

Here, we use the fibers of the projections~${\less} \mapsto {\IWOIPid{\less}}$, ${\less} \mapsto {\DWOIPdd{\less}}$ and ${\less} \mapsto {\WOIPd{\less}}$ to construct Hopf subalgebras of the poset Hopf algebra.
For this, we need the compatibility of these projections with the shuffle and the convolution on posets.

\begin{lemma}
\label{lem:WOIPdShuffle}
For any~${\less} \in \IPos_p$ and any~$1 \le q \le r \le p$, we have
\[
\WOIPd{\big( {\less}_{[q,r]} \big)} = \big( {\WOIPd{\less}} \big)_{[q,r]}.
\]
Therefore, for any~${\bless} \in \WOIP_m$ and~${\bless'} \in \WOIP_n$,
\[
\bigset{{\less} \in \IPos_{m+n}}{{\WOIPd{(\less_{[m]})}} = {\bless} \text{ and } {\WOIPd{(\less_{\overline{[n]}})}} = {\bless'}} = \bigsqcup_{\substack{{\prec} \in {\bless} \shiftedShuffle {\bless'} \\ \quad \cap \WOIP}} \set{{\less} \in \IPos_{m+n}}{\WOIPd{\less} = {\prec}}.
\]
A similar statement holds for~${\less} \mapsto {\IWOIPid{\less}}$ and for~${\less} \mapsto {\DWOIPdd{\less}}$.
\end{lemma}

\begin{proof}
We make the proof for~$\IWOIP$, the proof for~$\DWOIP$ is symmetric, and the result follows for~$\WOIP = \IWOIP \cap \DWOIP$.
The first statement immediately follows from the fact that the condition to delete~$(a,c)$ in~$\IWOIPid{\less}$ only involves relations of~$\less$ in the interval~$[a,c]$.
By Proposition~\ref{prop:characterizationShuffleProduct}, $({\bless} \shiftedShuffle {\bless'}) \cap \IWOIP$ is the set of~${\prec} \in \IWOIP_{m+n}$ such that~${\prec_{[m]}} = {\bless}$ and~${\prec_{\overline{[n]}}} = {\bless'}$, which shows the second statement.
\end{proof}

\begin{lemma}
\label{lem:WOIPdConvolution}
For any~${\less} \in \IPos_p$, the total cuts of~$\less$ are precisely the total cuts of~$\WOIPd{\less}$. Moreover, if~$(X,Y)$ is a total cut of~$\less$, then
\[
(\WOIPd{\less})_X = \WOIPd{(\less_X)}
\qquad\text{and}\qquad
(\WOIPd{\less})_Y = \WOIPd{(\less_Y)}.
\]
Therefore, for any~${\bless} \in \WOIP_p$ with a total cut~$(X,Y)$,
\[
\set{{\less} \in \IPos_p}{{\WOIPd{\less}} = {\bless}} = \set{{\less} \in \IPos_p}{\begin{array}{ll} \WOIPd{({\less}_X)} = {\bless_X} & {\less} \cap (Y \times X) = \varnothing \\ \WOIPd{({\less}_Y)} = {\bless_Y} & \text{and } (X \times Y) \subseteq {\less}\end{array}}.
\]
A similar statement holds for~${\less} \mapsto {\IWOIPid{\less}}$ and for~${\less} \mapsto {\DWOIPdd{\less}}$.
\end{lemma}

\begin{proof}
We make the proof for~$\IWOIP$, the proof for~$\DWOIP$ is symmetric, and the result follows for~$\WOIP = \IWOIP \cap \DWOIP$.
Consider a partition~$[p] = X \sqcup Y$.
Assume that~$(X,Y)$ is a total cut of~$\less$ and consider~$x \in X$ and~$y \in Y$.
If~$y < x$, then~$x \IWOIPid{\less} y$ since~$(x,y) \in \Dec{\less} = \Dec{(\IWOIPid{\less})}$.
If~$x < y$, then~$x \IWOIPid{\less} y$ since~$x \less y$ and for any~$x = b_0 < b_1 < \dots < b_k < b_{k+1} = y$, we have~$b_\ell \less b_{\ell+1}$, where~$\ell$ is maximal such that~$b_\ell \in X$.
Finally,  $y \not\IWOIPid{\less} x$ since~$y \not\less x$ and~${\IWOIPid{\less}} \subseteq {\less}$.
We conclude that~$(X,Y)$ is a total cut of~$\IWOIPid{\less}$.
The reverse inclusion is similar.

We now consider a total cut~$(X,Y)$ of~$\less$ and prove that~$(\IWOIPid{\less})_X = \IWOIPid{(\less_X)}$ (the other equality is similar).
Observe first that~$\Dec{((\IWOIPid{\less})_X)} = \Dec{(\IWOIPid{(\less_X)})} = {\Dec{\less_X}}$ so we focus on increasing relations.
Let~$x_1 < \dots < x_m$ be the elements of~$X$, and consider~$1 \le i < j \le m$.
If~$(i,j) \in (\IWOIPid{\less})_X$, then~$x_i \less x_j$ and there is no $x_i < b_1 < \dots < b_k < x_j$ such that~$x_i \not\less b_1 \not\less \dots \not\less b_k \not\less x_j$.
In particular, there is no such~$b_1, \dots, b_k$ in~$X$ and we obtain that~$(i,j) \in \IWOIPid{(\less_X)}$.
Conversely, assume that~$(i,j) \in \IWOIPid{(\less_X)}$ and consider~$x_i = b_0 < b_1 < \dots < b_k < b_{k+1} = x_j$.
We distinguish two cases:
\begin{itemize}
\item If~$b_1 = x_{i_1}, \dots, b_k = x_{i_k}$ all belong to~$X$. Since~$(i,j) \in \IWOIPid{(\less_X)}$, there is~$\ell$ such that~$i_\ell \less_X i_{\ell+1}$. This implies that~$b_\ell \less b_{\ell+1}$.
\item Otherwise, consider the last~$\ell$ such that~$b_\ell \in X$. Then~$b_{\ell+1} \in Y$ and we have~$b_\ell \less b_{\ell+1}$.
\end{itemize}
In both cases, we have obtained that there is $\ell$ such that~$b_\ell \less b_{\ell+1}$.
We therefore obtain that~$(x_i, x_j) \in \IWOIPid{\less}$, so that~$(i,j) \in \IWOIPid{(\less_X)}$.

Finally, the last statement is just a reformulation of the former.
\end{proof}

For~${\bless} \in \WOIP$, consider the element
\[
\FWOIPsub_\bless \eqdef \sum \FPos_\less.
\]
where the sum runs over all ${\less} \in \IPos$ such that~$\WOIPd{\less} = {\bless}$.
We denote by~$\K\WOIP^\subalg$ the linear subspace of~$\K\IPos$ spanned by the elements~$\FWOIPsub_\bless$ for~${\bless} \in \WOIP$.
We define similarly subspaces spanned by the fibers of the~$\IWOIP$ and the~$\DWOIP$ deletions.

\begin{proposition}
\label{prop:WOIPsubHopfAlgebra}
The subspace~$\K\WOIP^\subalg$ is stable by the product~$\product$ and the coproduct~$\coproduct$ and thus defines a Hopf subalgebra of~$(\K\IPos, \product, \coproduct)$.
A similar statement holds for~$\IWOIP$ and~$\DWOIP$.
\end{proposition}

\begin{proof}
The proof is a computation relying on Lemmas~\ref{lem:WOIPdShuffle} and~\ref{lem:WOIPdConvolution}. We therefore make the proof here for~$\WOIP$, the same proof works verbatim for~$\IWOIP$ or~$\DWOIP$.
We first show the stability by product. For~${\bless} \in \WOIP_m$ and~${\bless'} \in \WOIP_n$, we have by Lemma~\ref{lem:WOIPdShuffle}
\begin{align*}
\FWOIPsub_{\bless} \product \FWOIPsub_{\bless'}
& = \bigg( \sum_{\WOIPd{\dashv} = \bless} \!\!\! \FPos_\dashv \bigg) \product \bigg( \sum_{\WOIPd{\dashv'} = \bless'} \!\!\! \FPos_{\dashv'} \bigg)
= \!\!\! \sum_{\substack{\WOIPd{\dashv} = \bless \\ \WOIPd{\dashv'} = \bless'}} \!\!\! \FPos_\dashv \product \FPos_{\dashv'} \\
& = \!\!\! \sum_{\substack{\WOIPd{\dashv} = \bless \\ \WOIPd{\dashv'} = \bless'}} \sum_{\substack{{\less} \in \IPos_{m+n} \\ {\less}_{[m]} = {\dashv} \\ \less_{\overline{[n]}} = {\dashv'}}} \!\!\! \FPos_\less
= \!\!\! \sum_{\substack{{\less} \in \IPos_{m+n} \\ \WOIPd{(\less_{[m]})} = {\bless} \\ \WOIPd{(\less_{\overline{[n]}})} = {\bless'}}} \!\!\! \FPos_\less
= \!\!\! \sum_{\substack{{\prec} \in {\bless} \shiftedShuffle {\bless'} \\ \quad \cap \WOIP}} \sum_{\WOIPd{\less} = {\prec}} \!\!\! \FPos_\less \\
& = \!\!\! \sum_{\substack{{\prec} \in {\bless} \shiftedShuffle {\bless'} \\ \quad \cap \WOIP}} \!\!\! \FWOIPsub_{\prec}.
\end{align*}
We now show the stability by coproduct.
For~${\bless} \in \WOIP_p$, we have by Lemma~\ref{lem:WOIPdConvolution}
\begin{align*}
\coproduct(\FWOIPsub_\bless)
& = \coproduct \bigg( \sum_{\WOIPd{\less} = \bless} \FPos_\less \bigg)
= \!\!\! \sum_{\WOIPd{\less} = \bless} \!\!\! \coproduct(\FPos_\less)
= \!\!\! \sum_{\WOIPd{\less} = \bless} \;\; \sum_{\substack{(X,Y) \text{ total} \\ \text{cut of } \less}} \!\!\! \FPos_{\less_X} \otimes \FPos_{\less_Y} \\
& = \!\!\! \sum_{\substack{(X,Y) \text{ total} \\ \text{cut of } \bless}} \;\; \sum_{\WOIPd{\less} = \bless} \!\!\! \FPos_{\less_X} \otimes \FPos_{\less_Y}
= \!\!\! \sum_{\substack{(X,Y) \text{ total} \\ \text{cut of } \bless}} \bigg( \sum_{\WOIPd{\dashv} = {\bless}_X} \!\!\! \FPos_{\dashv} \bigg) \otimes \bigg( \sum_{\WOIPd{\dashv'} = {\bless}_Y} \!\!\! \FPos_{\dashv'} \bigg) \\
& = \!\!\! \sum_{\substack{(X,Y) \text{ total} \\ \text{cut of } \bless}} \!\!\! \FWOIPsub_{\bless_X} \otimes \FWOIPsub_{\bless_Y}.
\qedhere
\end{align*}
\end{proof}

\begin{proposition}
The map~$\FWOIPquo_\bless \mapsto \FWOIPsub_\bless$ defines a Hopf algebra isomorphism from the quotient Hopf algebra~$\K\WOIP^\quotient$ to the Hopf subalgebra~$\K\WOIP^\subalg$.
A similar statement holds for~$\IWOIP$ and~$\DWOIP$.
\end{proposition}

\begin{proof}
This is immediate since the formulas for the product and coproduct on the bases~$(\FWOIPquo_\bless)$ and~$(\FWOIPsub_\bless)$ coincide:
\[
\FWOIP_{\bless} \product \FWOIP_{\bless'} = \sum_{\substack{{\prec} \in {\bless} \shiftedShuffle {\bless'} \\ \quad \cap \WOIP}} \FWOIP_{\prec}
\qquad\text{and}\qquad
\coproduct(\FWOIP_\bless) = \sum_{\substack{(X,Y) \text{ total} \\ \text{cut of } \bless}} \FWOIP_{\bless_X} \otimes \FWOIP_{\bless_Y}.
\qedhere
\]
\end{proof}

\subsubsection{Elements}
\label{subsubsec:permutationsSubalgebra}

\enlargethispage{-.4cm}
We now construct a Hopf algebra on~$\WOEP$ as a subalgebra of the integer poset algebra.
To this end, consider the two maps from~$\IPos$ to~$\WOEP$ defined by
\[
\WOEPid{\less} \eqdef \maxle{(\IWOIPid{\less})}
\qquad\text{and}\qquad
\WOEPdd{\less} \eqdef \minle{(\DWOIPdd{\less})}.
\]
We need to describe the fibers of these maps.

\begin{lemma}
\label{lem:fibersWOEPd}
For~${\less} \in \IPos_n$ and~${\bless} \in \WOEP_n$, we have
\begin{itemize}
\item ${\WOEPid{\less}} = {\bless}$ $\iff$ for all~$(a,c) \in \Inc{({\less} \ssm {\bless})}$ there exists~$a < b < c$ such that~$a \bmore b \bmore c$,
\item ${\WOEPdd{\less}} = {\bless}$ $\iff$ for all~$(c,a) \in \Dec{({\less} \ssm {\bless})}$ there exists~$a < b < c$ such that~$a \bless b \bless c$.
\end{itemize}
\end{lemma}

\begin{proof}
We only prove the first statement, the second is symmetric.
Assume~${\WOEPid{\less}} = {\bless}$ and~$(a,c) \in \Inc{({\less} \ssm {\bless})}$.
Since~$(a,c)$ is deleted in~${\WOEPid{\less}} = \maxle{(\IWOIPid{\less})}$, it is already deleted in~$\IWOIPid{\less}$.
Therefore, there exists~$a < b_1 < \dots < b_k < c$ such that~$a \not\less b_1 \not\less \dots \not\less b_k \not\less c$.
By definition of~${\bless} = \maxle{(\IWOIPid{\less})}$, this implies that~$a \bmore b_1 \bmore \dots \bmore b_k \bmore c$.

Conversely, assume that for all~$(a,c) \in \Inc{({\less} \ssm {\bless})}$ there exists~$a < b < c$ such that~$a \bmore b \bmore c$.
Assume by contradiction that~${\WOEPid{\less}} \ne {\bless}$.
By definition of~${\WOEPid{\less}} = \maxle{(\IWOIPid{\less})}$, this implies that there exists~$(a,c) \in \Inc{({\IWOIPid{\less}} \ssm {\bless})}$.
Choose such an~$(a,c)$ with~$c-a$ minimal.
Since~$(a,c) \in \Inc{({\IWOIPid{\less}} \ssm {\bless})} \subseteq \Inc{({\less} \ssm {\bless})}$, there exists~$a < b < c$ such that~$a \bmore b \bmore c$.
Since~$a \IWOIPid{\less} c$, we have~$a \IWOIPid{\less} b$ or~$b \IWOIPid{\less} c$, thus either~$(a,b)$ or~$(b,c)$ belongs to~$\Inc{({\IWOIPid{\less}} \ssm {\bless})}$ contradicting the minimality of~$c-a$.
\end{proof}

We now study the compatibility of the projections~${\less} \mapsto {\WOEPid{\less}}$ and~${\less} \mapsto {\WOEPdd{\less}}$ with poset restrictions and cuts.

\begin{lemma}
\label{lem:WOEPdShuffle}
For any~${\less} \in \IPos_p$ and any~$1 \le q \le r \le p$, we have
\[
\WOEPid{\big( {\less}_{[q,r]} \big)} = \big( {\WOEPid{\less}} \big)_{[q,r]}
\qquad\text{and}\qquad
\WOEPdd{\big( {\less}_{[q,r]} \big)} = \big( {\WOEPdd{\less}} \big)_{[q,r]}.
\]
\end{lemma}

\begin{proof}
Same as Lemma~\ref{lem:WOIPdShuffle} since the presence of~$(a,c)$ in both~${\less} \mapsto {\IWOIPid{\less}}$ and~${\less} \mapsto {\maxle{\less}}$ only depends on the relations of~$\less$ in the interval~$[a,c]$.
\end{proof}

In contrast, Lemma~\ref{lem:WOIPdConvolution} does not apply verbatim here since~$\less$ has in general less total cuts than~$\WOEPid{\less}$ and~$\WOEPdd{\less}$.
We adapt this lemma as follows.

\begin{lemma}
\label{lem:WOEPdConvolution}
For any~${\less} \in \IPos_p$, all total cuts of~$\less$ are total cuts of both~${\WOEPid{\less}}$ and~${\WOEPdd{\less}}$.
Moreover, for any~${\less} \in \IPos_p$ and~${\bless} \in \WOEP_p$, and for any total cut~$(X,Y)$ of both~$\less$ and~$\bless$,
\begin{itemize}
\item ${\WOEPid{\less}} = {\bless}$ if and only if~${\WOEPid{(\less_X)}} = {\bless_X}$ and~${\WOEPid{(\less_Y)}} = {\bless_Y}$,
\item ${\WOEPdd{\less}} = {\bless}$ if and only if~${\WOEPdd{(\less_X)}} = {\bless_X}$ and~${\WOEPdd{(\less_Y)}} = {\bless_Y}$,
\end{itemize}
\end{lemma}

\begin{proof}
We only prove the statement for~${\WOEPid{\less}}$, the proof for~${\WOEPdd{\less}}$ is symmetric.
Consider a total cut~$(X,Y)$ of~$\less$ and let~$x \in X$ and~$y \in Y$.
If~$x > y$, then~$(x,y) \in {\Dec{\less}} \subseteq {\WOEPid{\less}}$.
If~$x < y$, then for any~$x < b < y$, we have either~$b \in X$ and thus~$b \less y$, or~$b \in Y$ and thus~$x \less b$.
Therefore, $(x,y) \in {\IWOIPid{\less}} \subseteq {\WOEPid{\less}}$.
We conclude that~$(X,Y)$ is also a total cut of~${\WOEPid{\less}}$.

Consider now~${\less} \in \IPos_p$ and~${\bless} \in \WOEP_p$, and a total cut~$(X,Y)$ of both~$\less$ and~$\bless$.
Assume first that~${\WOEPid{\less}} = {\bless}$.
Consider~$(a,c) \in \Inc{({\less_X} \ssm {\bless_X})}$.
Since~${\WOEPid{\less}} = {\bless}$ and~$(a,c) \in ({\less} \ssm {\bless})$, Lemma~\ref{lem:fibersWOEPd} ensures that there is~$a < b < c$ such that~$a \bmore b \bmore c$.
Since~$(X,Y)$ is a cut of~$\bless$, we obtain that~$b \in X$.
Therefore, $a \bmore_X b \bmore_X c$.
We conclude by Lemma~\ref{lem:fibersWOEPd} that~${\WOEPid{(\less_X)}} = {\bless_X}$.
The proof is identical for~${\WOEPid{(\less_Y)}} = {\bless_Y}$.

Conversely, assume that~${\WOEPid{(\less_X)}} = {\bless_X}$ and~${\WOEPid{(\less_Y)}} = {\bless_Y}$.
Consider~$(a,c) \in \Inc{({\less} \ssm {\bless})}$.
Since~$(X,Y)$ is a total cut of both~$\less$ and~$\bless$, we know that $a$ and $c$ either both belong to~$X$ or both belong to~$Y$, say for instance~$X$.
Since~${\WOEPid{(\less_X)}} = {\bless_X}$ we now that there exists~$a < b < c$ such that~$a \bmore_X b \bmore_X c$.
We conclude by Lemma~\ref{lem:fibersWOEPd} that~${\WOEPid{\less}} = {\bless}$.
\end{proof}

For~${\bless} \in \WOEP$, consider the elements
\[
\FWOEPsubi_\bless \eqdef \sum_{\substack{{\less} \in \IPos \\ \WOEPid{\less} = {\bless}}} \FPos_\less
\qquad\text{and}\qquad
\FWOEPsubd_\bless \eqdef \sum_{\substack{{\less} \in \IPos \\ \WOEPdd{\less} = {\bless}}} \FPos_\less.
\]
We denote by~$\K\WOEP^\subalgi$ and~$\K\WOEP^\subalgd$ the linear subspaces of~$\K\IPos$ spanned by the elements~$\FWOEPsubi_\bless$ and~$\FWOEPsubd_\bless$ respectively for~${\bless} \in \WOEP$.

\begin{proposition}
\label{prop:WOEPsubHopfAlgebra}
The subspaces~$\K\WOEP^\subalgi$ and~$\K\WOEP^\subalgd$ are stable by the product~$\product$ and the coproduct~$\coproduct$ and thus defines Hopf subalgebras of~$(\K\IPos, \product, \coproduct)$.
\end{proposition}

\begin{proof}
We only make the proof for~$\K\WOEP^\subalgi$, the statement for~$\K\WOEP^\subalgd$ is symmetric.
We first show the stability by product.
Using Lemma~\ref{lem:WOEPdShuffle}, and the exact same computation as in the first part of the proof of Proposition~\ref{prop:WOIPsubHopfAlgebra}, replacing~$\WOIP$ by~$\WOEP$, we obtain that
\[
\FWOEPsubi_{\bless} \product \FWOEPsubi_{\bless'}
=
\sum_{\substack{{\prec} \in {\bless} \shiftedShuffle {\bless'} \\ \quad \cap \WOEP}} \FWOEPsubi_{\prec}.
\]

We now show the stability by coproduct.
For~${\bless} \in \WOEP_p$, we have by Lemma~\ref{lem:WOEPdConvolution}
\begin{align*}
\coproduct(\FWOEPsubi_\bless)
& = \coproduct \bigg( \sum_{\WOEPid{\less} = \bless} \FPos_\less \bigg)
= \!\!\! \sum_{\WOEPid{\less} = \bless} \!\!\! \coproduct(\FPos_\less)
= \!\!\! \sum_{\WOEPid{\less} = \bless} \;\; \sum_{\substack{(X,Y) \text{ total} \\ \text{cut of } \less}} \!\!\! \FPos_{\less_X} \otimes \FPos_{\less_Y} \\
& = \!\!\! \sum_{\substack{(X,Y) \text{ total} \\ \text{cut of } \bless}} \;\; \sum_{\substack{\WOEPid{\less} = \bless \\ \text{with total} \\ \text{cut } (X,Y)}} \!\!\! \FPos_{\less_X} \otimes \FPos_{\less_Y}
= \!\!\! \sum_{\substack{(X,Y) \text{ total} \\ \text{cut of } \bless}} \bigg( \sum_{\WOEPid{\dashv} = {\bless}_X} \!\!\! \FPos_{\dashv} \bigg) \otimes \bigg( \sum_{\WOEPid{\dashv'} = {\bless}_Y} \!\!\! \FPos_{\dashv'} \bigg) \\
& = \!\!\! \sum_{\substack{(X,Y) \text{ total} \\ \text{cut of } \bless}} \!\!\! \FWOEPsubi_{\bless_X} \otimes \FWOEPsubi_{\bless_Y}.
\qedhere
\end{align*}
\end{proof}

\begin{proposition}
The map~$\FWOEPquo_\bless \mapsto \FWOEPsubi_\bless$ (resp.~$\FWOEPquo_\bless \mapsto \FWOEPsubd_\bless$) defines a Hopf algebra isomorphism from the quotient Hopf algebra~$\K\WOEP^\quotient$ to the Hopf subalgebra~$\K\WOEP^\subalgi$ (resp.~$\K\WOEP^\subalgd$). Therefore, $\K\WOEP^\subalgi$ and~$\K\WOEP^\subalgi$ are isomorphic to the Malvenuto--Reutenauer Hopf algebra on permutations~\cite{MalvenutoReutenauer}.
\end{proposition}

\begin{proof}
This is immediate since the formulas for the product and coproduct on the bases~$(\FWOEPquo_\bless)$, $(\FWOEPsubd_\bless)$ and~$(\FWOEPsubd_\bless)$ coincide:
\[
\FWOEP_{\bless} \product \FWOEP_{\bless'} = \sum_{\substack{{\prec} \in {\bless} \shiftedShuffle {\bless'} \\ \quad \cap \WOEP}} \FWOEP_{\prec}
\qquad\text{and}\qquad
\coproduct(\FWOEP_\bless) = \sum_{\substack{(X,Y) \text{ total} \\ \text{cut of } \bless}} \FWOEP_{\bless_X} \otimes \FWOEP_{\bless_Y}.
\qedhere
\]
\end{proof}

\section{Binary trees, Tamari intervals, and Schr\"oder trees}
\label{sec:TO}

We now consider three families of specific integer posets corresponding to the elements ($\TOEP$), the intervals ($\TOIP$) and the faces ($\TOFP$) in the Tamari order on binary trees.
We construct Hopf algebras on~$\TOEP$, $\TOIP$ and~$\TOFP$ as subalgebras (Section~\ref{subsec:subalgebrasWO}) of the integer poset Hopf algebra~$(\IPos, \product, \coproduct)$, using surjections from posets to $\TOEP$, $\TOIP$ of~$\TOFP$ whose fibers are stable by product and coproduct.
We obtain Hopf algebras on~$\TOEP$ (resp.~$\TOFP$) isomorphic to the Loday--Ronco Hopf algebra on binary trees~\cite{LodayRonco} (resp.~to the Chapoton Hopf algebra on Schr\"oder trees~\cite{Chapoton}), and we obtain a Hopf algebra on Tamari intervals that was not constructed earlier to the best of our knowledge.

\subsection{Binary trees and the Loday--Ronco algebra}
\label{subsec:LodayRonco}

We always label the vertices of a binary tree~$\tree$ in inorder, meaning that each vertex is labeled after all vertices of its left child and before all vertices of its right child.
This labeling makes~$\tree$ a binary search tree, meaning that the label of each vertex is larger than all labels in its left child and smaller than all labels in its right child.
For a permutation~$\sigma$, we denote by~$\bt(\sigma)$ the tree obtained by \defn{binary search tree insertion} of~$\sigma$: it is obtained by inserting the entries of~$\sigma$ from right to left such that each intermediate step remains a binary search tree.
Said differently, $\bt(\sigma)$ is the unique binary search tree~$\tree$ such that if vertex~$i$ is a descendant of vertex~$j$ in~$\tree$, then~$i$ appears before~$j$ in~$\sigma$.
An example is illustrated in \fref{fig:bstinsertion}.

\begin{figure}[h]
    \centerline{\includegraphics[scale=.7]{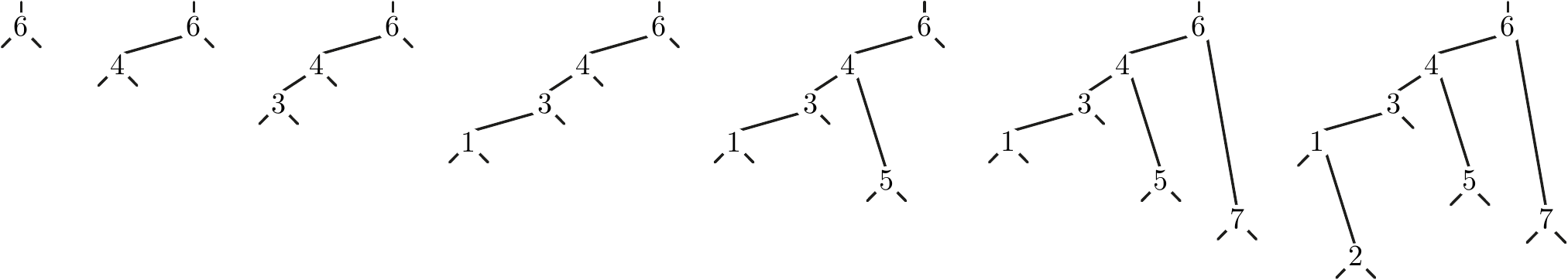}}
    \caption{Binary search tree insertion of the permutation~$\sigma = 2751346$.}
    \label{fig:bstinsertion}
\end{figure}

Recall that the classical \defn{Tamari lattice} is the lattice on binary trees whose cover relations are the right rotations.
It is also the lattice quotient of the weak order under the congruence relation given by the fibers of the binary search tree insertion.

Finally, recall that the Loday--Ronco algebra~\cite{LodayRonco} is the Hopf subalgebra of the Malvenuto--Reutenauer algebra generated by the elements
\[
\F_{\tree} \eqdef \sum_{\substack{\sigma \in \fS \\ \bt(\sigma) = \tree}} \F_\sigma
\]
for all binary trees~$\tree$.

\begin{example}
\label{exm:algebraLodayRonco}
For instance
\input{example_algebra_LR}
and
\input{example_cogebra_LR}
\end{example}

\subsection{Tamari order element, interval and face posets}
\label{subsec:TOEIFP}

\enlargethispage{.5cm}
We now briefly recall how the elements, the intervals and the faces of the Tamari lattice can be interpreted as specific interval posets as developed in~\cite{ChatelPilaudPons}.

\subsubsection{Elements}

We consider the tree~$\tree$ as the \defn{Tamari order element poset}~$\less_{\tree}$ defined by~${i \less_{\tree} j}$ when~$i$ is a descendant of~$j$ in~$\tree$. In other words, the Hasse diagram of~$\less_{\tree}$ is the tree~$\tree$ oriented towards its root. An illustration is provided in \fref{fig:toep}. Note that the increasing (resp.~decreasing) subposet of~$\less_{\tree}$ is given by~$i \Inc{\less_{\tree}} j$ (resp.~$i \Dec{\less_{\tree}} j$) if and only if $i$ belongs to the left (resp.~right) subtree of~$j$ in~$\tree$.

\begin{figure}[h]
    \input{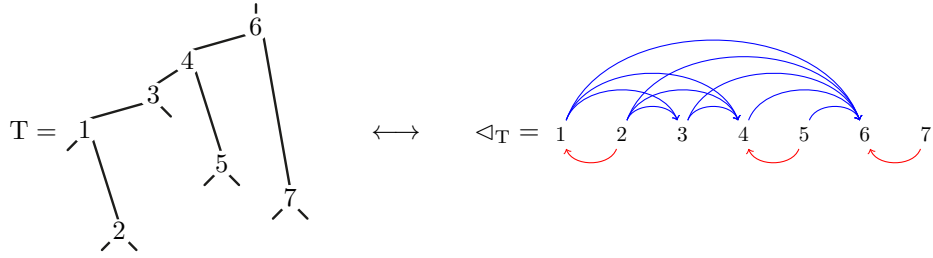}
    \vspace{-.3cm}
    \caption{A Tamari Order Element Poset ($\TOEP$).}
    \label{fig:toep}
    \vspace{-.3cm}
\end{figure}

We define
\[
\TOEP_n \eqdef \bigset{{\less_{\tree}}}{\tree \in \fB_n}
\qquad\text{and}\qquad
\TOEP \eqdef \bigsqcup_{n \in \N} \TOEP_n.
\]

The following statements provide a local characterization of the posets of~$\TOEP_n$ and describe the weak order induced by~$\TOEP_n$.

\begin{proposition}[{\cite[Prop.~39]{ChatelPilaudPons}}]
\label{prop:characterizationTOEP}
A poset~${\less} \in \IPos_n$ is in~$\TOEP_n$ if and only if
\begin{itemize}
\item $\forall \; a < b < c, \; a \less c \implies b \less c$ and $a \more c \implies a \more b$,
\item for all~$a < c$ incomparable in~$\less$, there exists ${a < b < c}$ such that~$a \less b \more c$.
\end{itemize}
\end{proposition}

\begin{proposition}[{\cite[Prop.~40 \& 41]{ChatelPilaudPons}}]
\label{prop:weakOrderTOEP}
The map~$\tree \mapsto {\less}_{\tree}$ is a lattice isomorphism from the Tamari lattice on binary trees of~$\fB_n$ to the sublattice of the weak order on~$\IPos_n$ induced by $\TOEP_n$.
\end{proposition}

\subsubsection{Intervals}

We now present a similar interpretation of the intervals of the Tamari lattice.
For~$\tree \wole \tree' \in \fB_n$, we consider the Tamari order interval~$[\tree, \tree'] \eqdef \set{\tree[S] \in \fB_n}{\tree \wole \tree[S] \wole \tree'}$, and interpret it as the \defn{Tamari order interval poset}
\(
{{\less_{[\tree, \tree']}} \eqdef \bigcap_{\tree \wole \tree[S] \wole \tree'} {\less_{\tree}} = {\less_{\tree}} \cap {\less_{\tree'}} = {\Inc{\less_{\tree'}}} \cap {\Dec{\less_{\tree}}}.}
\)
This poset~${\less_{[\tree, \tree']}}$ was introduced in~\cite{ChatelPons} with the motivation that its linear extensions are precisely the linear extensions of all binary trees in the interval~$[\tree, \tree']$.
See~\fref{fig:toip}.
\begin{figure}[ht]
    \input{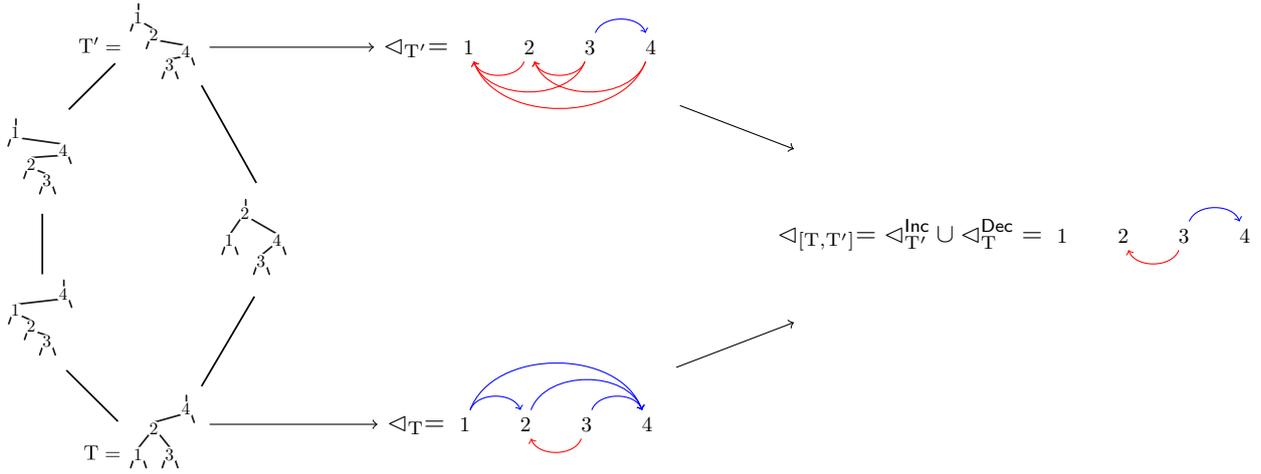}
    \caption{A Tamari Order Interval Poset ($\TOIP$).}
    \label{fig:toip}
    \vspace{-.5cm}
\end{figure}

We define
\[
\TOIP_n \eqdef \bigset{{\less_{[\tree,\tree']}}}{\tree, \tree' \in \fB_n, \tree \wole \tree'}
\qquad\text{and}\qquad
\TOIP \eqdef \bigsqcup_{n \in \N} \TOIP_n.
\]

The following statements provide a local characterization of the posets of~$\TOIP_n$ and describe the weak order induced by~$\TOIP_n$.

\begin{proposition}[{\cite[Thm.~2.8]{ChatelPons}}]
\label{prop:characterizationTOIP}
A poset~${\less} \in \IPos(n)$ is in~$\TOIP(n)$ if and only if  ${a \less c \Rightarrow b \less c}$ and ${a \more c \Rightarrow a \more b}$ for all~$1 \le a < b < c \le n$.
\end{proposition}

\enlargethispage{-.5cm}
As in Section~\ref{subsubsec:WOIP}, note that this characterization clearly splits into a condition on the increasing relations and a condition on the decreasing relations of~$\less$. This defines two super-families~$\ITOIP_n$ and~$\DTOIP_n$ of~$\TOIP_n$ with~$\TOIP_n = \ITOIP_n \cap \DTOIP_n$.

\begin{proposition}[{\cite[Prop.~40 \& 41]{ChatelPilaudPons}}]
\label{prop:weakOrderTOIP}
~
\begin{enumerate}[(i)]
\item If~$\tree[S] \wole \tree[S]'$ and~$\tree \wole \tree'$ in~$\fB_n$, then~${\less_{[\tree[S], \tree[S]']}} \wole {\less_{[\tree, \tree']}} \iff \tree[S] \wole \tree$ and~$\tree[S]' \wole \tree'$.
\item The weak order on~$\TOIP_n$ is a lattice whose meet~${\less_{[\tree[S],\tree[S]']}} \meetTOIP {\less_{[\tree,\tree']}} = {\less_{[\tree[S] \meetTO \tree, \tree[S]' \meetTO \tree']}}$ and join~${\less_{[\tree[S],\tree[S]']}} \joinTOIP {\less_{[\tree,\tree']}} = {\less_{[\tree[S] \joinTO \tree, \tree[S]' \joinTO \tree']}}$.
Moreover, the weak order on~$\TOIP_n$ is a sublattice of the weak order on~$\IPos_n$.
\end{enumerate}
\end{proposition}

\subsubsection{Faces}

Consider now a face of the associahedron, that is, a Schr\"oder tree~$\tree[S]$ (a rooted tree where each internal node has at least two children).
We label the angles between two consecutive children in inorder, meaning that each angle is labeled after the angles in its left child and before the angles in its right child.
We associate to~$\tree[S]$ the poset~$\less_{\tree[S]}$ where~$i \less_{\tree[S]} j$ if and only if the angle labeled~$i$ belongs to the left or to the right child of the angle labeled~$j$.
See \fref{fig:tofp}.
\begin{figure}[h]
    \vspace{-.2cm}
    \input{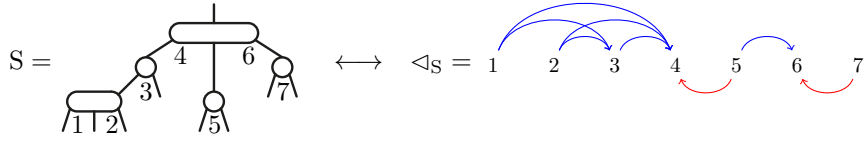}
    \vspace{-1cm}
    \caption{A Tamari Order Face Poset ($\TOFP$).}
    \label{fig:tofp}
    \vspace{-.3cm}
\end{figure}

We define
\[
\TOFP_n \eqdef \set{{\less_{\tree[S]}}}{\tree[S] \text{ Schr\"oder tree on $[n]$}}
\qquad\text{and}\qquad
\TOFP \eqdef \bigsqcup_{n \in \N} \TOFP_n.
\]

The following statements provide a local characterization of the posets of~$\TOFP_n$ and describe the weak order induced by~$\TOFP_n$.

\begin{proposition}[{\cite[Prop.~65]{ChatelPilaudPons}}]
\label{prop:characterizationTOFP}
A poset~${\less}$ is in~$\TOFP_n$ if and only if ${\less} \in \TOIP_n$  and for all~$a < c$ incomparable in~$\less$, either there exists~$a < b < c$ such that~$a \not\more b \not\less c$, or for all~$a < b < c$ we have~$a \more b \less c$.
\end{proposition}

\begin{proposition}[{\cite[Sect.~2.2.3]{ChatelPilaudPons}}]
\label{prop:weakOrderTOFP}
~
\begin{enumerate}
\item For any Schr\"oder trees~$\tree[S], \tree[S]'$, we have~${\less}_{\tree[S]} \wole {\less}_{\tree[S]'} \iff \tree[S] \wole \tree[S]'$ in the \defn{facial weak order} on the associahedron~$\Asso[n]$ studied in~\cite{PalaciosRonco, DermenjianHohlwegPilaud}.
This order is a quotient of the facial weak order on the permutahedron by the fibers of the Schr\"oder tree insertion~$\st$.
\item The weak order on~$\TOFP_n$ is a lattice but not a sublatice of the weak order on~$\IPos_n$, nor on~$\WOIP_n$, nor on~$\TOIP_n$.
\end{enumerate}
\end{proposition}

To conclude, let us recall that there is a natural \defn{Schr\"oder tree insertion} map from ordered partitions to Schr\"oder trees, similar to the binary search tree insertion from permutations to binary trees. The Schr\"oder tree~$\st(\pi)$ obtained from an ordered partition~$\pi$ is the unique Schr\"oder tree such that if~$i$ is a descendant of~$j$ (meaning that~$i$ appears in a vertex that is a descendant of the vertex containing~$j$), then $i$ appears before~$j$ in~$\pi$. See~\cite{Chapoton, ChatelPilaud, PilaudPons-permutrees} for details and \fref{fig:stinsertion} for an illustration.

\begin{figure}[h]
    \centerline{\includegraphics[scale=.9]{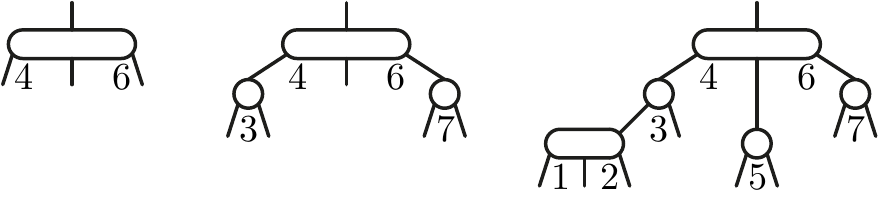}}
    \vspace{-.3cm}
    \caption{Schr\"oder tree insertion of the ordered partition~$\pi = 125 | 37 | 46$.}
    \label{fig:stinsertion}
\end{figure}

\subsection{Quotient algebras}
\label{subsec:quotientAlgebrasTO}

Contrarily to Section~\ref{subsec:quotientAlgebrasWO}, a naive construction of quotient Hopf algebras cannot directly work for~$\TOEP$, $\TOIP$ and~$\TOFP$.
Although they clearly define quotient algebras, they do not define quotient cogebras since the convolution does nos satisfy a property similar to~(ii) of Proposition~\ref{prop:algebraWOEP}: an element not in~$\TOEP$ (resp. $\TOIP$ and~$\TOFP$) can appear in the convolution product of two elements of $\TOEP$ (resp. $\TOIP$ and~$\TOFP$) as in the following example:

\input{coproductBrokenTOIP}

Indeed, see that the element on the left is neither a $\TOEP$, $\TOIP$ nor a~$\TOFP$ but it belongs to the convolution of the two integer posets on the right, which are both $\TOEP$, $\TOIP$ and~$\TOFP$.

\subsection{Subalgebras}
\label{subsec:subalgebrasTO}

We now construct Hopf algebras on~$\TOEP$, $\TOIP$ and~$\TOFP$ as subalgebras of the Hopf algebras~$\K\WOEP$, $\K\WOIP$ and~$\K\WOFP$ respectively.
For this, we need the \defn{$\TOIP$ deletion} defined in~\cite[Sect.~2.2.4]{ChatelPilaudPons} by
\[
{\TOIPd{\less}} \eqdef {\less} \ssm ( \set{(a,c)}{\exists \; a < b < c, \; b \not\less c} \cup \set{(c,a)}{\exists \; a < b < c, \; a \not\more b} ).
\]
This operation is illustrated on \fref{fig:TOIPd}.

\begin{figure}[h]
    \vspace{-.5cm}
    \input{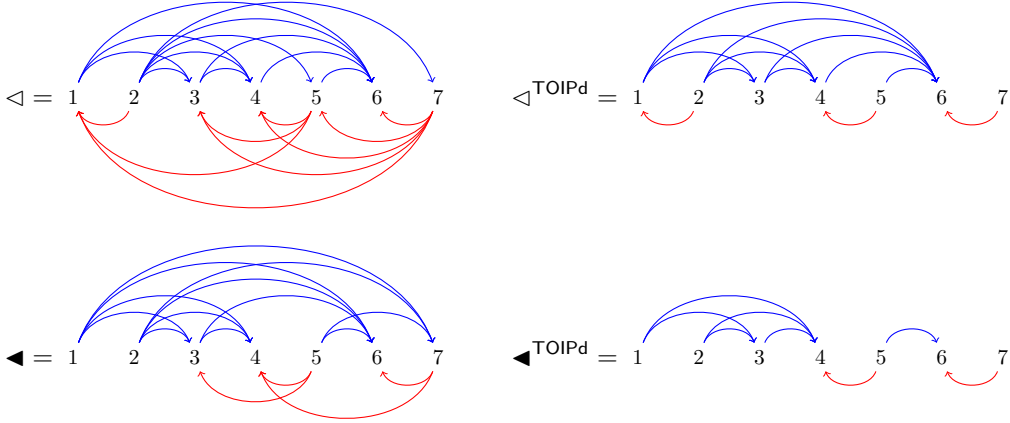}
    \vspace{-1.2cm}
    \caption{The $\TOIP$ deletion.}
    \label{fig:TOIPd}
\end{figure}

This map projects~$\WOEP_n$, $\WOIP_n$ and~$\WOFP_n$ to~$\TOEP_n$, $\TOIP_n$ and~$\TOFP_n$ respectively.
In fact, it is a simple generalization of both the binary tree insertion and the Schr\"oder tree insertion.

\begin{proposition}[{\cite[Prop.~48]{ChatelPilaudPons}}]
\label{prop:TOIPd/bst}
For any permutation~$\sigma$, any weak order interval~$\sigma \wole \sigma'$, and any ordered partition~$\pi$, we have
\[
{\TOIPd{(\less_\sigma)}} = {\less_{\bt(\sigma)}},
\qquad
{\TOIPd{(\less_{[\sigma,\sigma']})}} = {\less_{[\bt(\sigma), \; \bt(\sigma')]}}
\qquad\text{and}\qquad
{\TOIPd{(\less_\pi)}} = {\less_{\st(\pi)}},
\]
where~$\bt(\sigma)$ is the binary tree insertion of the permutation~$\sigma$ and~$\st(\pi)$ is the Schr\"oder tree insertion of the ordered partition~$\pi$.

\end{proposition}

\begin{example}
Compare Figures~\ref{fig:woep} and~\ref{fig:toep}, Figures~\ref{fig:woip} and~\ref{fig:toip}, and Figures~\ref{fig:wofp} and~\ref{fig:tofp}.
\end{example}

We now use this map~${\less} \mapsto {\TOIPd{\less}}$, mimicking the construction of the Loday--Ronco algebra on binary trees~\cite{LodayRonco, HivertNovelliThibon-algebraBinarySearchTrees} as a Hopf subalgebra of the Malvenuto--Reutenauer algebra on permutations~\cite{MalvenutoReutenauer}.
For~${\bless} \in \TOIP$, consider the element
\[
\FTOIP_\bless \eqdef \sum \FWOIP_\less
\]
where the sum runs over all ${\less} \in \WOIP$ such that~$\TOIPd{\less} = {\bless}$.
We denote by~$\K\TOIP$ the linear subspace of~$\K\IPos$ spanned by the elements~$\FTOIP_\bless$ for~${\bless} \in \TOIP$.
Similarly, we define the linear subspace~$\K\TOEP$ (resp.~$\K\TOFP$) spanned by the elements
\[
\FTOEP_\bless \eqdef \sum \FWOEP_\less
\qquad
\text{(resp. $\FTOEP_\bless \eqdef \sum \FWOEP_\less$)}
\]
for all~${\bless} \in \TOEP$ (resp.~$\TOFP$) where the sum runs over all~${\less} \in \WOEP$ (resp.~$\WOFP$) such that~${\TOIPd{\less}} = {\bless}$.

\begin{remark}
Note that the fiber of a $\TOEP$ (resp.~$\TOFP$) under the map ${\less} \mapsto {\TOIPd{\less}}$ is not in~$\WOEP$ (resp.~$\WOFP$) in general.
For example, $\less_{[132,312]} = \raisebox{.1cm}{\scalebox{0.8}{\input{relations/r3_12_32}}}$ is not in~$\WOEP$, but
\[
{\TOIPd{\less_{[132,312]}}} = {\TOIPd{(\raisebox{.1cm}{\scalebox{0.8}{\input{relations/r3_12_32}}})}} = \raisebox{.1cm}{\scalebox{0.8}{\input{relations/r3_12_32}}} = {\less_{\!\!\!\includegraphics[scale=1.5]{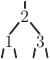}}}
\]
is in~$\TOEP$.
The element~$\FTOEP_\bless$ (resp.~$\FTOFP_\bless$) is defined as the sum over the fiber of~$\bless$ in~$\WOEP$ (resp.~in~$\WOFP$).
\end{remark}

\begin{example}
\label{exm:algebraTOEP}
Here is an example of computation of a product and a coproduct of elements~$\FTOEP_\bless$ computed in the algebra~$\K\WOEP$.
In the second line, for each element~$\FWOEP_\less$ of the sum in~$\K\WOEP$, we have bolded the subrelation~$\TOIPd{\less}$.
Observe that the result can again be expressed in the basis~$\FTOEP_\bless$ as will be proven in Proposition~\ref{prop:TOIPsubHopfAlgebra}.

\input{example_algebraTOEP}
\input{example_cogebraTOEP}
\end{example}

\begin{example}
\label{exm:algebraTOIP}
Here is an example of computation of a product and a coproduct of elements~$\FTOIP_\bless$ computed in the algebra~$\K\WOIP$.
In the second line, for each element~$\FWOIP_\less$ of the sum in~$\K\WOIP$, we have bolded the subrelation~$\TOIPd{\less}$.
Observe that the result can again be expressed in the basis~$\FTOIP_\bless$ as will be proven in Proposition~\ref{prop:TOIPsubHopfAlgebra}.

\input{example_algebraTOIP}
\input{example_cogebraTOIP}
\end{example}

\begin{example}
\label{exm:algebraTOFP}
Here is an example of computation of a product and a coproduct of elements~$\FTOFP_\bless$ computed in the algebra~$\K\WOFP$.
In the second line, for each element~$\FWOFP_\less$ of the sum in~$\K\WOFP$, we have bolded the subrelation~$\TOIPd{\less}$.
Observe that the result can again be expressed in the basis~$\FTOFP_\bless$ as will be proven in Proposition~\ref{prop:TOIPsubHopfAlgebra}.

\input{example_algebraTOFP}
\input{example_cogebraTOFP}
\end{example}

\begin{proposition}
\label{prop:TOIPsubHopfAlgebra}
The subspace~$\K\TOIP$ is stable by the product~$\product$ and the coproduct~$\coproduct$ and thus defines a Hopf subalgebra of~$(\K\WOIP, \product, \coproduct)$.
A similar statement holds for~$\TOEP$ and~$\TOFP$.
\end{proposition}

\begin{proof}
We start with the product.
Consider~${\bless} \in \TOIP_m$ and~${\bless'} \in \TOIP_n$ and define
\begin{itemize}
\item $U$ as the set of $\WOIP$'s of the form~${{\less} \sqcup \overline{{\less'}} \sqcup \rel[I] \sqcup \rel[D]}$ for any~${\less} \in \WOIP_m$ with~${\TOIPd{\less}} = {\bless}$, any~${\less'} \in \WOIP_n$ with~${\TOIPd{\less'}} = {\bless'}$, any~$\rel[I] \subseteq [m] \times \overline{[n]}$ and any~$\rel[D] \subseteq \overline{[n]} \times [m]$.
\item $V$ as the set of~$\TOIP$'s of the form~${{\bless} \sqcup {\bless'} \sqcup \rel[I] \sqcup \rel[D]}$ for any~$\rel[I] \subseteq [m] \times \overline{[n]}$ and~$\rel[D] \subseteq \overline{[n]} \times [m]$.
\end{itemize}
Note that~${\dashv} \in U$ if and only if~${\TOIPd{\dashv}} \in V$.
Since the product~$\FTOIP_\bless \product \FTOIP_{\bless'}$ contains exactly all the posets of~$U$, we conclude that a poset appears in~$\FTOIP_\bless \product \FTOIP_{\bless'}$ if and only if all posets in its $\TOIP$ deletion fiber do.
Therefore, the product~$\FTOIP_\bless \product \FTOIP_{\bless'}$ belongs to the subspace~$\K\TOIP$.

We now deal with the coproduct.
Consider~${\bless} \in \TOIP_p$ and a partition~$[p] = X \sqcup Y$. Consider the set~$U_{X,Y}$ of~$\WOIP$'s of the form~${{\less} \cup {\less'} \cup (X \times Y)}$ where~${{\less} \subseteq X^2}$ is such that~${\TOIPd{(\less_X)} = {\bless}_X}$ and~${{\less'} \subseteq Y^2}$ is such that~${\TOIPd{(\less'_Y)} = {\bless}_Y}$.
Note that either none or all~${\dashv} \in U_{X,Y}$ satisfy~${{\TOIPd{\dashv}} = {\bless}}$.
Since the coproduct~$\coproduct(\FTOIP_\bless)$ contains exactly all tensors~${\less}_X \otimes {\less'_Y}$ such that~${{\less} \cup {\less'} \cup (X \times Y) \in U_{X,Y}}$ for all partitions~$[p] = X \sqcup Y$, we conclude that a tensor~${\less}_X \otimes {\less'_Y}$ appears in~$\coproduct(\FTOIP_\bless)$ if and only if all the tensors~${\dashv}_X \otimes {\dashv_Y'}$  with~${\TOIPd{\dashv}} = {\TOIPd{\less}}$ and~${\TOIPd{\dashv'}} = {\TOIPd{\less'}}$ appear in~$\coproduct(\FTOIP_\bless)$ as well.
Therefore, the coproduct~$\coproduct(\FTOIP_\bless)$ belongs to the subspace~$\K\TOIP$.
\end{proof}

\begin{proposition}
For~${\less} \in \TOIP_m$ and~${\bless} \in \TOIP_n$, the product~$\FPos_{\less} \product \FPos_{\bless}$ is the sum of~$\FPos_{\dashv}$, where~$\dashv$ runs over the interval between $\underprod{\less}{\bless}$ and $\overprod{\less}{\bless}$ in the weak order on~$\TOIP_{m+n}$.
A similar statement holds for~$\TOEP$ and~$\TOFP$.
\end{proposition}

\begin{proof}
It is a direct consequence of Proposition~\ref{prop:characterizationShuffleProduct} and the fact that for any two Tamari order element (resp.~interval, resp.~face)  posets~${{\less}, {\bless} \in \WOIP}$, the relations~$\underprod{\less}{\bless}$ and~$\overprod{\less}{\bless}$ are both Tamari order element (resp.~interval, resp.~face) posets.
\end{proof}

To the best of our knowledge, Proposition~\ref{prop:TOIPsubHopfAlgebra} provides the first Hopf structure on intervals of the Tamari lattice.
Our next example illustrates the product and coproduct in this Hopf algebra on Tamari intervals, recasting Example~\ref{exm:algebraTOIP} in terms of Tamari intervals.

\begin{example}
For instance,
\input{example_algebra_interval}
and
\input{example_cogebra_interval}
\end{example}

In contrast, for elements and faces, the Hopf structures of Proposition~\ref{prop:TOIPsubHopfAlgebra} already appear in the literature.
We have seen already in Section~\ref{subsec:LodayRonco} that the Loday--Ronco Hopf algebra on binary trees is the Hopf subalgebra of the Malvenuto--Reutenauer Hopf algebra on permutations generated by the sums over the fibers of the binary search tree insertion~$\sigma \mapsto \bt(\sigma)$.
Similarly, F.~Chapoton defined in~\cite{Chapoton} a Hopf algebra on Schr\"oder trees obtained as a Hopf subalgebra of his Hopf algebra on ordered partitions generated by the sums over the fibers of the Schr\"oder tree insertion~$\pi \mapsto \st(\pi)$.
We refer to~\cite{Chapoton} for more details and just provide an example of product and coproduct in this Hopf algebra.

\begin{example}
\label{exm:algebraChapotonSchroderTrees}
For instance,
\input{example_algebra_Schroder}
and
\input{example_cogebra_Schroder}
\end{example}

\begin{proposition}
\begin{itemize}
\item The map~$\tree \mapsto {\less_{\tree}}$ is a Hopf algebra isomorphism from the Loday--Ronco algebra on binary trees~\cite{LodayRonco, HivertNovelliThibon-algebraBinarySearchTrees} to~$(\K\TOEP, \product, \coproduct)$.
\item The map~$\tree[S] \mapsto {\less_{\tree[S]}}$ is a Hopf algebra isomorphism from the Chapoton algebra on Schr\"oder trees~\cite{Chapoton} to~$(\K\TOFP, \product, \coproduct)$.
\end{itemize}
\end{proposition}

\begin{proof}
This immediately follow from Propositions~\ref{prop:isomorphismMalvenutoReutenauer}, \ref{prop:isomorphismChapoton} and~\ref{prop:TOIPd/bst}.
\end{proof}

\begin{example}
Compare Examples~\ref{exm:algebraLodayRonco} and~\ref{exm:algebraTOEP}, and Examples~\ref{exm:algebraChapotonSchroderTrees} and~\ref{exm:algebraTOFP}.
\end{example}

\begin{remark}
To conclude, let us mention that similar ideas can be used to uniformly construct Hopf algebra structures on permutrees, permutree intervals, and Schr\"oder permutrees as defined in~\cite{PilaudPons-permutrees}.
Following~\cite{ChatelPilaud, PilaudPons-permutrees}, one first defines some decorated versions of the Hopf algebras~$\K\WOEP$, $\K\WOIP$ and~$\K\WOFP$, where each poset on~$[n]$ appears~$4^n$ times with all possible different orientations.
One then constructs Hopf algebras on~$\K\PEP[]$, $\K\PIP[]$ and~$\K\PFP[]$ using the fibers of the surjective map~$({\less}, \orientation) \mapsto {\PIPd{\less}}$ defined in~\cite{ChatelPilaudPons}.
See~\cite{PilaudPons-permutrees} for details.
\end{remark}


\section*{Acknowledgments}

The computation and tests needed along the research were done using the open-source mathematical software \texttt{Sage}~\cite{Sage} and its combinatorics features developed by the \texttt{Sage-combinat} community~\cite{SageCombinat}.


\bibliographystyle{alpha}
\bibliography{integerRelationAlgebra}
\label{sec:biblio}

\end{document}

%% file: figures/productRelations.tex
\begin{align*}
\FRel_{\scalebox{.5}{\input{figures/relations/r2_12}}} \! \!  \product \FRel_{\scalebox{.5}{\input{figures/relations/r1}}} \! \!  =
\FRel_{\scalebox{.5}{\input{figures/relations/r3_12_13_23}}} \! \!  +
\FRel_{\scalebox{.5}{\input{figures/relations/r3_12_13}}} \! \!  +
\FRel_{\scalebox{.5}{\input{figures/relations/r3_12_13_23_31}}} \! \!  + \dots +
\FRel_{\scalebox{.5}{\input{figures/relations/r3_12_31_32}}} \! \! .
\end{align*}

%% file: figures/relations/r2_12.tex
\begin{tikzpicture}[baseline, scale=\interscale]
\node(T1) at (0,0) {1};
\node(T2) at (1,0) {2};
\draw[->, line width = 0.5, color=blue] (T1) edge [bend left=70] (T2);
\draw[->,line width = 0.5, color=white, opacity=0] (T1) edge [bend left=70] (T2);
\draw[->,line width = 0.5, color=white, opacity=0] (T2) edge [bend left=70] (T1);
\end{tikzpicture}

%% file: figures/relations/r1.tex
\begin{tikzpicture}[baseline, scale=\interscale]
\node(T1) at (0,0) {1};
\end{tikzpicture}

%% file: figures/coproductRelations.tex
\begin{align*}
\coproduct \Big(
\FRel_{\scalebox{.5}{\input{figures/relations/r3_12_13_32}}} \! \!
\Big)
=
\FRel_{\scalebox{.5}{\input{figures/relations/r3_12_13_32}}} \! \!  \otimes \FRel_\varnothing
+
\FRel_{\scalebox{.5}{\input{figures/relations/r1}}} \! \!
\otimes
\FRel_{\scalebox{.5}{\input{figures/relations/r2_21}}} \! \!
+
\FRel_{\scalebox{.5}{\input{figures/relations/r2_12}}} \! \!
\otimes
\FRel_{\scalebox{.5}{\input{figures/relations/r1}}} \! \!
+
\FRel_\varnothing \otimes  \FRel_{\scalebox{.5}{\input{figures/relations/r3_12_13_32}}}
\end{align*}

%% file: figures/relations/r2_21.tex
\begin{tikzpicture}[baseline, scale=\interscale]
\node(T1) at (0,0) {1};
\node(T2) at (1,0) {2};
\draw[->, line width = 0.5, color=red] (T2) edge [bend left=70] (T1);
\draw[->,line width = 0.5, color=white, opacity=0] (T1) edge [bend left=70] (T2);
\draw[->,line width = 0.5, color=white, opacity=0] (T2) edge [bend left=70] (T1);
\end{tikzpicture}

%% file: figures/product_coproduct.tex
For that, let~$\rel[R]$ and~$\rel[S]$ both be the unique relation of size~$1$.
On the one hand, we have:
\begin{align*}
\coproduct \big( \FRel_{\scalebox{.5}{\input{figures/relations/r1}}}\!\! \big)
\product
\coproduct \big( \FRel_{\scalebox{.5}{\input{figures/relations/r1}}}\!\! \big)
&=
\big(
\FRel_{\scalebox{.5}{\input{figures/relations/r1}}}\!\! \otimes \FRel_\varnothing +
\FRel_\varnothing \otimes \FRel_{\scalebox{.5}{\input{figures/relations/r1}}}\!\!
\big) \product
\big(
\FRel_{\scalebox{.5}{\input{figures/relations/r1}}}\!\! \otimes \FRel_\varnothing +
\FRel_\varnothing \otimes \FRel_{\scalebox{.5}{\input{figures/relations/r1}}}\!\!
\big) \\
&= \big( \FRel_{\scalebox{.5}{\input{figures/relations/r1}}}\!\!  \product \FRel_{\scalebox{.5}{\input{figures/relations/r1}}}\!\!  \big)
\otimes \FRel_\varnothing
+
2 \big( \FRel_{\scalebox{.5}{\input{figures/relations/r1}}}\!\!  \otimes \FRel_{\scalebox{.5}{\input{figures/relations/r1}}}\!\! \big)
+
\FRel_\varnothing \otimes \big( \FRel_{\scalebox{.5}{\input{figures/relations/r1}}}\!\!  \product \FRel_{\scalebox{.5}{\input{figures/relations/r1}}}\!\!  \big),
\end{align*}
with
\begin{align*}
\FRel_{\scalebox{.5}{\input{figures/relations/r1}}}\!\!  \product \FRel_{\scalebox{.5}{\input{figures/relations/r1}}}\!\! =
\FRel_{\scalebox{.5}{\input{figures/relations/r2_0}}}\!\! +
\FRel_{\scalebox{.5}{\input{figures/relations/r2_12}}}\!\! +
\FRel_{\scalebox{.5}{\input{figures/relations/r2_21}}}\!\! +
\FRel_{\scalebox{.5}{\input{figures/relations/r2_12_21}}}\!\!.
\end{align*}
On the other hand, we have:
\begin{align*}
\coproduct \big( \FRel_{\scalebox{.5}{\input{figures/relations/r1}}}\!\!  \product \FRel_{\scalebox{.5}{\input{figures/relations/r1}}}\!\! \big)
&=
\coproduct \big( \FRel_{\scalebox{.5}{\input{figures/relations/r2_0}}}\!\! + \FRel_{\scalebox{.5}{\input{figures/relations/r2_12}}}\!\! + \FRel_{\scalebox{.5}{\input{figures/relations/r2_21}}}\!\! + \FRel_{\scalebox{.5}{\input{figures/relations/r2_12_21}}}\!\! \big) \\
&=
\coproduct \big( \FRel_{\scalebox{.5}{\input{figures/relations/r2_0}}}\!\! \big) +
\coproduct \big( \FRel_{\scalebox{.5}{\input{figures/relations/r2_12}}}\!\! \big) +
\coproduct \big( \FRel_{\scalebox{.5}{\input{figures/relations/r2_21}}}\!\! \big) +
\coproduct \big( \FRel_{\scalebox{.5}{\input{figures/relations/r2_12_21}}}\!\! \big) \\
&=
\big( \FRel_{\scalebox{.5}{\input{figures/relations/r2_0}}}\!\! \otimes \FRel_\varnothing +
\FRel_\varnothing \otimes \FRel_{\scalebox{.5}{\input{figures/relations/r2_0}}}\!\!  \big) 
+ \big( \FRel_{\scalebox{.5}{\input{figures/relations/r2_12}}}\!\! \otimes \FRel_\varnothing +
\FRel_{\scalebox{.5}{\input{figures/relations/r1}}}\!\! \otimes \FRel_{\scalebox{.5}{\input{figures/relations/r1}}}\!\!
+ \FRel_\varnothing \otimes \FRel_{\scalebox{.5}{\input{figures/relations/r2_12}}}\!\!  \big) \\
& \quad + \big( \FRel_{\scalebox{.5}{\input{figures/relations/r2_21}}}\!\! \otimes \FRel_\varnothing +
\FRel_{\scalebox{.5}{\input{figures/relations/r1}}}\!\! \otimes \FRel_{\scalebox{.5}{\input{figures/relations/r1}}}\!\!
+ \FRel_\varnothing \otimes \FRel_{\scalebox{.5}{\input{figures/relations/r2_21}}}\!\!  \big)
+ \big( \FRel_{\scalebox{.5}{\input{figures/relations/r2_12_21}}}\!\! \otimes \FRel_\varnothing +
\FRel_\varnothing \otimes \FRel_{\scalebox{.5}{\input{figures/relations/r2_12_21}}}\!\!  \big) \\
&= \big( \FRel_{\scalebox{.5}{\input{figures/relations/r2_0}}}\!\! +
\FRel_{\scalebox{.5}{\input{figures/relations/r2_12}}}\!\! +
\FRel_{\scalebox{.5}{\input{figures/relations/r2_21}}}\!\! +
\FRel_{\scalebox{.5}{\input{figures/relations/r2_12_21}}}\!\! \big) \otimes \FRel_\varnothing + 2  \big( \FRel_{\scalebox{.5}{\input{figures/relations/r1}}}\!\!  \otimes \FRel_{\scalebox{.5}{\input{figures/relations/r1}}}\!\! \big) \\
& \quad + \FRel_\varnothing \otimes \big( \FRel_{\scalebox{.5}{\input{figures/relations/r2_0}}}\!\! +
\FRel_{\scalebox{.5}{\input{figures/relations/r2_12}}}\!\! +
\FRel_{\scalebox{.5}{\input{figures/relations/r2_21}}}\!\! +
\FRel_{\scalebox{.5}{\input{figures/relations/r2_12_21}}}\!\! \big).
\end{align*}
Therefore
\[
\coproduct \big( \FRel_{\scalebox{.5}{\input{figures/relations/r1}}}\!\!  \product \FRel_{\scalebox{.5}{\input{figures/relations/r1}}}\!\! \big)
=
\coproduct \big( \FRel_{\scalebox{.5}{\input{figures/relations/r1}}}\!\! \big)
\product
\coproduct \big( \FRel_{\scalebox{.5}{\input{figures/relations/r1}}}\!\! \big).
\]

%% file: figures/relations/r2_0.tex
\begin{tikzpicture}[baseline, scale=\interscale]
\node(T1) at (0,0) {1};
\node(T2) at (1,0) {2};
\draw[->,line width = 0.5, color=white, opacity=0] (T1) edge [bend left=70] (T2);
\draw[->,line width = 0.5, color=white, opacity=0] (T2) edge [bend left=70] (T1);
\end{tikzpicture}

%% file: figures/multbasesRel.tex
For instance,
\quad
$\ERel^{\scalebox{.5}{\input{figures/relations/r2_0}}} \! \! =
\FRel_{\scalebox{.5}{\input{figures/relations/r2_0}}} \! \! +
\FRel_{\scalebox{.5}{\input{figures/relations/r2_21}}}$
\quad and \quad
$\HRel^{\scalebox{.5}{\input{figures/relations/r2_0}}} \! \! =
\FRel_{\scalebox{.5}{\input{figures/relations/r2_0}}} \! \! +
\FRel_{\scalebox{.5}{\input{figures/relations/r2_12}}}$.

%% file: figures/productRelationsMultbases.tex
\begin{align*}
\ERel^{\scalebox{.5}{\input{figures/relations/r2_0}}} \! \!  \product \ERel^{\scalebox{.5}{\input{figures/relations/r1}}} \! \! & =
\big( \FRel_{\scalebox{.5}{\input{figures/relations/r2_0}}} \! \!
+ \FRel_{\scalebox{.5}{\input{figures/relations/r2_21}}} \! \! \big)
\product  \FRel_{\scalebox{.5}{\input{figures/relations/r1}}} \\
& = \FRel_{\scalebox{.5}{\input{figures/relations/r2_0}}} \! \!  \product \FRel_{\scalebox{.5}{\input{figures/relations/r1}}} \! \!
 + \FRel_{\scalebox{.5}{\input{figures/relations/r2_21}}} \! \!  \product \FRel_{\scalebox{.5}{\input{figures/relations/r1}}} \\[-.3cm]
& = \FRel_{\scalebox{.5}{\input{figures/relations/r3_13_23}}} \! \!
 + \dots
 + \FRel_{\scalebox{.5}{\input{figures/relations/r3_31_32}}} \! \!
 + \FRel_{\scalebox{.5}{\input{figures/relations/r3_13_23_21}}} \! \!
 + \dots
 + \FRel_{\scalebox{.5}{\input{figures/relations/r3_21_31_32}}} \! \!   = \ERel^{\scalebox{.5}{\input{figures/relations/r3_13_23}}} \! \! .
\end{align*}

%% file: figures/coproductRelationsMultbases.tex
\begin{align*}
\coproduct \big( \ERel^{\scalebox{.5}{\input{figures/relations/r2_12}}} \! \! \big) &=
\coproduct \big( \FRel_{\scalebox{.5}{\input{figures/relations/r2_12}}} \! \!  
+  \FRel_{\scalebox{.5}{\input{figures/relations/r2_12_21}}} \! \!  
+  \FRel_{\scalebox{.5}{\input{figures/relations/r2_0}}} \! \!
+  \FRel_{\scalebox{.5}{\input{figures/relations/r2_21}}} \! \! 
\big) \\
&=  \FRel_{\scalebox{.5}{\input{figures/relations/r2_12}}} \! \! \otimes \FRel_\varnothing 
+ \FRel_{\scalebox{.5}{\input{figures/relations/r1}}} \! \! \otimes \FRel_{\scalebox{.5}{\input{figures/relations/r1}}} \! \! 
+ \FRel_\varnothing  \otimes \FRel_{\scalebox{.5}{\input{figures/relations/r2_12}}} \! \!
+ \FRel_{\scalebox{.5}{\input{figures/relations/r2_12_21}}} \! \! \otimes \FRel_\varnothing 
+ \FRel_\varnothing  \otimes \FRel_{\scalebox{.5}{\input{figures/relations/r2_12_21}}} \! \! \\
&~~~~~+ \FRel_{\scalebox{.5}{\input{figures/relations/r2_0}}} \! \! \otimes \FRel_\varnothing 
+ \FRel_\varnothing  \otimes \FRel_{\scalebox{.5}{\input{figures/relations/r2_0}}} \! \!
+ \FRel_{\scalebox{.5}{\input{figures/relations/r2_21}}} \! \! \otimes \FRel_\varnothing 
+ \FRel_{\scalebox{.5}{\input{figures/relations/r1}}} \! \! \otimes \FRel_{\scalebox{.5}{\input{figures/relations/r1}}} \! \! 
+ \FRel_\varnothing  \otimes \FRel_{\scalebox{.5}{\input{figures/relations/r2_21}}} \! \! \\[-.2cm]
&= \ERel^{\scalebox{.5}{\input{figures/relations/r2_12}}} \! \! \otimes \ERel^\varnothing
+ 2 \left( \ERel^{\scalebox{.5}{\input{figures/relations/r1}}} \! \! \otimes \ERel^{\scalebox{.5}{\input{figures/relations/r1}}} \! \! \right)
+ \ERel^\varnothing \otimes \ERel^{\scalebox{.5}{\input{figures/relations/r2_12}}} \! \! .
\end{align*}

%% file: figures/coproductDualRelationsMultbases.tex
\[
\rbcoproduct \Big( \ERel^{\scalebox{.5}{\input{figures/relations/r3_total}}} \! \! \Big) =
\ERel^{\scalebox{.5}{\input{figures/relations/r1}}} \! \! \otimes \ERel^{\scalebox{.5}{\input{figures/relations/r2_12}}}  \! \!
+  \ERel^{\scalebox{.5}{\input{figures/relations/r2_12}}}  \! \!  \otimes \ERel^{\scalebox{.5}{\input{figures/relations/r1}}} \! \!
\qquad\text{and}\qquad
\rbcoproduct \Big( \ERel^{\scalebox{.5}{\input{figures/relations/r3_12_13_32}}} \! \! \Big) =
\ERel^{\scalebox{.5}{\input{figures/relations/r1}}} \! \! \otimes \ERel^{\scalebox{.5}{\input{figures/relations/r2_21}}} \! \! .
\]

%% file: figures/productPosets.tex
\begin{align*}
\FPos_{\scalebox{.5}{\input{figures/relations/r2_12}}} \! \! \product \FPos_{\scalebox{.5}{\input{figures/relations/r1}}} \! \!
&= \FPos_{\scalebox{.5}{\input{figures/relations/r3_12_13_23}}} \! \! +
\FPos_{\scalebox{.5}{\input{figures/relations/r3_12_13}}} \! \! +
\FPos_{\scalebox{.5}{\input{figures/relations/r3_12}}} \! \! +
\FPos_{\scalebox{.5}{\input{figures/relations/r3_12_13_32}}} \! \! +
\FPos_{\scalebox{.5}{\input{figures/relations/r3_12_32}}} \! \! +
\FPos_{\scalebox{.5}{\input{figures/relations/r3_12_31_32}}}.
\end{align*}

%% file: figures/relations/r3_12.tex
\begin{tikzpicture}[baseline, scale=\interscale]
\node(T1) at (0,0) {1};
\node(T2) at (1,0) {2};
\node(T3) at (2,0) {3};
\draw[->, line width = 0.5, color=blue] (T1) edge [bend left=70] (T2);
\draw[->,line width = 0.5, color=white, opacity=0] (T1) edge [bend left=70] (T3);
\draw[->,line width = 0.5, color=white, opacity=0] (T3) edge [bend left=70] (T1);
\end{tikzpicture}

%% file: figures/coproductPosets.tex
\begin{align*}
\coproduct \Big(
\FPos_{\scalebox{.5}{\input{figures/relations/r3_12_13_32}}}  \! \!
\Big)
=
\FPos_{\scalebox{.5}{\input{figures/relations/r3_12_13_32}}} \! \!  \otimes \FPos_\varnothing
+
\FPos_{\scalebox{.5}{\input{figures/relations/r1}}} \! \!
\otimes
\FPos_{\scalebox{.5}{\input{figures/relations/r2_21}}} \! \!
+
\FPos_{\scalebox{.5}{\input{figures/relations/r2_12}}}  \! \!
\otimes
\FPos_{\scalebox{.5}{\input{figures/relations/r1}}} \! \!
+
\FPos_\varnothing \otimes  \FPos_{\scalebox{.5}{\input{figures/relations/r3_12_13_32}}}.
\end{align*}

%% file: figures/example_maxminle.tex
\begin{tabular}{c|c|c|c|c}
$\less$ & $\in \IWOIP$ & $\in \DWOIP$ & $\maxle{\less}$ & $\minle{\less}$ \\
\hline
$\raisebox{-1cm}{\scalebox{.8}{\input{figures/relations/poset6}}}$ & yes & no & $\raisebox{-1cm}{\scalebox{.8}{\input{figures/relations/poset1432}}}$ &  \\[-.15cm]
\hline
$\raisebox{-1cm}{\scalebox{.8}{\input{figures/relations/poset5}}}$ & no & yes & & $\raisebox{-1cm}{\scalebox{.8}{\input{figures/relations/poset1342}}}$ \\[-.15cm]
\hline
$\raisebox{-1cm}{\scalebox{.8}{\input{figures/relations/woip_1423_1432}}}$ & yes & yes & $\raisebox{-1cm}{\scalebox{.8}{\input{figures/relations/poset1423}}}$ & $\raisebox{-1cm}{\scalebox{.8}{\input{figures/relations/poset1432}}}$
\end{tabular}

%% file: figures/productWOEP.tex
\begin{align*}
\FWOEPquo_{\scalebox{.5}{\input{figures/relations/r2_12}}} \! \! \product \FWOEPquo_{\scalebox{.5}{\input{figures/relations/r1}}} \! \!
&= \FWOEPquo_{\scalebox{.5}{\input{figures/relations/r3_12_13_23}}} \! \! +
\FWOEPquo_{\scalebox{.5}{\input{figures/relations/r3_12_13_32}}} \! \! +
\FWOEPquo_{\scalebox{.5}{\input{figures/relations/r3_12_31_32}}}.
\end{align*}

%% file: figures/coproductWOEP.tex
\begin{align*}
\coproduct \Big(
\FWOEPquo_{\scalebox{.5}{\input{figures/relations/r3_12_13_32}}}  \! \!
\Big)
=
\FWOEPquo_{\scalebox{.5}{\input{figures/relations/r3_12_13_32}}} \! \!
\otimes
\FWOEPquo_\varnothing
+
\FWOEPquo_{\scalebox{.5}{\input{figures/relations/r1}}}
\otimes
\FWOEPquo_{\scalebox{.5}{\input{figures/relations/r2_21}}} \! \!
+
\FWOEPquo_{\scalebox{.5}{\input{figures/relations/r2_12}}}  \! \!
\otimes
\FWOEPquo_{\scalebox{.5}{\input{figures/relations/r1}}}
+
\FWOEPquo_\varnothing
\otimes
\FWOEPquo_{\scalebox{.5}{\input{figures/relations/r3_12_13_32}}}.
\end{align*}

%% file: figures/productWOIP.tex
\begin{align*}
\FWOIPquo_{\scalebox{.5}{\input{figures/relations/r2_0}}} \! \! \product \FWOIPquo_{\scalebox{.5}{\input{figures/relations/r1}}} \! \!
&= \FWOIPquo_{\scalebox{.5}{\input{figures/relations/r3_13_23}}} \! \! +
\FWOIPquo_{\scalebox{.5}{\input{figures/relations/r3_23}}} \! \! +
\FWOIPquo_{\scalebox{.5}{\input{figures/relations/r3_0}}} \! \! +
\FWOIPquo_{\scalebox{.5}{\input{figures/relations/r3_32}}} \! \! +
\FWOIPquo_{\scalebox{.5}{\input{figures/relations/r3_31_32}}} \! \! .
\end{align*}

%% file: figures/relations/r3_23.tex
\begin{tikzpicture}[baseline, scale=\interscale]
\node(T1) at (0,0) {1};
\node(T2) at (1,0) {2};
\node(T3) at (2,0) {3};
\draw[->, line width = 0.5, color=blue] (T2) edge [bend left=70] (T3);
\draw[->,line width = 0.5, color=white, opacity=0] (T1) edge [bend left=70] (T3);
\draw[->,line width = 0.5, color=white, opacity=0] (T3) edge [bend left=70] (T1);
\end{tikzpicture}

%% file: figures/relations/r3_0.tex
\begin{tikzpicture}[baseline, scale=\interscale]
\node(T1) at (0,0) {1};
\node(T2) at (1,0) {2};
\node(T3) at (2,0) {3};
\draw[->,line width = 0.5, color=white, opacity=0] (T1) edge [bend left=70] (T3);
\draw[->,line width = 0.5, color=white, opacity=0] (T3) edge [bend left=70] (T1);
\end{tikzpicture}

%% file: figures/relations/r3_32.tex
\begin{tikzpicture}[baseline, scale=\interscale]
\node(T1) at (0,0) {1};
\node(T2) at (1,0) {2};
\node(T3) at (2,0) {3};
\draw[->, line width = 0.5, color=red] (T3) edge [bend left=70] (T2);
\draw[->,line width = 0.5, color=white, opacity=0] (T1) edge [bend left=70] (T3);
\draw[->,line width = 0.5, color=white, opacity=0] (T3) edge [bend left=70] (T1);
\end{tikzpicture}

%% file: figures/coproductWOIP.tex
\begin{align*}
\coproduct \Big(
\FWOIPquo_{\scalebox{.5}{\input{figures/relations/r3_12_32}}}  \! \!
\Big)
=
\FWOIPquo_{\scalebox{.5}{\input{figures/relations/r3_12_32}}} \! \!  \otimes \FWOIPquo_\varnothing
+
\FWOIPquo_{\scalebox{.5}{\input{figures/relations/r2_0}}}  \! \!
\otimes
\FWOIPquo_{\scalebox{.5}{\input{figures/relations/r1}}} \! \!
+
\FWOIPquo_\varnothing \otimes  \FWOIPquo_{\scalebox{.5}{\input{figures/relations/r3_12_32}}}.
\end{align*}

%% file: figures/productWOFP.tex
\begin{align*}
\FWOFPquo_{\scalebox{.5}{\input{figures/relations/r2_0}}} \! \! \product \FWOFPquo_{\scalebox{.5}{\input{figures/relations/r1}}} \! \! =
\FWOIPquo_{\scalebox{.5}{\input{figures/relations/r3_13_23}}} \! \! +
\FWOFPquo_{\scalebox{.5}{\input{figures/relations/r3_0}}} \! \! +
\FWOFPquo_{\scalebox{.5}{\input{figures/relations/r3_31_32}}} \! \! .
\end{align*}

%% file: figures/coproductWOFP.tex
\begin{align*}
\coproduct \Big(
\FWOFPquo_{\scalebox{.5}{\input{figures/relations/r3_12_32}}}  \! \!
\Big)
=
\FWOFPquo_{\scalebox{.5}{\input{figures/relations/r3_12_32}}} \! \!  \otimes \FWOFPquo_\varnothing
+
\FWOFPquo_{\scalebox{.5}{\input{figures/relations/r2_0}}}  \! \!
\otimes
\FWOFPquo_{\scalebox{.5}{\input{figures/relations/r1}}} \! \!
+
\FWOFPquo_\varnothing \otimes  \FWOFPquo_{\scalebox{.5}{\input{figures/relations/r3_12_32}}}.
\end{align*}

%% file: figures/example_iwoipid_dwoipdd.tex
\centerline{
	\begin{tabular}{c@{\quad}c@{\quad}c}
		\multirow{ 2}{*}{${\less} = \raisebox{-1.55cm}{\scalebox{0.8}{\input{relations/posetExm}}}$} &
		${\IWOIPid{\less}} = \raisebox{-1.55cm}{\scalebox{0.8}{\input{relations/posetIWOIP}}}$ &
		\multirow{ 2}{*}{${\WOIPd{\less}} = \raisebox{-1.55cm}{\scalebox{0.8}{\input{relations/posetWOIP}}}$} \\
		& ${\DWOIPdd{\less}} = \raisebox{-1.55cm}{\scalebox{0.8}{\input{relations/posetDWOIP}}}$
	\end{tabular}
}

%% file: figures/example_algebra_LR.tex
\[
\begin{array}{@{}c@{${} = {}$}c@{+}c@{+}c@{}}
\F_{\!\!\!\includegraphics[scale=1.5]{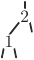}} \product \F_{\!\!\!\includegraphics[scale=1.5]{example_product2}} &
 \multicolumn{3}{l}{\F_{12} \product \big( \F_{132} + \F_{312} \big)}
\\[-.4cm]
& \begin{pmatrix} \quad \F_{12354} + \F_{12534} \\ + \; \F_{15234} + \F_{51234} \end{pmatrix}
& \begin{pmatrix} \quad \F_{13254} + \F_{31254} \\ + \; \F_{13524} + \F_{31524} \\ + \; \F_{15324} + \F_{51324} \\ + \; \F_{35124} + \F_{53124} \end{pmatrix}
& \begin{pmatrix} \quad \F_{13542} + \F_{31542} \\ + \; \F_{35142} + \F_{53142} \\ + \; \F_{15342} + \F_{51342} \\ + \; \F_{35412} + \F_{53412} \end{pmatrix}
\\[.8cm]
& \F_{\!\!\!\!\!\!\includegraphics[scale=1.5]{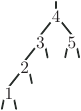}}
& \F_{\!\!\!\!\!\!\includegraphics[scale=1.5]{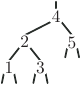}}
& \F_{\!\!\includegraphics[scale=1.5]{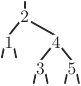}}
\end{array}
\]

%% file: figures/example_cogebra_LR.tex
\[
\begin{array}{@{}c@{${} = {}$}c@{\;+\;}c@{\;+\;}c@{\;+\;}c@{}}
\coproduct(\F_{\!\!\!\includegraphics[scale=1.5]{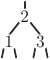}}\!\!) &
 \multicolumn{3}{l}{\coproduct \big( \F_{132} + \F_{312} \big)}
\\[-.2cm]
& (\F_{132} + \F_{312}) \otimes \F_\varnothing
& (\F_{12} + \F_{21}) \otimes \F_1
& \F_1 \otimes (\F_{12} + \F_{21})
& \F_\varnothing \otimes (\F_{132} + \F_{312})
\\[.4cm]
& \F_{\!\!\!\includegraphics[scale=1.5]{example_coproduct1}} \otimes \F_\varnothing
& (\F_{\!\!\!\includegraphics[scale=1.5]{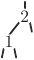}} + \F_{\!\includegraphics[scale=1.5]{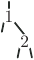}}\!\!) \otimes \F_{\!\includegraphics[scale=1.5]{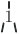}}
& \F_{\!\includegraphics[scale=1.5]{example_coproduct2}} \otimes (\F_{\!\!\!\includegraphics[scale=1.5]{example_coproduct4}} + \F_{\!\includegraphics[scale=1.5]{example_coproduct3}}\!\!)
& \F_\varnothing \otimes \F_{\!\!\!\includegraphics[scale=1.5]{example_coproduct1}} \! \! .
\end{array}
\]

%% file: figures/coproductBrokenTOIP.tex
\begin{align*}
\begin{aligned}\scalebox{.8}{\input{figures/relations/r3_12_13_32}}\end{aligned} \in \begin{aligned}\scalebox{.8}{\input{figures/relations/r1}}\end{aligned} \convolution
\begin{aligned}\scalebox{.8}{\input{figures/relations/r2_21}}\end{aligned}
\end{align*}

%% file: figures/example_algebraTOEP.tex
\medskip
\noindent
$\FTOEP_{\scalebox{.5}{\input{figures/relations/r2_12}}} \! \! \! \! \product \FTOEP_{\scalebox{.5}{\input{figures/relations/r3_12_32}}} \! \! \! \! = \FWOEP_{\scalebox{.5}{\input{figures/relations/r2_12}}}\! \! \! \! \product \big( \FWOEP_{\scalebox{.5}{\input{figures/relations/r3_12_13_32}}} + \FWOEP_{\scalebox{.5}{\input{figures/relations/r3_12_31_32}}} \big)$ \\
$\begin{array}{@{}c@{${} = {}$}c@{+}c@{+}c@{}}
\;\; & \begin{pmatrix} \quad \FWOEP_{\! \! \! \! \scalebox{.5}{\input{figures/relations/perm_12354}}}\! \! \! \! + \FWOEP_{\! \! \! \!\scalebox{.5}{\input{figures/relations/perm_12534}}} \! \! \! \! \\ \! \! \! \! + \; \FWOEP_{\! \! \! \! \scalebox{.5}{\input{figures/relations/perm_15234}}} \! \! \! \! + \FWOEP_{\! \! \! \! \scalebox{.5}{\input{figures/relations/perm_51234}}} \! \! \! \! \end{pmatrix}
& \begin{pmatrix} \quad \FWOEP_{\! \! \! \! \scalebox{.5}{\input{figures/relations/perm_13254}}} \! \! \! \! + \FWOEP_{\! \! \! \! \scalebox{.5}{\input{figures/relations/perm_31254}}} \! \! \! \! \\ \! \! \! \! + \; \FWOEP_{\! \! \! \! \scalebox{.5}{\input{figures/relations/perm_13524}}}\! \! \! \! + \FWOEP_{\! \! \! \! \scalebox{.5}{\input{figures/relations/perm_31524}}} \! \! \! \! \\ \! \! \! \!  + \; \FWOEP_{\! \! \! \! \scalebox{.5}{\input{figures/relations/perm_15324}}}\! \! \! \! + \FWOEP_{\! \! \! \! \scalebox{.5}{\input{figures/relations/perm_51324}}} \! \! \! \! \\ \! \! \! \! +  \; \FWOEP_{\! \! \! \! \scalebox{.5}{\input{figures/relations/perm_35124}}} \! \! \! \! + \FWOEP_{\! \! \! \! \scalebox{.5}{\input{figures/relations/perm_53124}}} \! \! \! \! \end{pmatrix}
& \begin{pmatrix} \quad \FWOEP_{\! \! \! \! \scalebox{.5}{\input{figures/relations/perm_13542}}} \! \! \! \! + \FWOEP_{\! \! \! \! \scalebox{.5}{\input{figures/relations/perm_31542}}} \! \! \! \! \\ \! \! \! \! + \; \FWOEP_{\! \! \! \! \scalebox{.5}{\input{figures/relations/perm_35142}}} \! \! \! \! + \FWOEP_{\! \! \! \! \scalebox{.5}{\input{figures/relations/perm_53142}}} \! \! \! \! \\ \! \! \! \! + \; \FWOEP_{\! \! \! \! \scalebox{.5}{\input{figures/relations/perm_15342}}} \! \! \! \! + \FWOEP_{\! \! \! \! \scalebox{.5}{\input{figures/relations/perm_51342}}} \! \! \! \! \\  \! \! \! \! + \; \FWOEP_{\! \! \! \! \scalebox{.5}{\input{figures/relations/perm_35412}}} \! \! \! \! + \FWOEP_{\! \! \! \! \scalebox{.5}{\input{figures/relations/perm_53412}}} \! \! \! \! \end{pmatrix}
\\[1.5cm]
& \FTOEP_{\! \! \! \! \scalebox{.5}{\input{figures/relations/r5_tree12354}}}
& \FTOEP_{\! \! \! \! \scalebox{.5}{\input{figures/relations/r5_tree13254}}}
& \FTOEP_{\! \! \! \! \scalebox{.5}{\input{figures/relations/r5_tree13542}}} \! \! .
\end{array}$

%% file: figures/example_cogebraTOEP.tex
\begin{align*}
\coproduct \Big(
\FTOEP_{\! \scalebox{.5}{\input{figures/relations/r3_12_32}}}  \! \!
\Big)
&=
\coproduct \Big(
\FWOEP_{\! \scalebox{.5}{\input{figures/relations/r3_12_13_32}}} \! \! +
\FWOEP_{\! \scalebox{.5}{\input{figures/relations/r3_12_31_32}}}  \! \!
\Big) \\
& = 
\Big( \FWOEP_{\! \scalebox{.5}{\input{figures/relations/r3_12_13_32}}} \! \! \otimes \FWOEP_{\varnothing} +
\FWOEP_{\! \scalebox{.5}{\input{figures/relations/r1}}} \! \! \otimes \FWOEP_{\! \scalebox{.5}{\input{figures/relations/r2_21}}} \! \! +
\FWOEP_{\! \scalebox{.5}{\input{figures/relations/r2_12}}} \! \! \otimes \FWOEP_{\! \scalebox{.5}{\input{figures/relations/r1}}} \! \! +
\FWOEP_{\varnothing}  \otimes \FWOEP_{\! \scalebox{.5}{\input{figures/relations/r3_12_13_32}}} \! \! 
\Big) \\
& + \,
\Big( \FWOEP_{\! \scalebox{.5}{\input{figures/relations/r3_12_31_32}}} \! \! \otimes \FWOEP_{\varnothing} +
\FWOEP_{\! \scalebox{.5}{\input{figures/relations/r1}}} \! \! \otimes \FWOEP_{\! \scalebox{.5}{\input{figures/relations/r2_12}}} \! \! +
\FWOEP_{\! \scalebox{.5}{\input{figures/relations/r2_21}}} \! \! \otimes \FWOEP_{\! \scalebox{.5}{\input{figures/relations/r1}}} \! \! +
\FWOEP_{\varnothing} \otimes \FWOEP_{\! \scalebox{.5}{\input{figures/relations/r3_12_31_32}}} \! \! 
\Big) \\
& = 
\FTOEP_{\! \scalebox{.5}{\input{figures/relations/r3_12_32}}} \! \! \otimes \FTOEP_{\varnothing} +
\FTOEP_{\! \scalebox{.5}{\input{figures/relations/r1}}} \! \! \otimes \Big( 
\FTOEP_{\! \scalebox{.5}{\input{figures/relations/r2_21}}} \! \! +
\FTOEP_{\! \scalebox{.5}{\input{figures/relations/r2_12}}} \! \! \Big) \\
& + \,
\Big( \FTOEP_{\! \scalebox{.5}{\input{figures/relations/r2_12}}} \! \! +
\FTOEP_{\! \scalebox{.5}{\input{figures/relations/r2_21}}} \! \! \Big) \otimes
\FTOEP_{\! \scalebox{.5}{\input{figures/relations/r1}}} \! \! +
\FTOEP_{\varnothing} \otimes \FTOEP_{\! \scalebox{.5}{\input{figures/relations/r3_12_32}}} \! \! .
\end{align*}

%% file: figures/example_algebraTOIP.tex
\medskip
\noindent
$\FTOIP_{\! \! \scalebox{.5}{\input{figures/relations/r2_12}}} \! \! \! \! \product \FTOIP_{\! \! \scalebox{.5}{\input{figures/relations/r3_32}}} \! \! \! \! = \FWOIP_{\! \! \scalebox{.5}{\input{figures/relations/r2_12}}} \! \! \! \! \product \big(
\FWOIP_{\! \! \scalebox{.5}{\input{figures/relations/r3_32}}} \! \! \! \! + \FWOIP_{\! \! \scalebox{.5}{\input{figures/relations/r3_31_32}}} \! \big)$ \\
$\begin{array}{l@{}l}
\;\; &=
\big(
\FWOIP_{\! \! \! \! \scalebox{.5}{\input{figures/relations/r5_minexp_toip_123_54}}} \! \! \! \! \! \! + \dots
+  \FWOIP_{\! \! \! \! \scalebox{.5}{\input{figures/relations/r5_mid_toip_123_54}}} \! \! \! \! \! \! + \dots
+  \FWOIP_{\! \! \! \! \scalebox{.5}{\input{figures/relations/r5_maxexp_toip_123_54}}} \! \! \! \! \big) +
\big( \FWOIP_{\! \! \! \! \scalebox{.5}{\input{figures/relations/r5_minexp_toip_12_54}}} \! \! \! \! \! \! + \dots
+  \FWOIP_{\! \! \! \! \scalebox{.5}{\input{figures/relations/r5_mid_toip_12_54}}} \! \! \! \! \! \! + \dots
+  \FWOIP_{\! \! \! \! \scalebox{.5}{\input{figures/relations/r5_maxexp_toip_12_54}}} \! \! \! \!
\big) \\
& + \,
\big(
\FWOIP_{\! \! \! \! \scalebox{.5}{\input{figures/relations/r5_minexp_toip_12_32_54}}} \! \! \! \! \! \! + \dots
+  \FWOIP_{\! \! \! \! \scalebox{.5}{\input{figures/relations/r5_mid_toip_12_32_54}}} \! \! \! \! \! \! + \dots
+  \FWOIP_{\! \! \! \! \scalebox{.5}{\input{figures/relations/r5_maxexp_toip_12_32_54}}} \! \! \! \! \big) +
\big( \FWOIP_{\! \! \! \! \scalebox{.5}{\input{figures/relations/r5_minexp_toip_12_32_542}}} \! \! \! \! \! \! + \dots
+  \FWOIP_{\! \! \! \! \scalebox{.5}{\input{figures/relations/r5_mid_toip_12_32_542}}} \! \! \! \! \! \! + \dots
+  \FWOIP_{\! \! \! \! \scalebox{.5}{\input{figures/relations/r5_maxexp_toip_12_32_542}}} \! \! \! \!
\big) \\
& = 
\FTOIP_{\! \! \! \! \scalebox{.5}{\input{figures/relations/r5_toip_123_54}}} \! \! \! \! \! \! + 
\FTOIP_{\! \! \! \! \scalebox{.5}{\input{figures/relations/r5_toip_12_54}}} \! \! \! \! \! \! +
\FTOIP_{\! \! \! \! \scalebox{.5}{\input{figures/relations/r5_toip_12_32_54}}} \! \! \! \! \! \! +
\FTOIP_{\! \! \! \! \scalebox{.5}{\input{figures/relations/r5_toip_12_32_542}}} \! \! .
\end{array}$

%% file: figures/example_cogebraTOIP.tex
\begin{align*}
\coproduct \Big(
\FTOIP_{\scalebox{.5}{\input{figures/relations/r3_32}}}  \! \!
\Big)
&=
\coproduct \Big(
\FWOIP_{\scalebox{.5}{\input{figures/relations/r3_32}}}  \! \! +
\FWOIP_{\scalebox{.5}{\input{figures/relations/r3_31_32}}} \Big) \\
&=
\Big( \FWOIP_{\scalebox{.5}{\input{figures/relations/r3_32}}}  \! \! +
\FWOIP_{\scalebox{.5}{\input{figures/relations/r3_31_32}}} \! \!
\Big) \otimes \FWOIP_\varnothing
+
\FWOIP_{\scalebox{.5}{\input{figures/relations/r2_0}}} \! \!
\otimes
\FWOIP_{\scalebox{.5}{\input{figures/relations/r1}}} \! \!
+
 \FTOIP_\varnothing \otimes
 \Big( \FWOIP{\scalebox{.5}{\input{figures/relations/r3_32}}}  \! \! +
\FWOIP_{\scalebox{.5}{\input{figures/relations/r3_31_32}}} \! \!
\Big) \\
&= 
\FTOIP_{\scalebox{.5}{\input{figures/relations/r3_32}}}  \! \! 
\otimes \FTOIP_\varnothing
+
\FTOIP_{\scalebox{.5}{\input{figures/relations/r2_0}}} \! \!
\otimes
\FTOIP{\scalebox{.5}{\input{figures/relations/r1}}} \! \!
+
\FTOIP_\varnothing \otimes 
\FTOIP_{\scalebox{.5}{\input{figures/relations/r3_32}}} \! \! .
\end{align*}

%% file: figures/example_algebraTOFP.tex
\begin{align*}
\FTOFP_{\! \! \scalebox{.5}{\input{figures/relations/r3_23_21}}} \! \! \! \! \product \FTOFP_{\! \! \scalebox{.5}{\input{figures/relations/r2_0}}} \! \! \! \! &
= 
\FWOFP_{\! \! \scalebox{.5}{\input{figures/relations/r3_23_21}}} \! \! \! \! \product \FWOFP_{\! \! \scalebox{.5}{\input{figures/relations/r2_0}}} \! \! \! \! \\
& = 
\FWOFP_{\! \! \scalebox{.5}{\input{figures/relations/r5_tofpexp_14_234_21}}} \! \! \! \! +
\FWOFP_{\! \! \scalebox{.5}{\input{figures/relations/r5_tofpexp_23_21}}} \! \! \! \! +
\Big(
\FWOFP_{\! \! \scalebox{.5}{\input{figures/relations/r5_tofpexp1_23_21_43_53}}} \! \! \! \! +
\FWOFP_{\! \! \scalebox{.5}{\input{figures/relations/r5_tofpexp2_23_21_43_53}}} \! \! \! \! +
\FWOFP_{\! \! \scalebox{.5}{\input{figures/relations/r5_tofpexp3_23_21_43_53}}} \! \! \! \!
\Big) \\
&= 
\FTOFP_{\! \! \scalebox{.5}{\input{figures/relations/r5_tof_14_234_21}}} \! \! \! \! +
\FTOFP_{\! \! \scalebox{.5}{\input{figures/relations/r5_tof_23_21}}} \! \! \! \! +
\FTOFP_{\! \! \scalebox{.5}{\input{figures/relations/r5_tof_23_21_43_53}}} \! \! .
\end{align*}

%% file: figures/example_cogebraTOFP.tex
\begin{align*}
\coproduct \Big(
\FTOFP_{\! \scalebox{.5}{\input{figures/relations/r3_12_32}}}  \! \!
\Big)
& =
\coproduct \Big(
\FWOFP_{\! \scalebox{.5}{\input{figures/relations/r3_12_13_32}}} \! \! +
\FWOFP_{\! \scalebox{.5}{\input{figures/relations/r3_12_32}}}  \! \! +
\FWOFP_{\! \scalebox{.5}{\input{figures/relations/r3_12_31_32}}}  \! \!
\Big) \\
& = 
\Big(
\FWOFP_{\! \scalebox{.5}{\input{figures/relations/r3_12_13_32}}} \! \! \otimes \FWOFP_{\varnothing} +
\FWOFP_{\! \scalebox{.5}{\input{figures/relations/r1}}} \! \! \otimes \FWOFP_{\! \scalebox{.5}{\input{figures/relations/r2_21}}} \! \! +
\FWOFP_{\! \scalebox{.5}{\input{figures/relations/r2_12}}} \! \! \otimes \FWOFP_{\! \scalebox{.5}{\input{figures/relations/r1}}} \! \! +
\FWOFP_{\varnothing}  \otimes \FWOFP_{\! \scalebox{.5}{\input{figures/relations/r3_12_13_32}}} \! \! 
\Big) \\
& + \,
\Big(
\FWOFP_{\! \scalebox{.5}{\input{figures/relations/r3_12_32}}} \! \! \otimes \FWOFP_{\varnothing} +
\FWOFP_{\! \scalebox{.5}{\input{figures/relations/r2_0}}} \! \! \otimes \FWOFP_{\! \scalebox{.5}{\input{figures/relations/r1}}} \! \! +
\FWOFP_{\varnothing} \otimes \FWOFP_{\! \scalebox{.5}{\input{figures/relations/r3_12_32}}} \! \! 
\Big) \\
& + \,
\Big(
\FWOFP_{\! \scalebox{.5}{\input{figures/relations/r3_12_31_32}}} \! \! \otimes \FWOFP_{\varnothing} +
\FWOFP_{\! \scalebox{.5}{\input{figures/relations/r1}}} \! \! \otimes \FWOFP_{\! \scalebox{.5}{\input{figures/relations/r2_12}}} \! \! +
\FWOFP_{\! \scalebox{.5}{\input{figures/relations/r2_21}}} \! \! \otimes \FWOFP_{\! \scalebox{.5}{\input{figures/relations/r1}}} \! \! +
\FWOFP_{\varnothing} \otimes \FWOFP_{\! \scalebox{.5}{\input{figures/relations/r3_12_31_32}}} \! \! 
\Big) \\
& = 
\FTOFP_{\! \scalebox{.5}{\input{figures/relations/r3_12_32}}} \! \! \otimes \FTOFP_{\varnothing} +
\FTOFP_{\! \scalebox{.5}{\input{figures/relations/r1}}} \! \! \otimes \FTOFP_{\! \scalebox{.5}{\input{figures/relations/r2_21}}} \! \! +
\FTOFP_{\! \scalebox{.5}{\input{figures/relations/r2_12}}} \! \! \otimes \FTOFP_{\! \scalebox{.5}{\input{figures/relations/r1}}} \! \! +
\FTOFP_{\! \scalebox{.5}{\input{figures/relations/r2_0}}} \! \! \otimes \FTOFP_{\! \scalebox{.5}{\input{figures/relations/r1}}} \! \! \\
& + \,
\FTOFP_{\! \scalebox{.5}{\input{figures/relations/r1}}} \! \! \otimes \FTOFP_{\! \scalebox{.5}{\input{figures/relations/r2_12}}} \! \! +
\FTOFP_{\! \scalebox{.5}{\input{figures/relations/r2_21}}} \! \! \otimes \FTOFP_{\! \scalebox{.5}{\input{figures/relations/r1}}} \! \! +
\FTOFP_{\varnothing} \otimes \FTOFP_{\! \scalebox{.5}{\input{figures/relations/r3_12_32}}} \! \! .
\end{align*}

%% file: figures/example_algebra_interval.tex
\begin{align*}
\F_{\left[ \begin{aligned}\includegraphics[scale=1.5]{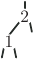}\end{aligned} \; , \; \begin{aligned}\includegraphics[scale=1.5]{example_product_interval1}\end{aligned} \right]}
\product
\F_{\left[ \begin{aligned}\includegraphics[scale=1.5]{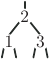}\end{aligned} \; , \; \begin{aligned}\includegraphics[scale=1.5]{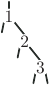}\end{aligned} \right]}
& =
\F_{\left[ \begin{aligned}\includegraphics[scale=1.5]{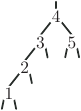}\end{aligned} \; , \; \begin{aligned}\includegraphics[scale=1.5]{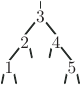}\end{aligned} \right]}
+
\F_{\left[ \begin{aligned}\includegraphics[scale=1.5]{example_product_interval4}\end{aligned} \; , \; \begin{aligned}\includegraphics[scale=1.5]{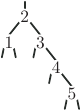}\end{aligned} \right]}
\\
& +
\F_{\left[ \begin{aligned}\includegraphics[scale=1.5]{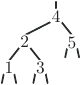}\end{aligned} \; , \; \begin{aligned}\includegraphics[scale=1.5]{example_product_interval6}\end{aligned} \right]}
+
\F_{\left[ \begin{aligned}\includegraphics[scale=1.5]{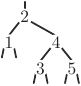}\end{aligned} \; , \; \begin{aligned}\includegraphics[scale=1.5]{example_product_interval6}\end{aligned} \right]}
\end{align*}

%% file: figures/example_cogebra_interval.tex
\[
\coproduct \Big(
\F_{\left[ \begin{aligned}\includegraphics[scale=1.5]{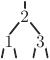}\end{aligned} \; , \; \begin{aligned}\includegraphics[scale=1.5]{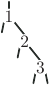}\end{aligned} \right]}
\Big )
=
\F_{\left[ \begin{aligned}\includegraphics[scale=1.5]{example_coproduct_interval1}\end{aligned} \; , \; \begin{aligned}\includegraphics[scale=1.5]{example_coproduct_interval2}\end{aligned} \right]}
\otimes
\F_{\varnothing}
+
\F_{\left[ \begin{aligned}\includegraphics[scale=1.5]{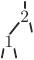}\end{aligned} \; , \; \begin{aligned}\includegraphics[scale=1.5]{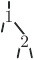}\end{aligned} \right]}
\otimes
\F_{\left[ \begin{aligned}\includegraphics[scale=1.5]{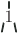}\end{aligned} \; , \; \begin{aligned}\includegraphics[scale=1.5]{example_coproduct_interval5}\end{aligned} \right]}
+
\F_{\varnothing}
\otimes
\F_{\left[ \begin{aligned}\includegraphics[scale=1.5]{example_coproduct_interval1}\end{aligned} \; , \; \begin{aligned}\includegraphics[scale=1.5]{example_coproduct_interval2}\end{aligned} \right]}.
\]

%% file: figures/example_algebra_Schroder.tex
\begin{align*}
\F_{\!\!\!\includegraphics[scale=.5]{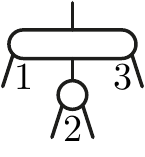}} \product \F_{\!\!\!\includegraphics[scale=.5]{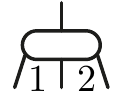}} &
= 
\F_{2|13} \product \F_{12} \\[-.5cm]
& = 
\F_{2|13|45} +
\F_{2|1345} +
\big( \F_{2|45|13} +
\F_{245|13} +
\F_{45|2|13}
\big) \\[.1cm]
&= 
\F_{\!\!\includegraphics[scale=.5]{example_product_Schroder2}} +
\F_{\!\!\!\includegraphics[scale=.5]{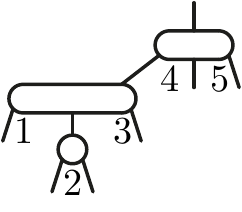}} +
\F_{\!\includegraphics[scale=.5]{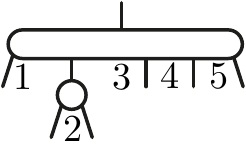}} 
\end{align*}

%% file: figures/example_cogebra_Schroder.tex
\begin{align*}
\coproduct \Big(
\F_{\!\!\!\includegraphics[scale=.5]{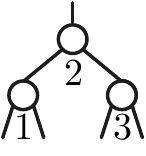}}
\Big)
& =
\coproduct \big(
\F_{1|3|2} +
\F_{13|2} +
\F_{3|1|2} +
\big) \\[-.5cm]
& = 
\big(
\F_{1|3|2} \otimes \F_{\varnothing} +
\F_{1} \otimes \F_{2|1} +
\F_{1|2} \otimes \F_{1} +
\F_{\varnothing}  \otimes \F_{1|3|2}
\big) \\
& +
\big(
\F_{13|2} \otimes \F_{\varnothing} +
\F_{12} \otimes \F_{1} +
\F_{\varnothing}  \otimes \F_{13|2}
\big) \\
& +
\big(
\F_{3|1|2} \otimes \F_{\varnothing} +
\F_{1} \otimes \F_{1|2} +
\F_{2|1} \otimes \F_{1} +
\F_{\varnothing}  \otimes \F_{3|1|2}
\big) \\
& =
\F_{\!\!\!\includegraphics[scale=.5]{example_coproduct_Schroder1}} \otimes \FTOFP_{\varnothing} +
\F_{\!\includegraphics[scale=.5]{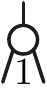}} \otimes \F_{\!\includegraphics[scale=.5]{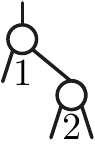}} +
\F_{\!\!\!\includegraphics[scale=.5]{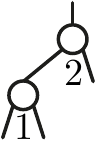}} \otimes \F_{\!\includegraphics[scale=.5]{example_coproduct_Schroder5}} +
\F_{\!\!\includegraphics[scale=.5]{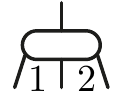}} \otimes \F_{\!\includegraphics[scale=.5]{example_coproduct_Schroder5}} \\
& +
\F_{\!\includegraphics[scale=.5]{example_coproduct_Schroder5}} \otimes \F_{\!\!\!\includegraphics[scale=.5]{example_coproduct_Schroder3}} +
\F_{\!\includegraphics[scale=.5]{example_coproduct_Schroder2}} \otimes \F_{\!\includegraphics[scale=.5]{example_coproduct_Schroder5}} +
\FTOFP_{\varnothing} \otimes \F_{\!\!\!\includegraphics[scale=.5]{example_coproduct_Schroder1}} \! \! .
\end{align*}